\documentclass{amsart}
\usepackage{array}
\newcolumntype{P}[1]{>{\raggedright\let\newline\\\arraybackslash\hspace{0pt}}m{#1}}
\usepackage{graphicx,verbatim, amsmath, amssymb, amsthm, amsfonts, amsxtra,ifthen,mathtools,caption,enumitem,hhline,bbm,capt-of,longtable}	
\usepackage[bookmarks=true, bookmarksopen=false,%
    colorlinks=true,%
    linkcolor=darkblue,%
    citecolor=darkblue,%
    filecolor=darkblue,%
    menucolor=darkblue,%
    linktoc=all,%
    urlcolor=darkblue
]{hyperref}
\usepackage{xcolor}
\definecolor{darkblue}{cmyk}{1,0.3,0,0.1}  %blue

\newcounter{globalthm}
\makeatletter
\def\thmhead@plain#1#2#3{%
  \thmname{#1}\thmnumber{\@ifnotempty{#1}{ }\@upn{#2}%
   \stepcounter{globalthm}\ \textnormal{(\hyperlink{convention}{\#\@upn{\arabic{globalthm}}})}}%
  \thmnote{ {\the\thm@notefont(#3)}}}
\let\thmhead\thmhead@plain
\makeatother

\newtheorem{proposition}{Proposition}[section]
\newtheorem{theorem}[proposition]{Theorem}
\newtheorem{corollary}[proposition]{Corollary}
\newtheorem{lemma}[proposition]{Lemma}
\newtheorem{thm}[proposition]{Theorem}
\newtheorem{conjecture}[proposition]{Conjecture}

\theoremstyle{definition}

\newtheorem{question}[proposition]{Question}

\theoremstyle{remark}
\newtheorem{remark}[proposition]{Remark}

\numberwithin{equation}{section}

\newcommand{\newword}[1]{\textbf{\emph{#1}}}

\newcommand{\integers}{\mathbb Z}

\newcommand{\reals}{\mathbb R}

\newcommand{\cut}{\operatorname{cut}}

\newcommand{\JIrr}{\operatorname{JIrr}}
\newcommand{\ji}{\operatorname{ji}}
\newcommand{\Sh}{\operatorname{Sh}}

\newcommand{\sgn}{\operatorname{sgn}}
\newcommand{\vsgn}{\mathbf{sgn}}

\newcommand{\Ram}{{\operatorname{Ram}}}

\newcommand{\cov}{\mathrm{cov}}

\newcommand{\covers}{{\,\,\,\cdot\!\!\!\! >\,\,}}
\newcommand{\covered}{{\,\,<\!\!\!\!\cdot\,\,\,}}
\newcommand{\set}[1]{{\left\lbrace #1 \right\rbrace}}
\newcommand{\pidown}{\pi_\downarrow}

\newcommand{\br}[1]{{\langle #1 \rangle}}

\newcommand{\F}{{\mathcal F}}
\newcommand{\D}{{\mathfrak D}}

\newcommand{\p}{{\mathfrak p}}

\newcommand{\join}{\vee}
\newcommand{\meet}{\wedge}

\newcommand{\ck}{\spcheck}

\newcommand{\0}{{\mathbf{0}}}

\newcommand{\Tits}{\mathrm{Tits}}
\newcommand{\Cone}{\mathrm{Cone}}

\newcommand{\Proj}{\mathrm{Proj}}
\newcommand{\relint}{\mathrm{relint}}

\DeclareMathOperator{\Span}{Span}
\DeclareMathOperator{\supp}{supp}
\DeclareMathOperator{\inv}{inv}

\newcommand{\g}{\mathbf{g}}

\newcommand{\m}{\mathfrak{m}}

\renewcommand{\k}{\mathbbm{k}}
\newcommand{\kk}{\mathbf{k}}

\renewcommand{\u}{\mathbf{u}}
\newcommand{\w}{\mathbf{w}}
\newcommand{\tB}{{\tilde{B}}}

\newcommand{\C}{\mathcal{C}}

\newcommand{\Scat}{\operatorname{Scat}}
\newcommand{\Fan}{\operatorname{Fan}}
\newcommand{\ScatFan}{\operatorname{ScatFan}}

\newcommand{\re}{\mathrm{re}}

\renewcommand{\d}{{\mathfrak d}}
\newcommand{\seg}[1]{\overline{#1}}
\newcommand{\hy}{\hat{y}}
\newcommand{\ab}{\uparrow}
\newcommand{\bel}{\downarrow}

\newcommand{\cm}[3]{(#1\Vert#2)_{#3}}

\newcommand{\cmcircarrow}{{{\circlearrowright}}}

\newcommand{\cmcirc}[3]{(#1\,\cmcircarrow\,#2)_{#3}}

\newcommand{\dd}{{\mathbf d}}
\newcommand{\aff}{\mathrm{aff}}
\newcommand{\fin}{\mathrm{fin}}

\newcommand{\DF}{{\mathcal {DF}}}
\newcommand{\DCScat}{{\operatorname{DCScat}}}
\newcommand{\adj}[2]{\operatorname{adj}_{#1}(#2)}

\newcommand{\Upper}{\operatorname{Upper}}
\newcommand{\ScatTB}{\Scat^T\!(B)}
\newcommand{\ScatFanTB}{\ScatFan^T\!(B)}

\newcommand{\eigenspace}[1]{U^{#1}}
\newcommand{\RSChar}{\Phi}
\newcommand{\RS}{\RSChar}
\newcommand{\RSre}{\RS^\re}
\newcommand{\RSpos}{\RS^+}

\newcommand{\RSfin}{\RS_\fin}

\newcommand{\SimplesChar}{\Pi}
\newcommand{\Simples}{\SimplesChar}

\newcommand{\RSTChar}{\Upsilon}
\newcommand{\RST}[1]{\RSTChar^{#1}}
\newcommand{\RSTfin}[1]{\RST{#1}_\fin}
\newcommand{\SimplesTChar}{\Xi}
\newcommand{\SimplesT}[1]{\SimplesTChar^{#1}}
\newcommand{\simpleT}{\beta}
\newcommand{\SuppT}{\operatorname{Supp}_\SimplesTChar}
\newcommand{\TravInfChar}{\Psi}

\newcommand{\proj}{\to}
\newcommand{\TravProj}[1]{\overrightarrow{\TravInfChar}^{#1}}
\newcommand{\inj}{\leftarrow}
\newcommand{\TravInj}[1]{\overleftarrow{\TravInfChar}^{#1}}
\newcommand{\AP}[1]{\RS_{#1}}
\newcommand{\APre}[1]{\AP{#1}^\re}
\newcommand{\APTChar}{\Lambda}
\newcommand{\APT}[1]{\APTChar_{#1}}      
\newcommand{\APTre}[1]{\APT{#1}^\re}

\allowdisplaybreaks

\author{Nathan Reading}
\author{Salvatore Stella}
\title{Cluster scattering diagrams of acyclic affine type}
\address[N. Reading]{Department of Mathematics, North Carolina State University, Raleigh, NC, USA}
\address[S. Stella]{Dipartimento di Ingegneria e Scienze dell'Informazione e Matematica, Università degli Studi dell'Aquila, IT}
\thanks{Nathan Reading was partially supported by the National Science Foundation under Grant Numbers DMS-1500949 and DMS-2054489 and by the Simons Foundation under award number 581608.\\ 
\indent Salvatore Stella was partially supported by ISF grant \#1144/16, the University of Haifa, the University of Leicester, and the University of Rome ``La Sapienza''.}
\begin{document}

\begin{abstract}
We give an explicit construction of the cluster scattering diagram for any acyclic exchange matrix of affine type.
We show that the corresponding cluster scattering fan coincides both with the mutation fan and with a fan constructed in the almost positive roots model.
\end{abstract}

\maketitle

\vspace{-8pt}

\setcounter{tocdepth}{2}
\tableofcontents

\section{Introduction}\label{intro}
Cluster scattering diagrams were introduced by Gross, Hacking, Keel, and Kontsevich \cite{GHKK}  (``GHKK'') to bring the powerful machinery of scattering diagrams 
to bear on the study of cluster algebras.
In the process, GHKK proved (or corrected and proved) longstanding structural conjectures about cluster algebras.
Since then, cluster scattering diagrams have become an important tool in the theory.

The key result of GHKK gives a recursive construction of the cluster scattering diagram, but no general, concrete description.
A primary output of a cluster scattering diagram is a collection of theta functions, which, in essence, is a superset of the cluster monomials.
However, the definition of a theta function, while combinatorial, requires listing all possible ``broken lines'' with a certain starting slope and endpoint.
This can be prohibitively complicated in high dimension and/or when the cluster scattering diagram has infinitely many walls.
Furthermore, computing structure constants for the multiplication of theta functions involves an even more complicated enumeration of broken lines.
For all of these reasons, there is a need for combinatorial models of cluster scattering diagrams, to make the machinery of scattering diagrams concrete in specific cases.

For cluster algebras of acyclic finite type, the construction of cluster scattering diagrams is easy, by combining existing results on cluster algebras and results of GHKK.
(See \cite[Remark~4.8]{scatcomb}.)
As an alternative, \cite[Theorem~4.3]{scatcomb} gives a tidy construction of cluster scattering diagrams of acyclic finite type using the machinery of shards \cite{hyperplane,congruence,shardint} and sortable elements \cite{sortable,typefree}.

We will say that an exchange matrix $B$ is of \newword{affine type} if it is mutation-equivalent to an acyclic exchange matrix $B'$ whose underlying Cartan matrix is of affine type in the usual sense.
As a consequence of \cite[Theorem~3.5]{Seven} and \mbox{\cite[Theorem~1.1]{FeShThTu12}}, if $B$ is of rank $\ge 3$, then it is of affine type if and only if the associated cluster algebra has linear growth but is not finite.
(For a less restrictive definition of affine type, with different motivations, see~\cite{FeLamp}.)

In this paper, we construct cluster scattering diagrams of acyclic affine type with principal coefficients.
One can easily extend this construction to other choices of coefficients (assuming that the columns of the extended exchange matrix are independent as required in the definition of scattering diagrams) by making substitutions to all scattering terms, without changing the geometry of the walls.
(See Remark~\ref{other coeffs}.)

The construction again uses the machinery of shards and sortable elements.
We draw on results of \cite{afframe} that construct the affine $\g$-vector fan as the doubled Cambrian fan and results of \cite{affdenom} that establish an almost positive roots model in affine type.
In a future paper in preparation \cite{afftheta}, we will use the results of this paper to compute theta functions and study their structure constants.

In this paper, we also show that, in the acyclic affine case, the cluster scattering fan (the fan cut out by the walls of the cluster scattering diagram \cite{scatfan}) coincides with two other fans:  a fan called the mutation fan \cite{universal} that exists for any exchange matrix $B$, and the generalized associahedron fan, constructed in \cite{affdenom} for $B$ of acyclic affine type.
Since the generalized associahedron fan is constructed as an explicit clique complex of its rays, the coincidence of these three fans provides an explicit construction of the cluster scattering fan and mutation fan in terms of rays.

The coincidence of the cluster scattering fan and mutation fan in acyclic affine type leads easily to a more general result:
The cluster scattering fan and the mutation fan coincide in the affine case, even without the assumption of acyclicity.

Mutation fans were defined in \cite{universal} in order to construct universal geometric cluster algebras.
In another future paper, we plan to use the fact that the mutation fan and the generalized associahedron fan coincide to construct universal geometric cluster algebras of affine type, proving a conjecture from~\cite{universal}.

The proofs in this paper draw heavily on three sets of tools:  the sortable elements, Cambrian lattices/fans and doubled Cambrian fans of \cite{cambrian,sortable,camb_fan,typefree,afframe}; the shards/lattice theory of \cite{hyperplane,congruence,cambrian,sort_camb,shardint}; and the almost positive roots/generalized associahedra of \cite{associahedra,FoZe03a,FoZe03,MRZ,afforb,affdenom}.
Some parts of the proofs use new results about these tools (most often generalizing something that was known in finite type).
Beyond the basic definitions, tools from the theory of scattering diagrams do not play a role in the proofs.
This is to be expected, because there are no general tools to \emph{explicitly} construct cluster scattering diagrams.
Instead, we must rely on combinatorial models.

\begin{remark}\label{other efforts}
Ours is not the only effort to understand cluster scattering diagrams using different tools.  
For example, Bridgeland \cite{Bridgeland} has a general construction of cluster scattering diagrams in terms of motivic Hall algebras. 
In addition, Hanson, Igusa, Kim, and Todorov \cite{HIKT} worked out the local structure of a cluster scattering fan of affine type near its limiting ray.
Working in the representation-theoretic language of semi-invariant pictures, they described this local structure in terms of standard wall-and-chamber structures of Nakayama algebras.
\end{remark}

\hypertarget{convention}{}

\noindent
\textbf{Numbering convention.}
The original preprint of this paper had environments numbered sequentially within each section.
The published version instead numbers environments sequentially from the beginning of the paper.
To avoid possible confusion we here keep both numbering schemes.
For example, our first main result is \textbf{Theorem~2.2} (\hyperlink{convention}{\#3}).
In the text, references use only the original numbering scheme.

\section{Main results}
We now describe our main results in more detail, leaving careful definitions until Section~\ref{def sec}.

\subsection{The cluster scattering diagram}
We give several characterizations of the cluster scattering diagram in acyclic affine type, all having to do with a cutting relation on hyperplanes orthogonal to roots, and most phrased in terms of shards.

An exchange matrix $B$ determines a Cartan matrix $A$, a root system $\Phi \subset V$, and a Coxeter group $W$ as described in Section~\ref{rs sec}. 
An acyclic $B$ also determines a Coxeter element $c$ of $W$ and, in this case, $A$ and $c$ together amount to the same information as $B$.
(Unfortunately, over the years, these equivalent ways of encoding the same data led to an ill-assorted bunch of notations; in order not to confuse readers already familiar with the topic we resist the urge to rationalize these notations here.)
When $B$ is acyclic and of affine type, $A$, $\Phi$, and $W$ are of affine type in the usual sense.

We write $\ScatTB$ for the (transposed) \newword{cluster scattering diagram} for $B$ with principal coefficients.  
(The transpose, relative to the conventions of \cite{GHKK} is necessary to keep our conventions in line with \cite{ca4}.)
This is a collection of \newword{walls} $(\d,f_\d)$, where $\d$ is a codimension-$1$ rational cone in $V^*$ and $f_\d$ is a formal power series in variables $\hy$ indexed by the simple roots.
Specifically, if $\d \subset \beta^\perp$ with $\beta$ a positive primitive vector, then $f_\d$ is a univariate formal power series evaluated on~$\hy^\beta$, the monomial whose exponents are given by the simple-root coordinates of~$\beta$.

\begin{remark}\label{other coeffs}
Here we construct $\ScatTB$ with principal coefficients.
By \mbox{\cite[Proposition~2.6]{scatfan}}, for arbitrary coefficients (chosen so that the extended exchange matrix has linearly independent columns), the scattering fan can then be obtained from the principal-coefficients scattering diagram by making a substitution in all scattering terms.
The substitution amounts to replacing the monomials $\hy_i$ that appear here with the more general monomials $\hy_i$ defined in \cite[(7.11)]{ca4}.
\end{remark}

Following constructions for finite type, in Section~\ref{shard sec} we will define a \newword{cutting} relation on hyperplanes orthogonal to roots (including the hyperplane orthogonal to the imaginary root $\delta$).
Each hyperplane is cut along its intersection with other hyperplanes into finitely many pieces, whose closures are called \newword{shards}.
The shards that intersect the Tits cone are in bijection with the join-irreducible elements of the weak order on $W$.
Given a join-irreducible element $j$, we write $\Sh(j)$ for the corresponding shard and let $f_j$ be $1+\hy^\beta$, where $\beta$ is the positive root orthogonal to $\Sh(j)$.

We now introduce a distinguished wall called the \newword{imaginary wall} $(\d_\infty,f_\infty)$.
It is supported on the hyperplane $\delta^\perp$ that forms the boundary of the Tits cone and it is defined in terms of a skew-symmetric bilinear form $\omega_c$ that depends on~$B$.
We set $\RSfin^{\omega_c+}=\set{\beta\in\RSfin:\omega_c(\beta,\delta)>0}$, where $\RSfin$ denotes the underlying finite root system, and define 
$\d_\infty=\{ x \in \partial \Tits(A) : \br{x,\beta}\le0,\, \forall\beta \in \RSfin^{\omega_c+}\}$. 
We also define~$f_\infty$ to be   
\begin{equation}\label{f inf def}
f_\infty=\begin{cases}
(1-\hy^\delta)^{-2}&\text{if }W\text{ is not of type }A_{2k}^{(2)},\text{ or}\\
(1+\hy^\delta)\cdot(1-\hy^\delta)^{-2}&\text{if }W\text{ is of type }A_{2k}^{(2)}.
\end{cases}\end{equation}

The $c$-sortable elements are certain elements of $W$ that, in finite type, play a role in Coxeter-Catalan combinatorics.
We write $\JIrr_c(W)$ for the set of join-irreducible $c$-sortable elements.

Finally, define the scattering diagram $\DCScat(A,c)$ as 
\begin{equation}\label{DCScat union}
\set{(\Sh(j),f_j):j\in\JIrr_c(W)}\cup\set{(-\Sh(j),f_j):j\in\JIrr_{c^{-1}}(W)}\cup\set{(\d_\infty,f_\infty)}.
\end{equation}
In general, the definition of scattering diagrams allows multiple copies of the same wall, but we don't allow them in $\DCScat(A,c)$.
Thus \eqref{DCScat union} should be understood in the ordinary sense of set union, rather than multiset union.
There is some overlap between the first two sets of walls, but simple reasoning based on Proposition~\ref{c beta adj} shows that the overlap consists of a finite set of walls.

Our first main result is the following theorem.

\begin{theorem}\label{DCScat}
If $B$ is an acyclic exchange matrix of affine type and $A$ and $c$ are the corresponding Cartan matrix and Coxeter element, then $\ScatTB$ is $\DCScat(A,c)$.
\end{theorem}

In \cite{affdenom}, we combinatorially defined a set $\AP{c}$ of \newword{almost positive Schur roots} and showed that it consists of the denominator vectors of cluster variables associated to $B$ together with the imaginary root $\delta$.
In Proposition~\ref{DC Phic}, we characterize the normal vectors to walls appearing in \eqref{DCScat union} as precisely the positive roots in $\AP{c}$, thus obtaining the following corollary of Theorem~\ref{DCScat}.

\begin{theorem}\label{scat Schur}
If $B$ is an acyclic exchange matrix of affine type, then the set of positive normals to walls of $\ScatTB$ is $\AP{c}\cap\RSpos$.
\end{theorem}

Inspired by Theorem~\ref{scat Schur}, we can also define the cluster scattering diagram directly in terms of $\AP{c}$, the cutting relation, and the form $\omega_c$.
Given a pair $\beta,\gamma$ of positive roots, we say $\gamma$ \newword{cuts} $\beta$ if $\beta^\perp$ is cut along its intersection with $\gamma^\perp$ in the construction of shards.
Write $\cut(\beta)$ for the set of positive roots $\gamma$ that cut $\beta$.
This is always a finite set.
Define
\begin{equation}\label{d beta}
\d_\beta=\set{x\in V^*:\br{x,\beta}=0\text{ and }\br{x,\gamma}\le0,\,\forall\gamma\in\cut(\beta)\text{ with }\omega_c(\gamma,\beta)>0}.
\end{equation}
If $\beta$ is real, define $f_\beta$ to be $1+\hy^\beta$.
The unique imaginary root in $\AP{c}$ is $\delta$, and we define $f_\delta=f_\infty$ as in \eqref{f inf def}.

\begin{theorem}\label{easy scat}
If $B$ is an acyclic exchange matrix of affine type with corresponding Coxeter element $c$ and root system $\RS$, then $\ScatTB$ consists of the walls $\set{(\d_\beta,f_\beta):\beta\in\AP{c}\cap\RSpos}$.
\end{theorem}

The construction of cluster scattering diagrams in \cite{GHKK} defines a wall with normal vector $\beta$ to be \newword{outgoing} if (in our notation) the vector $\omega_c(\,\cdot\,,\beta)$ is not in the wall.
We define a codimension-$1$ cone in $V^*$ with normal vector $\beta\in V$ to be \newword{gregarious} if $-\omega_c(\,\cdot\,,\beta)$ is in the relative interior of the cone.

\begin{theorem}\label{aff greg}
Suppose $B$ is acyclic of affine type, associated to a root system~$\RS$ and a Coxeter element $c$.
Then $\ScatTB$ can be constructed entirely of gregarious walls, with exactly one wall in each hyperplane $\beta^\perp$, where $\beta$ runs over all positive roots in $\AP{c}$, including the imaginary root $\delta$.
\end{theorem}

Indeed, we can add detail to Theorem~\ref{aff greg} by describing walls in terms of shards and $\omega_c$.

\begin{theorem}\label{aff greg shard}
Suppose $B$ is acyclic of affine type, associated to a root system~$\RS$ and a Coxeter element $c$.
Then the walls of $\ScatTB$ are
\begin{itemize}
\item the unique shard in $\beta^\perp$ containing $-\omega_c(\,\cdot\,,\beta)$ for each real root $\beta\in\APre{c}$ and
\item the union of all shards in $\delta^\perp$ containing $-\omega_c(\,\cdot\,,\delta)$,
\end{itemize}
with $f_\beta$ and $f_\delta$ as in Theorem~\ref{easy scat}.
\end{theorem}

In general, the cluster scattering diagram is defined up to a notion of equivalence.
We emphasize, however, that the descriptions of $\ScatTB$ in Theorems~\ref{DCScat}, \ref{easy scat}, and~\ref{aff greg shard} are identical, not merely equivalent.

The arguments of this paper simplify to recover a finite-type result \cite[Proposition~4.10]{scatcomb}, and in fact, augment the finite-type result by giving explicit inequalities for the walls.
See Theorem~\ref{CScat}.

The finite-type analogs of Theorems~\ref{aff greg} and~\ref{aff greg shard} are \cite[Proposition~4.12]{scatcomb} and \cite[Corollary~4.13]{scatcomb}.
Preliminary results in ongoing work with Greg Muller and Shira Viel suggest that similar results hold for cluster scattering diagrams associated to marked surfaces.
Similar results also hold trivially in rank $2$.
Since all of these examples are special in that they involve mutation-finite exchange matrices, it seems premature to make a general conjecture.
However, it seems appropriate to at least pose some questions.
(Rather than $\omega_c$, we use the analogous form $\omega$ that can be defined directly from $B$, because we are not restricting the questions to acyclic $B$.)

\begin{question}
For a general exchange matrix $B$, can the (transposed) cluster scattering diagram for $B$ be constructed with at most one wall in each hyperplane, and with that wall being gregarious?
\end{question}

\begin{question}  
For a general exchange matrix $B$, are all of the walls orthogonal to (real or imaginary) roots?
For a (real or imaginary) root $\beta$, is the wall orthogonal to a root $\beta$ the union of all shards in $\beta^\perp$ that contain $-\omega(\,\cdot\,,\beta)$?
\end{question}

\subsection{Three fans}\label{fans results}
Any consistent scattering diagram defines a complete (and often infinite) fan \cite[Theorem~3.1]{scatfan}.
We write $\ScatFanTB$ for the scattering fan associated to $\ScatTB$.

The mutation fan $\F_B$ associated to an exchange matrix $B$ encodes the piecewise-linear geometry of matrix mutation and is central to the notion of universal geometric cluster algebras~\cite{universal}.

In \cite{affdenom}, in acyclic finite or affine type, we defined a complete fan $\Fan_c(\RS)$ whose rays are spanned by the roots $\AP{c}$, and whose cones are defined by a certain compatibility relation on $\AP{c}$.
This fan contains the fan of denominator vectors of cluster variables as a subfan.
There is a piecewise-linear homeomorphism $\nu_c$ that takes denominator vectors to $\g$-vectors.
This map is linear on every cone of $\Fan_c(\RS)$, so $\nu_c(\Fan_c(\RS))$ is a fan that contains the $\g$-vector fan as a subfan, but also subdivides the space outside the $\g$-vector fan.

The following is our main theorem about these three fans.

\begin{theorem}\label{fans thm}
Suppose that $B$ is an acyclic exchange matrix of affine type, let $\RS$ be the associated root system, and let $c$ be the associated Coxeter element.  
Then $\ScatFanTB$, $\F_{B^T}$, and $\nu_c(\Fan_c(\RS))$ all coincide.
\end{theorem}

We will use Theorem~\ref{fans thm} to prove the following more general theorem about affine scattering fans and mutation fans without the assumption of acyclicity.

\begin{theorem}\label{gen fans thm}
If $B$ is an exchange matrix of affine type, then the scattering fan $\ScatFanTB$ and the mutation fan $\F_{B^T}$ coincide.
\end{theorem}

The scattering fan and the mutation fan do not always coincide, but there is a refinement relation in general, as described in the following theorem \cite[Theorem~4.10]{scatfan}.
\begin{theorem}\label{scat ref mut}
$\ScatFanTB$ refines $\F_{B^T}$ for any exchange matrix~$B$.
\end{theorem}
The following conjecture is \cite[Conjecture~4.11]{scatfan}.
\begin{conjecture}\label{scat eq mut}
For any $B$, the scattering fan $\ScatFanTB$ coincides with the mutation fan $\F_{B^T}$ if and only if either 
\begin{enumerate}[label=\rm(\roman*), ref=\roman*]
\item
$B$ is $2\times2$ and of finite or affine type or 
\item
$B$ is $n\times n$ for $n\neq2$ and of finite \emph{mutation} type.
\end{enumerate}
\end{conjecture}

The case of Conjecture~\ref{scat eq mut} where $n<2$ is trivial and the case where $n=2$ is proved by comparing \cite[Section~3]{greedytheta} with \cite[Section~9]{universal}.
Also, $B$ is of infinite mutation type if and only if there exists $B'$ mutation equivalent to $B$ and indices~$i,j$ such that the restriction of $B'$ to~$i,j$ is not of finite or affine type \cite[Theorem~2.8]{FeShTu12a}.
Thus, one expects the $n=2$ results to imply that $\ScatFanTB$ strictly refines $\F_{B^T}$ whenever $B$ is of infinite mutation type.
When $B$ is of finite type, and thus of finite mutation type, $\ScatFanTB$ and $\F_{B^T}$ both coincide with the $\g$-vector fan by \cite[Lemma~2.10]{GHKK}, \cite[Corollary~5.9]{GHKK}, and \cite[Section~10]{universal}.
If $B$ is of affine type, it is again of finite mutation type \cite[Theorem~3.5]{Seven}, so Theorem~\ref{gen fans thm} provides more evidence for Conjecture~\ref{scat eq mut}.

\section{Definitions and background}\label{def sec}
In this section, we give the detailed definitions necessary to understand statements of the main results.
(However, we postpone the definition of the mutation fan until Section~\ref{aff mut fan sec}, because it plays no role in the paper until then.)

\subsection{Root systems and Coxeter groups}\label{rs sec}
We assume the basic background on root systems and Coxeter groups.
Standard references include \cite{BjBr,Bourbaki,Humphreys}.
Much of the background can be found, with the notational conventions matching this paper, in \cite{afforb,affdenom} and, with slightly different conventions, in \cite{framework,afframe}.

The initial data for a cluster scattering diagram or cluster algebra with principal coefficients is an $n\times n$ \newword{exchange matrix}: a skew-symmetrizable integer matrix $B=[b_{ij}]$.
We make a particular choice of the skew-symmetrizing constants (the $d_i$ such that $d_ib_{ij}=-d_jb_{ji}$ for all $i,j\in\set{1,\ldots,n}$).
Specifically, we choose the $d_i$ so that $d_i^{-1}$ is an integer for each $i$ and $\gcd(d_i^{-1}:i\in\set{1,\ldots,n})=1$.

The exchange matrix $B$ determines a Cartan matrix $A=[a_{ij}]$ by $a_{ii}=2$ for $i\in\set{1,\ldots n}$ and $a_{ij}=-|b_{ij}|$ for $i,j\in\set{1,\ldots,n}$ with $i\neq j$.
The matrix $A$ is symmetrizable, specifically with $d_ia_{ij}=d_ja_{ji}$ for all $i,j\in\set{1,\ldots,n}$.
We make what is essentially the usual construction of roots, weights, co-roots, etc.\ associated to $A$.
(We depart from the standard conventions in Lie theory in that we put roots and co-roots in the same space and put weights and co-weights in the dual space.
This departure makes sense in our context as explained in \cite[Remark~2.1]{scatcomb} or \cite[Remark~2.1]{affdenom}.)
Let $V$ be an $n$-dimensional real vector space with a basis $\Simples=\set{\alpha_1,\ldots,\alpha_n}$.
The $\alpha_i$ are the \newword{simple roots}.
The \newword{simple co-roots} are $\alpha_i\ck=d_i^{-1}\alpha_i$.
Because of how we chose the constants $d_i$, the \newword{co-root lattice} $Q\ck=\Span_\integers(\alpha_1\ck,\ldots,\alpha_n\ck)$ is a finite-index sublattice of the \newword{root lattice} $Q=\Span_\integers(\alpha_1,\ldots,\alpha_n)$.
The Cartan matrix defines a symmetric bilinear form $K$ on $V$ by $K(\alpha\ck_i, \alpha_j)=a_{ij}$. 

Let $V^*$ be the dual space of $V$, with a basis $\rho_1,\ldots,\rho_n$ defined by $\br{\rho_i,\alpha_j\ck}=\delta_{ij}$, where $\br{\,\cdot\,,\cdot\,}:V^*\times V\to\reals$ is the usual pairing and $\delta_{ij}$ is the Kronecker delta function.
The $\rho_i$ are the \newword{fundamental weights}.
The \newword{fundamental co-weights} are the basis $\rho\ck_1,\ldots,\rho\ck_n$ for $V^*$ defined by $\br{\rho\ck_i,\alpha_j}=\delta_{ij}$.

We write $s_i$ for the linear map on $V$ defined by $s_i(v)=v-K(\alpha\ck_i, v)\alpha_i$.   The set $S=\set{s_i:i=1,\ldots,n}$ of \newword{simple reflections} generates the associated \newword{Coxeter group} $W$.
We will sometimes use $S$ as an alternative indexing set (rather than $\set{1,\ldots,n}$), writing $\alpha_s$ for $\alpha_i$ and $\rho_s$ for $\rho_i$ when $s=s_i$.

The set $\RSre$ of \newword{real roots} associated to $A$ is the set $\set{w\alpha_i:w\in W, i=1,\ldots,n}$.
There are also \newword{imaginary roots} that arise in Lie theory.  
We will not need the general definition of imaginary roots.
We will only need imaginary roots associated to a Cartan matrix of affine type, and we describe these specifically in Section~\ref{affdenom sec}.
We write $\RS$ for the \newword{root system} associated to $A$: the set of (real and imaginary) roots.
The notation $\RSpos$ stands for the set $\RS\cap\Span_{\reals_{\ge0}}(\alpha_1,\ldots,\alpha_n)$ of positive roots.

The \newword{inversion set} of $w\in W$ is the set $\inv(w)=\set{\beta\in\RSpos:w^{-1}(\beta)\in-\RSpos}$.
We use the symbol $\le$ for the \newword{weak order} on $W$, defined by $v\le w$ if and only if $\inv(v)\subseteq\inv(w)$, and the symbol $\covered$ for cover relations in the weak order.
These cover relations are $ws\covered w$ for $w\in W$ and $s\in S$ such that $\ell(ws)<\ell(w)$, where $\ell$ is the usual length function with respect to $S$.
The weak order is a meet-semilattice, and is a lattice if and only if $W$ is finite.
An element $j\in W$ is \newword{join-irreducible} if and only if it covers a unique element.

We write $T$ for the set $\set{w s w^{-1}:w\in W, s\in S}$ of \newword{reflections} in $W$.
We write $\beta\mapsto t_\beta$ for the bijection between real positive roots and $T$, where $t_\beta$ is the map $v\mapsto v-K(\beta\ck, v) \beta$ on $V$.
Given $t\in T$ we write $\beta_t$ for the corresponding positive real root.

The \newword{cover reflections} of $w\in W$ are the reflections $t$ such that $tw\covered w$.
Thus there is exactly one cover reflection $t=wsw^{-1}$ for each cover relation of the form $ws\covered w$.
The root $\beta_t$ is the unique element of $\inv(w)\setminus\inv(ws)$.
We write $\cov(w)$ for the set of cover reflections of $w$.

If $I\subseteq S$, then the \newword{parabolic subgroup} $W_I$ is the subgroup of $W$ generated by~$I$.
For each $w\in W$ and $I\subseteq S$, there is a unique element $w_I\in W_I$ such that $\inv(w_I)=\inv(w)\cap\RS_I$.
In particular, we will need the case where $I$ is obtained by removing one element from $S$, so for each $s\in S$, we define $\br{s}$ to mean $S\setminus\set{s}$.

\subsection{Coxeter elements and sortable elements}\label{cox sort sec}
The sign information lost in passing from $B$ to $A$ is precisely an orientation of the Dynkin diagram of $A$.
If this orientation is acyclic, then we say that $B$ is \newword{acyclic}.
In that case, the lost sign information is equivalent to a choice of \newword{Coxeter element} $c$ of $W$, namely the product of the simple reflections $S$ ordered so that $s_i$ precedes $s_j$ if $b_{ij}>0$.
We will often assume that the simple roots/reflections have been numbered such that this Coxeter element is $s_1\cdots s_n$, but there may be additional reduced words for this same Coxeter element.
Given a Coxeter element $c$ and given $s\in S$, we say $s$ is \newword{initial} in $c$ if there exists a reduced word for $c$ having $s$ as its first letter.
Similarly, $s$ is \newword{final} in $c$ if there is a reduced word for $c$ having $s$ as its last letter.
If $s$ is initial or final in $c$, then $scs$ is also a Coxeter element.
If $s$ is initial in $c$, then $sc$ is a Coxeter element of $W_\br{s}$, and if $s$ is final in $c$, then $cs$ is a Coxeter element of $W_\br{s}$.

Closely related to the Coxeter element $c$ are a skew-symmetric bilinear form $\omega_c$ and another form $E_c$, defined by 
\[
\omega_c(\alpha_i\ck,\alpha_j)=\begin{cases}
a_{ij}&\mbox{if }i>j,\\
0&\mbox{if }i=j,\mbox{ or}\\
-a_{ij}&\mbox{if }i<j,
\end{cases}
\qquad\text{and}\quad
E_c(\alpha\ck_i,\alpha_j)=\begin{cases}
a_{ij}&\text{if } i>j,\\
1&\text{if }i=j,\text{ or}\\
0&\text{if } i<j.
\end{cases}
\]
Note that $K(\alpha, \beta)=E_c(\alpha, \beta)+E_{c}(\beta, \alpha)$ and $\omega_c(\alpha,\beta)=E_c(\alpha,\beta)-E_c(\beta,\alpha)$ for all $\alpha,\beta\in V$.

We follow \cite{affdenom} in defining the piecewise linear homeomorphism $\nu_c:V\to V^*$.
(Cf. \cite[Section~5.3]{framework}.)
For a vector $\beta\in V$, write $I=\set{i:\br{\rho_i\ck,\beta}<0}$ and define ${\beta_+=\sum_{i\not\in I}\br{\rho_i\ck,\beta}\alpha_i}$.
Then
\[\nu_c(\beta)=-\sum_{i\in I}\br{\rho_i\ck,\beta}\rho_i-\sum_{i\not\in I}E_c(\alpha_i\ck,\beta_+)\rho_i.\]
In particular, if $\beta$ is in the nonnegative linear span of $\Simples$, then 
\[\nu_c(\beta)=-\sum_{i=1}^nE_c(\alpha_i\ck,\beta)\rho_i=-E_c(\,\cdot\,,\beta).\]

We now give a definition of \newword{$c$-sortable elements}, together with a distinguished reduced word for each $c$-sortable element, called its \newword{$c$-sorting word}.
The $c$-sortable elements are usually defined (e.g.\ in \cite{sortable,typefree}) by characterizing their $c$-sorting words, but here we are content to define both concepts by a recursion.
As a base case for the recursion, the identity element is $c$-sortable (and has empty $c$-sorting word) for any $c$ in any $W$.
Let $v\in W$ and suppose $s$ is initial in $c$.
If $s\le v$ (or equivalently if $\ell(sv)<\ell(v)$), then $v$ is $c$-sortable if and only if $sv$ is $scs$-sortable.
In this case, the $c$-sorting word for $v$ is $s$ followed by the $scs$-sorting word for $sv$.
If $s\not\le v$, then $v$ is $c$ sortable if and only if $v\in W_{\br{s}}$ and is $sc$-sortable.
In this case, the $c$-sorting word for $v$ equals the $sc$-sorting word for $v$.

If we write a reduced word $s_1\cdots s_n$ for $c$, then the $c$-sorting word for a $c$-sortable element $v$ consists of $k$ copies of $s_1\cdots s_n$ with $k\ge0$ followed by the $c$-sorting word for a $c$-sortable element $v'$ contained in a proper standard parabolic subgroup.

\subsection{Transposed cluster scattering diagrams}\label{trans scat sec}
We follow \cite[Section~2]{scatcomb} in constructing principal-coefficients cluster scattering diagrams in the context of root systems.
In particular we leave out some extra dimensions that occur in scattering diagrams in~\cite{GHKK}, and we take a global transpose relative to \cite{GHKK}.
(The purpose of the transpose is to match the conventions of \cite{ca4}.)

We write $P=\Span_\integers(\rho_1,\ldots,\rho_n)$ for the \newword{weight lattice} in $V^*$ and denote by $Q^+=\set{\beta\in Q\setminus\set{0} : \br{\rho_i\ck,\beta}\geq0, \, i=1,\ldots,n}$ the \newword{positive part of the root lattice}.
Let $\omega:V\times V\to\reals$ be the bilinear form given by $\omega(\alpha_i\ck,\alpha_j)=b_{ij}$.
When $B$ is acyclic, $\omega$ coincides with $\omega_c$.

We take $x_1,\ldots,x_n$ and $y_1,\ldots,y_n$ to be indeterminates and for $i$ from $1$ to $n$, we define $\hy_i=y_ix_1^{b_{1i}}\cdots x_n^{b_{ni}}$.
Given $\lambda=\sum_{i=1}^na_i\rho_i\in P$ and $\beta=\sum_{i=1}^nc_i\alpha_i\in Q^+$, define $x^\lambda \hy^\beta$ to be $x_1^{a_1}\cdots x_n^{a_n}\hy_1^{c_1}\cdots\hy_n^{c_n}$. 
Taking $\k$ to be a field of characteristic zero, we work in the ring $\k[x_1^{\pm1},\ldots,x_n^{\pm1}][[\hy_1,\ldots,\hy_n]]$ of formal power series in the $\hy_i$ with coefficients Laurent polynomials in the $x_i$.
Let $\m$ be the ideal consisting of series with constant term zero. 
Let $\k[[\hy]]$ be the subring consisting of formal power series in the $\hy_i$ with coefficients in $\k$.

A \newword{wall} is a pair $(\d,f_\d)$, where $\d$ is a codimension-$1$ subset of $V^*$ and $f_\d$ is in $\k[[\hy]]$, subject to these requirements:
\begin{enumerate}[label=\rm(\roman*), ref=\roman*]
\item $\d$ is contained in $\beta^\perp$ for some $\beta\in Q^+$ and is defined by inequalities of the form $\br{\,\cdot\,,\phi}\le0$ for $\phi\in Q$.
Importantly, we assume that $\beta$ has been chosen to be \emph{primitive} in $Q$, meaning that if $a\beta\in Q$ for some positive rational $a$, then $a$ is an integer.
\item $f_\d$ is a univariate power series evaluated at $\hy^\beta$ for this primitive $\beta$.  
\end{enumerate}
When we wish to identify the primitive vector $\beta$ associated to a wall, we will write $f_\d(\hy^\beta)$ for $f_\d$.
Two walls are \newword{parallel} if they have the same $\beta$. 
We will refer to $f_\d$ as the \newword{scattering term} on the wall $(\d,f_\d)$.

A \newword{scattering diagram} is a collection $\D$ of walls that may be infinite but must satisfy the following finiteness condition:
For all $k\ge 1$, the set of walls $(\d,f_\d)\in\D$ with $f_\d\not\equiv 1$ modulo $\m^{k+1}$ is finite.
The point of this finiteness condition is that for any given $k$, we can make computations in the scattering diagram on monomials of total degree $\le k$ by considering only finitely many walls, and the full computations can be made as limits of formal power series.
Note that since each scattering term $f_\d(\hy^\beta)$ is a formal power series in $\hy^\beta$, a sufficient condition implying the finiteness condition is that there be only finitely many walls in each hyperplane $\beta^\perp$.

A \newword{generic} path for a scattering diagram $\D$ is a piecewise differentiable path $\gamma:[0,1]\to V^*$ that does not pass through the relative boundary of any wall or the intersection of any two non-parallel walls, has endpoints not contained in any wall, and only intersects walls by crossing them transversely.
Given a generic path~$\gamma$ and a wall $(\d,f_\d)$ of $\D$ with $\gamma(t)\in\d$ for some $t\in(0,1)$, the \newword{wall-crossing automorphism} $\p_{\gamma,\d}$ is given by 
\begin{align}
\label{theta def x}
\p_{\gamma,\d}(x^\lambda)&=x^\lambda f_\d^{\br{\lambda,\pm\beta\ck}},\\
\label{theta def hat y}
\p_{\gamma,\d}(\hy^\phi)&=\hy^\phi f_\d^{\omega(\pm\beta\ck\!,\,\phi)},
\end{align}
where $\beta\ck$ is the normal vector to $\d$ that is contained in $Q^+$ and is primitive in $Q\ck$, taking $+\beta\ck$ if $\gamma$ crosses $\d$ \emph{against} the direction of $\beta\ck$ (as $t$ increases) or $-\beta\ck$ if $\gamma$ crosses $\d$ \emph{in} the direction of $\beta\ck$.

The \newword{path-ordered product} $\p_{\gamma,\D}$ for a generic path $\gamma$ is essentially the composition of all of the wall-crossing automorphisms for walls crossed by $\gamma$.
This composition makes sense even when the path crosses infinitely many walls; indeed the finiteness condition on $\D$ allows us to compute the composition modulo $\m^{k+1}$ as a finite composition for each $k\ge1$, and $\p_{\gamma,\D}$ is the limit (in the sense of formal power series) of these finite compositions as $k\to\infty$.

A scattering diagram $\D$ is \newword{consistent} if $\p_{\gamma,\D}$ depends only on the endpoints $\gamma(0)$ and $\gamma(1)$.
Two scattering diagrams $\D$ and $\D'$ are \newword{equivalent} if and only if $\p_{\gamma,\D}=\p_{\gamma,\D'}$ whenever $\gamma$ is generic for both $\D$ and $\D'$.

A wall $(\d,f_\d)$ with associated primitive normal $\beta$ is \newword{incoming} if the vector $\omega(\,\cdot\,,\beta)$ is in $\d$.
Otherwise, $(\d,f_\d)$ is \newword{outgoing}.
The crucial result of GHKK is \cite[Theorem~1.12]{GHKK}, which says that a consistent scattering diagram can be obtained by starting with walls $\set{(\alpha_i^\perp,1+\hy_i):i=1,\ldots,n}$ and appending \emph{outgoing} walls, and that such a scattering diagram is unique up to equivalence.
This is the \newword{transposed cluster scattering diagram with principal coefficients}, which we denote by $\ScatTB$.

A subset $C$ of $V^*$ is a \newword{closed convex cone} if it is closed under addition and closed under nonnegative scaling and also closed in the topological sense.
A \newword{face} of $C$ is a cone $F$ that is a subset of $C$ such that if $x,y\in C$ and $F$ intersects the line segment $\seg{xy}$ at a point other than $x$ or $y$, then the whole segment $\seg{xy}$ is in $F$.
A \newword{facet} of $C$ is a face of $C$ with dimension one less than the dimension of $C$.
A collection $\F$ of closed convex cones is a \newword{fan} if it satisfies the following two conditions: if $C\in\F$ then every face of $C$ is in $\F$; and if $C,D\in\F$ then $C\cap D$ is a face of $C$ and a face of $D$.
The notation $|\F|$ stands for the union of the cones in $\F$.
A fan $\F$ in $V^*$ is called \newword{complete} if $|\F|$ is all of $V^*$.

We now define a complete fan associated to a consistent scattering diagram.
(See~\cite{scatfan} for details.)
Essentially, this is the fan ``cut out by the walls'' of the scattering diagram, but some care must be taken to make that notion precise.

Suppose $\D$ is a scattering diagram and $\beta\in Q^+$.
The \newword{rampart} of $\D$ associated to $\beta$ is the union of the walls of $\D$ that are contained in $\beta^\perp$.  
(In acyclic affine type, we will construct the transposed cluster scattering diagram with at most one wall in each hyperplane.
As a consequence, each rampart is a single wall, so for the acyclic affine case, it is safe to replace ``rampart'' with ``wall'' throughout this definition.)
Given a point $p\in V^*$, let $\Ram_\D(p)$ stand for the set of ramparts of $\D$ containing $p$.

The \newword{support} $\supp(\D)$ of a scattering diagram $\D$ is the union of the walls of~$\D$.
Every consistent scattering diagram is equivalent to a scattering diagram $\D$ with \newword{(uniquely) minimal support}, meaning that $\supp(\D)\subseteq\supp(\D')$ for all $\D'$ equivalent to $\D$.

Now suppose $\D$ is consistent and has minimal support.
We define an equivalence relation on $V^*$ by declaring that $p$ is \newword{$\D$-equivalent} to $q$ if and only if there exists a path $\gamma$ with endpoints $p$ and $q$ on which $\Ram_\D(\,\cdot\,)$ is constant.
A \newword{$\D$-class} is a $\D$-equivalence class.
The closure of a $\D$-class is called a \newword{$\D$-cone}.  
Indeed, every $\D$-cone $C$ is the closure of a unique $\D$-class, and this class contains the \newword{relative interior} $\relint(C)$ of $C$.
Each $\D$-cone is a closed convex cone, and the collection $\Fan(\D)$ of all $\D$-cones and their faces is a complete fan in $V^*$.
We write $\ScatFanTB$ for the \newword{transposed scattering fan} $\Fan(\ScatTB)$.

\subsection{Shards}\label{shard sec}
We now discuss a notion, for $\beta\in\RS$, of cutting the hyperplane $\beta^\perp$ into shards, generalizing some constructions from finite Coxeter groups and hyperplane arrangements.
More specifically, we define shards and prove that the shards that intersect the Tits cone correspond to join-irreducible elements.

A \newword{rank-$2$ subsystem} $\RS'$ of $\RS$ is a rank-$2$ root system that is the intersection of $\RS$ with a plane.
The \newword{canonical roots} of $\RS'$ are a canonical pair of simple roots for $\RS'$, namely the unique pair of roots containing $\RS'\cap\RSpos$ in its nonnegative span.
Any two distinct positive roots in $\RS$ are contained in a unique rank-$2$ subsystem of~$\RS$, namely the intersection of $\RS$ with the span of the two roots.

Given $\beta,\gamma\in\RSpos$, say that $\gamma$ \newword{cuts} $\beta$ if $\gamma$ is a canonical root of the rank-$2$ subsystem $\RS'$ containing them, but $\beta$ is not a canonical root in $\RS'$.
Write $\cut(\beta)$ for the set of roots $\gamma$ such that $\gamma$ cuts $\beta$.
For each $\beta\in\RSpos$, consider the set $\beta^\perp\setminus\bigcup_{\gamma\in\cut(\beta)}\gamma^\perp$.
The closures of the connected components of this set are the \newword{shards} in $\beta^\perp$.
Less formally, we ``cut'' each $\beta^\perp$ along its intersections with all hyperplanes $\gamma^\perp$ with $\gamma\in\cut(\beta)$, and the resulting pieces are the shards.

We define $D$ to be the closed convex cone $\bigcap_{i=1}^n\set{x\in V^*: \br{x,\alpha_i}\ge 0}$ in $V^*$.
The Coxeter group $W$ acts on $V^*$ by the action dual to its action on $V$.
On the basis of fundamental weights, 
\begin{equation}\label{dual on rho}
s_k(\rho_j)=\begin{cases}
\rho_j&\text{if }j\neq k,\text{ or}\\
\rho_k-\sum_{i=1}^na_{ik}\rho_i&\text{if }j=k.
\end{cases}
\end{equation}

The map $w\mapsto wD$ is an injective map to a set of cones with disjoint interiors.  
The union of these cones is a convex cone called the \newword{Tits cone} $\Tits(A)$.

Each shard $\Sigma$ is a cone defined by inequalities of the form $\br{x,\beta}\le0$ for $\beta\in\RS$, so in particular if $\Sigma$ intersects $\Tits(A)$ in codimension $1$, then $\Sigma\cap\Tits(A)$ is a union of codimension-$1$ faces of cones $wD$ for $w\in W$.
Call $w$ an \newword{upper element} of $\Sigma$ if $wD\cap\Sigma$ has codimension $1$ in $V^*$ and $\inv(w)$ contains the positive root $\beta$ such that $\Sigma\subseteq\beta^\perp$.
We write $\Upper(\Sigma)$ for the set of upper elements of $\Sigma$, and we consider $\Upper(\Sigma)$ as a partially ordered set, with the order induced by the weak order on~$W$.

We now establish a bijection between shards intersecting $\Tits(A)$ and join-irreducible elements of $W$.
Given a join-irreducible element $j$, we write $j_*$ for the unique element covered by~$j$ and define $\Sh(j)$ to be the unique shard containing $jD\cap j_*D$.
Versions of the following proposition for finite hyperplane arrangements can be found as \cite[Proposition~2.2]{hplane_dim} and \cite[Proposition~3.5]{congruence}. 

\begin{proposition}\label{U Sigma ji}
Given a shard $\Sigma$ intersecting $\Tits(A)$ in codimension $1$, the poset $\Upper(\Sigma)$ has a unique minimal element $\ji(\Sigma)$, which is also the unique element of $\Upper(\Sigma)$ that is join-irreducible in $W$.
The maps $\Sigma\mapsto\ji(\Sigma)$ and $j\mapsto\Sh(j)$ are inverse bijections between the set of shards intersecting $\Tits(A)$ and the set of join-irreducible elements of the weak order on~$W$.
\end{proposition}

To prove Proposition~\ref{U Sigma ji}, we need a special property of the weak order.
As explained in the paragraph after \cite[Theorem~8.1]{typefree}, every interval in the weak order on $W$ is a finite semidistributive lattice, meaning that the following conditions hold for all $x,y,z\in W$ such that $x\join y\join z$ exists:
If $x\join y=x\join z$ then $x\join y=x\join(y\meet z)$, and if $x\meet y=x\meet z$ then $x\meet y=x\meet(y\join z)$.

We also need the following fact about rank-$2$ subsystems of $\RS$ and inversion sets of elements of $W$, which is well known.
(See, for example, \cite[Lemma~2.17]{typefree}.)

\begin{lemma}\label{biconv}
Suppose $w\in W$ and $\RS'$ is a rank-$2$ subsystem of $\RS$ with canonical roots $\beta$ and $\gamma$.
\begin{enumerate}[label=\bf\arabic*., ref=\arabic*]  
\item If $\beta,\gamma\in\inv(w)$ then $\RS'\subseteq\inv(w)$.
\item \label{biconv item 2}
If $\phi\in(\RS'\setminus\set{\beta,\gamma})$ and $\phi\in\inv(w)$, then $\inv(w)\cap\set{\beta,\gamma}\neq\emptyset$.
\item If $\phi,\psi\in\RS'\cap\inv(w)$ and $\chi\in\RS$ is a positive linear combination of $\phi$ and $\psi$, then $\chi\in\inv(w)$.
\end{enumerate}
\end{lemma}

\begin{proof}[Proof of Proposition~\ref{U Sigma ji}]
We first check that the poset $\Upper(\Sigma)$ is connected.
Given $v$ and $w$ in $\Upper(\Sigma)$, choose a point $p\in vD\cap\Sigma$ and a point $q\in wD\cap\Sigma$.
Choose generically, so that the line segment $\seg{pq}$ does not pass through any intersections of  hyperplanes in $\set{\beta^\perp:\beta\in\Phi^+}$ such that the intersection is of codimension $>2$ in $V^*$.
The entire line segment is in $\Sigma\cap\Tits(A)$ because this is a convex set.
Since $\inv(v)$ and $\inv(w)$ are both finite, their symmetric difference is also finite, and therefore we move along $\seg{pq}$ from $p$ to $q$ while passing through finitely many hyperplanes.
Every time $\seg{pq}$ passes from a cone $xD\cap\Sigma$ to an adjacent cone $yD\cap\Sigma$, with $x,y\in\Upper(\Sigma)$, since it does not leave $\Sigma$, it crosses through the hyperplanes in a finite rank-$2$ subsystem $\RS'$ in which $\beta$ is a canonical root.
We see that either $\inv(x)\cap\RS'=\set{\beta}$ and $\inv(y)\cap\RS'=\RS'$, or vice versa.
No other hyperplanes separate $xD$ from $yD$, and we see that $x\le y$ or vice versa.
We take finitely many such steps in passing from $v$ to $w$, and we conclude that $\Upper(\Sigma)$ is connected.

Suppose $\Sigma$ intersects $\Tits(A)$ in codimension $1$.
In particular, $\Sigma$ is contained in~$\beta_t^\perp$ for some reflection $t$.
Also, $\Upper(\Sigma)$ is nonempty, so it has at least one minimal element. 
(Every nonempty subset of $W$ has a minimal element in the weak order:  Take an element with minimal length.)

We check that an element $j$ of $\Upper(\Sigma)$ is minimal in $\Upper(\Sigma)$ if and only if it is join-irreducible in $W$.
If $j$ is join-irreducible in $W$, then every element strictly below $j$ is weakly below $j_*=tj$, which has $\beta_t\not\in\inv(tj)$.
On the other hand, if $j\in\Upper(\Sigma)$ is not join-irreducible, then let $t'$ be a reflection such that $j\covers t'j$ and $t'\neq t$.
Let $\RS'$ be the rank-$2$ subsystem containing $\beta_t$ and $\beta_{t'}$.
Since both $\inv(j)\setminus\set{\beta_t}$ and $\inv(j)\setminus\set{\beta_{t'}}$ are inversion sets of elements of $W$, Lemma~\ref{biconv} implies that both $\beta_t$ and $\beta_{t'}$ are canonical roots in $\RS'$ and that $\inv(j)\cap\RS'=\RS'$.
We also see that there is an element $w$ of $\Upper(\Sigma)$ with $\inv(w)=(\inv(j)\setminus\RS')\cup\set{\beta_t}$, so $j$ is not minimal in $\Upper(\Sigma)$.

Let $j$ be a minimal element of $\Upper(\Sigma)$.
As we have checked, $j$ is join-irreducible.
If $j$ is not the unique minimal element, then since $\Upper(\Sigma)$ is connected, there is a path in $\Upper(\Sigma)$ from $j$ to another minimal element.
Let $x$ be the first element along that path that is not $\ge j$, and let $y$ be the element immediately before $x$ (i.e.\ closer to $j$) on the path.
Thus $x\le y$ and $j\le y$ but $j\not\le x$.
We have $j\covers tj$, $x\covers tx$, and $y\covers ty$, and in each case the difference in inversion sets is $\set{\beta_t}$.
Thus $j\join ty=x\join ty=y$, but $j\meet x<j$, and thus $j\meet x\le j_*= tj$.
We have $(j\meet x)\join ty=ty$, contradicting semidistributivity in the interval below $y$.
Thus $j$ is the unique minimal element.

Now $j=\ji(\Sh(j))$ because $j\in\Upper(\Sh(j))$.
Since $\ji(\Sigma)$ is minimal in~$\Upper(\Sigma)$, the intersection $\ji(\Sigma)D\cap(\ji(\Sigma))_*D$ is contained in $\Sigma$, so $\Sigma=\Sh(\ji(\Sigma))$.
\end{proof}

\subsection{The affine almost positive roots model}\label{affdenom sec}
We now describe the affine version of the almost positive roots model.
The finite version is in \cite{FoZe03a,FoZe03,MRZ}, and more details on the affine version can be found in \cite{afforb,affdenom}.

The Cartan matrix $A$ (and the associated root system $\RS$ and Coxeter group $W$) are said to be of \newword{affine type} if $A$ is positive semidefinite but all of its principal minors are positive.
More details on root systems of affine type are in \cite[Chapter~4]{Kac90}.
If $A$ is of affine type, then there is a unique positive imaginary root $\delta$ such that the imaginary roots in $\RS$ are $\set{k\delta:k\in\integers\setminus\set{0}}$.

We depart from the tradition of taking $A$ to be an $(n+1)\times(n+1)$ matrix indexed by $\set{0,\ldots,n}$.
Instead, we choose an index $\aff\in\set{1,\ldots,n}$ and call $\alpha_\aff$ the \newword{affine simple root}.
We write $S_\fin=S\setminus\set{s_\aff}$ and write $W_\fin$ for the standard parabolic subgroup generated by $S_\fin$.
The choice of the index $\aff$ is made in such a way that $W$ is a semidirect product of $W_\fin$ with the lattice generated by $\set{\alpha_i\ck:i\neq \aff}$.
(In the traditional indexing, the affine simple root is indexed by $0$.)
We write $V_\fin$ for the linear span of $\set{\alpha_i:i\neq\aff}$ and $\RSfin$ for the finite irreducible root system $\RS\cap V_\fin$.

The fixed space in $V$ of a Coxeter element $c=s_1\cdots s_n$ is spanned by $\delta$.
We define~$\gamma_c$ to be the unique vector in $V_\fin$ that is a generalized $1$-eigenvector associated to $\delta$ (meaning that $c\gamma_c=\delta+\gamma_c$). 

The set $\eigenspace{c}=\set{v\in V:K(\gamma_c,v)=0}$ is a hyperplane in~$V$.
We write $\RSTfin{c}$ for the set $\RS\cap\eigenspace{c}\cap V_\fin$.
This is a finite root system of rank $n-2$, each of whose components are of type A.
We define $\APTre{c}$ to be the union of the $c$-orbits of the positive roots in $\RSTfin{c}$ and we define $\APT{c}$ to be $\APTre{c}\cup\set\delta$.
Every root contained in $\eigenspace{c}$ has a finite $c$-orbit, so $\APT{c}$ is a finite set.

The set of roots that participate in the affine almost positive roots model is the set $\AP{c}=-\Simples\cup(\RSpos\setminus \eigenspace{c})\cup\APT{c}$.
The set $\APre{c}=\AP{c}\setminus\set{\delta}$ of real roots in $\AP{c}$ is the set of $\dd$-vectors of cluster variables in the cluster algebra associated to $B$ \cite[Theorem~1.2]{affdenom}.  

We consider several permutations of $\AP{c}$ defined in \cite{affdenom} following the finite-type definitions in \cite{FoZe03a,FoZe03,MRZ}.
For each $s\in S$ that is initial or final in $c$, define an involution $\sigma_s:\AP{c}\to\AP{scs}$ by 
\[\sigma_s(\alpha)=
\begin{cases}
\alpha & \text{if }\alpha\in-\Simples\setminus\set{-\alpha_s}\\
s(\alpha) & \text{otherwise}.
\end{cases}\]
For $c=s_1\cdots s_n$, define $\tau_c:\AP{c}\to\AP{c}$ by $\tau_c=\sigma_{s_1}\cdots\sigma_{s_n}$.

We now prepare to define the $c$-compatibility degree, which associates an integer to each pair of roots in $\AP{c}$.
We will need additional notation related to roots contained in $\eigenspace{c}$.

The set $\RST{c}=\RS\cap\eigenspace{c}$ is essentially a direct sum of root systems of affine type, with one component (of affine type $\tilde A$) for each component of $\RSTfin{c}$.
However, in each component, the isotropic direction is $\reals\delta$.
Since the spans of these components intersect (in $\reals\delta$), $\RST{c}$ is not a reducible root system in the usual sense, but is a root system in a sense considered in \cite{Deodhar} and \cite{Dyer}.  
(See \cite[Theorem~2.7]{typefree}, \cite[Remark~2.13]{typefree}, and \cite[Remark~1.7]{afforb}.)
We write $\SimplesT{c}$ for the unique minimal set of roots whose nonnegative span is $\RST{c}\cap\RSpos$.
These should be thought of as the simple roots of $\RST{c}$.
The Dynkin diagram for $\RST{c}$ consists of cycles and the Coxeter element $c$ of $W$ acts as a rotation on each cycle, moving each root in $\SimplesT{c}$ to an adjacent root in the cycle.

The following lemma is \cite[Lemma~3.11]{affdenom}, rewritten in light of the fact that $\sigma_s$ coincides with $s$ on positive roots.
\begin{lemma}\label{s pos real in tubes}
If $s$ is initial or final in $c$, then $\APTre{scs}=s(\APTre{c})$.
\end{lemma}

Each real root $\beta\in\APTre{c}$ has a unique expression as a linear combination of vectors in $\SimplesT{c}$.
We write $\SuppT(\beta)$ for the set of vectors appearing in this expression with nonzero coefficient.

Given $\alpha\in\APTre{c}$ and $\simpleT\in\SimplesT{c}$, say $\simpleT$ is \newword{adjacent} to $\alpha$ if $\simpleT$ is in $\SuppT(c\alpha)\cup\SuppT(c^{-1}\alpha)$ but not in $\SuppT(\alpha)$.   
Given $\alpha,\beta\in\APTre{c}$, define $\adj\alpha\beta$ to be the number of roots in $\SuppT(\beta)$ that are adjacent to $\alpha$.
Furthermore, define
\[
  \cmcirc\alpha\beta c=
  \left\lbrace\begin{array}{rl}
    -1&\text{if }\alpha=\beta,\\
    0&\text{if }\SuppT(\alpha)\subsetneq\SuppT(\beta)\text{ or }\SuppT(\beta)\subsetneq\SuppT(\alpha),\\
    \adj\alpha\beta&\text{otherwise.}
  \end{array}\right.
\]

The \newword{$c$-compatibility degree} is the unique function $\AP{c}\times\AP{c}\to\integers$, given by $(\alpha,\beta)\mapsto\cm\alpha\beta c$ satisfying the following conditions:
\begin{align}
\label{compat base}
\cm{-\alpha_i}{\beta}c&=\br{\rho\ck_i,\beta}\text{ for }\alpha_i\text{ simple and }\beta\in\AP{c},\\
\label{compat cobase}
\cm{\beta}{{-\alpha_i}}c&=\br{\rho_i,\beta\ck}\text{ for }\alpha_i\text{ simple and }\beta\in\AP{c},\\
\label{compat U}		
\cm{\alpha}{\beta}c&=\cmcirc\alpha\beta c\text{ if }\alpha,\beta\in\APTre{c},\\
\label{compat delta U}
\cm{\delta}{\alpha}c&=\cm{\alpha}{\delta}c=0\text{ if }\alpha\in\APT{c},\text{ and}\\
\label{compat tau} 
\cm\alpha\beta c&=\cm{\tau_c\alpha}{\tau_c\beta}c.
\end{align} 

Two roots $\alpha$ and $\beta$ in $\AP{c}$ are \newword{$c$-compatible} if and only if ${\cm\alpha\beta c=0}$.
Although the $c$-compatibility degree is not symmetric, as a consequence of \cite[Proposition~4.12]{affdenom}, $c$-compatibility is a symmetric relation.

A \newword{$c$-cluster} is a maximal set of pairwise $c$-compatible roots in $\AP{c}$.
An \newword{imaginary $c$-cluster} is a $c$-cluster containing $\delta$, and a \newword{real $c$-cluster} is a $c$-cluster not containing $\delta$.

Given a set $C$ of pairwise $c$-compatible roots in $\AP{c}$, let $\Cone(C)$ denote the nonnegative real span of $C$.
The \newword{affine $c$-cluster fan} or \newword{affine generalized associahedron fan} $\Fan_c(\RS)$ consists of the cones $\Cone(C)$ for all sets $C$ of pairwise $c$-compatible roots in $\AP{c}$. 

If $\delta\in C$, then $\Cone(C)$ is an \newword{imaginary cone}.
Otherwise $\Cone(C)$ is a \newword{real cone}.
The subfan of $\Fan_c(\RS)$ consisting of real cones is the \newword{real affine $c$-cluster fan} $\Fan_c^\re(\RS)$.

We are interested in the image of $\Fan_c(\RS)$ under the piecewise linear map $\nu_c$ defined in Section~\ref{cox sort sec}.
The following theorem is essentially \cite[Theorem~1.1(1)]{affdenom}, although it also incorporates the immediate observation that $\nu_c$ is linear on the union of the imaginary cones in $\Fan_c(\RS)$ (because these cones are all in the nonnegative span of the simple roots).

\begin{theorem}\label{nu thm}
Suppose that $B$ is an acyclic exchange matrix of affine type, let $\RS$ be the associated root system, and let $c$ be the associated Coxeter element.  
The piecewise-linear homeomorphism $\nu_c$ acts linearly on each cone of $\Fan_c(\RS)$ and thus defines a complete fan $\nu_c(\Fan_c(\RS))$.
The isomorphism $\nu_c$ from $\Fan_c(\RS)$ to $\nu_c(\Fan_c(\RS))$ restricts to an isomorphism from $\Fan_c^\re(\RS)$ to the $\g$-vector fan. 
\end{theorem}
(The $\g$-vector fan associated to $B$ is defined from the cluster variables in a principal-coefficients cluster algebra.
We will not need the detailed definition here, but to avoid confusion, we point out that the $\g$-vector is defined to be an integer vector, but here we interpret vectors in $V^*$ as $\g$-vectors by taking fundamental-weight coordinates.)

\section{Beginning the cluster scattering diagram proofs}\label{scat proofs 1}
In this section, we begin to prove Theorem~\ref{DCScat}.
The proof will take several sections.

Since the notion of $\ScatTB$ is defined only up to equivalence, to show that $\ScatTB$ ``is'' $\DCScat(A,c)$, we will show that $\DCScat(A,c)$ satisfies the finiteness condition, show that it contains the walls $\set{(\alpha_i^\perp,1+\hy_i):i=1,\ldots,n}$, show that every other wall is outgoing, and show that $\DCScat(A,c)$ is consistent.
Since every wall in $\DCScat(A,c)$ has a nontrivial scattering term, and since no two walls are parallel, it is easy to see that $\DCScat(A,c)$ has minimal support as well.

The finiteness condition on $\DCScat(A,c)$ is supplied by the following proposition, which is a combination of \cite[Theorem~8.3]{typefree} and the case $k=1$ of \cite[Theorem~6.1]{sortable}.
(See also \cite[Corollary~3.9]{camb_fan}.)
\begin{proposition}\label{ji ref}
Suppose $W$ is a Coxeter group and $c$ is a Coxeter element of~$W$.
Then for any reflection $t$, there is at most one $c$-sortable join-irreducible element whose unique cover reflection is $t$.
If $W$ is finite, then there is exactly one.
\end{proposition}

Immediately from Proposition~\ref{ji ref}, we can conclude that at most two walls of $\DCScat(A,c)$ are in the same hyperplane, which is already enough for the finiteness condition.
(We will see later that there is at most one wall in each hyperplane.
There would be two walls when there exist a join-irreducible $c$-sortable element $j$ and a $c^{-1}$-sortable element $j'$ such that $\Sh(j)$ and $-\Sh(j')$ are different shards in the same hyperplane.
We will rule out that possibility in Proposition~\ref{Sigma j j'}.)

Every wall $(\alpha_s^\perp,1+\hy^{\alpha_s})$ is present because each $s\in S$ is $c$-sortable and join-irreducible, and $\Sh(s)=\alpha_s^\perp$ because $\cut(\alpha_s)=\emptyset$.

Thus to prove Theorem~\ref{DCScat}, it remains to verify that every other wall is outgoing, and $\DCScat(A,c)$ is consistent.
We prove the assertion about outgoing walls in three propositions, one for each of the kinds of walls in $\DCScat(A,c)$, namely walls associated to $c$-sortable join-irreducible elements, walls associated to $c^{-1}$-sortable join-irreducible elements, and the wall $(\d_\infty,f_\infty)$.
We state and prove the two propositions having to do with sortable elements in this section:

\begin{proposition}\label{out c} 
If $j\in\JIrr_c(W)\setminus S$, then $(\Sh(j),f_j)$ is outgoing and gregarious.
\end{proposition}

\begin{proposition}\label{out cinv}
If $j\in\JIrr_{c^{-1}}(W)\setminus S$, then $(-\Sh(j),f_j)$ is outgoing and gregarious.
\end{proposition}

To prove these propositions, we need to quote and prove a few additional background results.

The following lemma is \cite[Lemma~1.1]{sortable}.
It implies in particular that a join-irreducible element is in $W_I$ if and only if its unique cover reflection is.

\begin{lemma}
\label{cover para}
For any $I\subseteq S$, an element $w\in W$ is in~$W_I$ if and only if every cover reflection of~$w$ is in~$W_I$.
\end{lemma}

Given a Coxeter element $c$ and a subset $I\subseteq S$, the \newword{restriction} of $c$ to $W_I$ is the Coxeter element $c'$ of $W_I$ obtained by deleting the letters $S\setminus I$ from a reduced word for $c$.
(Recall that for any $w\in W$ and $I\subseteq S$, we defined an element $w_I\in W_I$.
In general $c_I$ is not equal to the restriction $c'$ of $c$ to $W_I$.
For example, if $S=\set{s_1,s_2}$ with $c=s_1s_2\neq s_2s_1$ and $I=\set{s_2}$, then $c'=s_2$ but $c_I$ is the identity.)
We write~$V_I$ for the subspace $\Span\set{\alpha_i:i\in I}$ of $V$.
The following lemma is immediate from the definition.

\begin{lemma} \label{OmegaRestriction}
Let $I\subseteq S$ and let $c'$ be the restriction of $c$ to $W_I$.
Then $\omega_c$ restricted to $V_I$ is $\omega_{c'}$. 
\end{lemma}

The \newword{inversion sequence} of a word $a_1\cdots a_k$ (in the alphabet $S$) is the sequence $(a_1\cdots a_{i-1}\alpha_{a_i}:i=1\ldots,k)$.
If $a_1\cdots a_k$ is a reduced word for $w$, then the $k$ roots in the inversion sequence of $a_1\cdots a_k$ are distinct and are precisely the roots in $\inv(w)$.
If $s\in S$ and $s<w$ in the weak order, then by writing a reduced word $a_1\cdots a_k$ for $w$ with $a_1=s$, we see that $\inv(w)={s\cdot\inv(sw) \cup\set{\alpha_s}}$.

The following is one direction of \cite[Proposition~3.11]{typefree}
\begin{proposition}\label{sort omega}
Let $v$ be a $c$-sortable element with $c$-sorting word $a_1\cdots a_k$ and let $\beta_1,\ldots,\beta_k$ be the inversion sequence of $a_1\cdots a_k$.
Then $\omega_c(\beta_i,\beta_j)\ge0$ for all $i<j$, and furthermore if $\omega_c(\beta_i,\beta_j)=0$ then $K(\beta_i,\beta_j)=0$.
\end{proposition}

The following lemma is \cite[Lemma~3.8]{typefree}.

\begin{lemma} \label{OmegaInvariance}
If~$s$ is initial or final in $c$, then $\omega_c(\beta,\beta')=\omega_{scs}(s\beta,s\beta')$ for all roots $\beta$ and $\beta'$.
\end{lemma}

We prove the following lemma for general Coxeter groups.
The finite-type version is \cite[Lemma~3.9]{congruence}.
\begin{lemma}\label{Sigma s}
Suppose $j$ is join-irreducible in $W$ and $\cov(j)=\set{t}$.
Then $\Sh(j)=\beta_t^\perp$ if and only if $j\in S$.
\end{lemma}
\begin{proof}
One direction is easy because $\cut(\alpha_s)=\emptyset$ for all $s\in S$.
On the other hand, suppose $j\not\in S$, so that $\ell(j)>1$.
Given a reduced word $a_1\cdots a_k$ for $j$, let $m$ be minimal such that the word $a_m\cdots a_k$ contains only the letters $a_{k-1}$ and~$a_k$.
Furthermore, choose the reduced word $a_1\cdots a_k$ to minimize $m$.
The roots $a_1\cdots a_{r-1}\alpha_{a_r}$ for $r=m,\ldots,k$ are all contained in the same rank-$2$ subsystem $\RS'$, with canonical roots $a_1\cdots a_{m-1}\alpha_{a_{k-1}}$ and $a_1\cdots a_{m-1}\alpha_{a_k}$.
One of these two is $a_1\cdots a_{m-1}\alpha_{a_m}$.
If the other is $a_1\cdots a_{k-1}\alpha_{a_k}$, then using a single braid move, we can obtain a new reduced word ending in $a_{k-1}$.
Then $j\covers ja_k$ and $j\covers ja_{k-1}$, contradicting the fact that $j$ is join-irreducible.
We conclude that $a_1\cdots a_{m-1}\alpha_{a_m}$ cuts $a_1\cdots a_{k-1}\alpha_{a_k}$.
But the latter is $\beta_t$, so $\Sh(j)\neq\beta_t^\perp$.
\end{proof}

Throughout the section, it will be useful to set a convention on being above or below a hyperplane.
Given a root $\beta\in\RS$ and a set $U\subseteq V^*$, we say that $U$ is \newword{above} the hyperplane $\beta^\perp$ if every point in $U$ either is separated from $D$ by $\beta^\perp$ or is on~$\beta^\perp$.
We say that $U$ is \newword{below} $\beta^\perp$ if every point in $U$ is either on $\beta^\perp$ or on the same side of $\beta^\perp$ as $D$.  

A version of the following proposition for finite simplicial hyperplane arrangements is \cite[Lemma~3.7]{congruence}.
The proposition shows in particular that each shard is defined by a finite list of inequalities.
\begin{proposition}\label{shard ineq}
If $j$ is join-irreducible in $W$ and ${\cov(j)=\set{t}}$, then 
\[\Sh(j)=\set{x\in V^*:\br{x,\beta_t}=0\text{ and }\br{x,\gamma}\le0\text{ for all }\gamma\in\cut(\beta_t)\cap\inv(j)}.\] 
\end{proposition}
\begin{proof}
For any root $\gamma\in\cut(\beta_t)$, there is another root $\gamma'\in\cut(\beta_t)$ such that $\gamma$,~$\gamma'$, and $\beta_t$ are all in the same rank-$2$ subsystem $\RS'$ and the canonical roots of $\RS'$ are $\gamma$ and $\gamma'$.
Up to swapping $\gamma$ and $\gamma'$, all of $\Sh(j)$ is above $\gamma^\perp$ and below $(\gamma')^\perp$.
Thus since $j\in\Upper(\Sh(j))$, also $jD$ is above $\gamma^\perp$, or in other words $\gamma\in\inv(j)$.
Thus $\Sh(j)$ is defined as the set of points in $\beta_t^\perp$ above all hyperplanes $\gamma^\perp$ with $\gamma\in\cut(\beta_t)\cap\inv(j)$.
\end{proof}

We will now relate the shards associated to $W$ with the shards associated to its parabolic subgroups.
Given $I\subseteq S$, we write $V_I$ for the subspace $\Span\set{\alpha_i:i\in I}$ of $V$, as before, and write $V^*_I$ for its dual.
The notation $\Proj_I$ stands for the surjection from $V^*$ onto $V_I^*$ that is dual to the inclusion of $V_I$ into $V$.
We define $D_I=\bigcap_{s \in I} \set{x\in V^*_I: \br{x,\alpha_s}\ge 0}$, which equals $\Proj_I(D)$.
For $j$ a join-irreducible element of $W_I$, we write $\Sh_I(j)$ for the shard in $V^*_I$ containing $jD_I\cap j_*D_I$.

The \newword{parabolic subsystem} $\RS_I$ is $\RS\cap V_I$.
The restriction of the cutting relation on $\RS$ to $\RS_I$ coincides with the cutting relation intrinsically defined on $\RS_I$.
Moreover, it is easily verified that when $\beta$ is in $\RS_I$, the set $\cut(\beta)$ is also contained in~$\RS_I$.
(See, for example, \cite[Lemma~6.6]{congruence} for a finite-type argument that generalizes easily.)
It is also well known that for $w\in W_I$, the inversion set $\inv(w)$ is the same set of roots of $\RS_I$ whether it is computed in $W$ or in $W_I$.
Thus for $j$ a join-irreducible element of $W_I$ with $\cov(j)=\set{t}$, 
\[\Sh_I(j)=\set{x\in V_I^*:\br{x,\beta_t}=0\text{ and }\br{x,\gamma}\le0\text{ for all }\gamma\in\cut(\beta_t)\cap\inv(j)}.\]
Comparing this expression for $\Sh_I(j)$ with Proposition~\ref{shard ineq}, we have the following lemma.

\begin{lemma}\label{Sigma para}
Suppose $I\subseteq S$.
If $j\in W$ is join-irreducible and lies in $W_I$, then $\Sh(j)=\Proj_I^{-1}\Sh_I(j)$.
\end{lemma}

If $s$ is initial in $c$ and $j$ is a $c$-sortable join-irreducible element with $s\not\le j$, then~$j$ is in $W_\br{s}$, and in particular we have the following consequence of Lemma~\ref{Sigma para}.

\begin{proposition}\label{Sigma init para}
Suppose $j$ is a $c$-sortable join-irreducible element and $s$ is initial in $c$ with $s\not\le j$.
Then $\Sh(j)=\Proj_\br{s}^{-1}\Sh_\br{s}(j)$.
\end{proposition}

We will also need the following fact relating $\Sh(j)$ to $\Sh(sj)$ when $s<j$.

\begin{proposition}\label{Sigma sj} 
Suppose $j$ is join-irreducible with $\cov(j)=\set{t}$ and let $s\in S$ have $s<j$.
Then $(s\cdot\Sh(sj))\supseteq\Sh(j)$ and 
\[(s\cdot\Sh(sj))\cap\set{x\in V^*:\br{x,\alpha_s}\le0}=\Sh(j)\cap\set{x\in V^*:\br{x,\alpha_s}\le0}.\]
\end{proposition}
\begin{proof}
Since $\pm\alpha_s$ are the only roots that change sign under the action of $s$, for any rank-$2$ subsystem $\RS'$ that does not contain $\alpha_s$, the map $\beta\mapsto s\beta$ preserves the cutting relation in $\RS'$.
Thus in particular, for some root $\gamma$, if the rank-$2$ subsystem $\RS'$ containing $\beta_t$ and $\gamma$ has $\alpha_s\not\in\RS'$, then $\gamma$ cuts $\beta_t$ if and only if $s\gamma$ cuts $s\beta_t$.
Thus $s\cdot\cut(s\beta_t)$ and $\cut(\beta_t)$ agree except possibly that one or the other may contain $\alpha_s$.

Applying $s$ to $\Sh(sj)$ yields a piece of the hyperplane $\beta_t^\perp$.
Proposition~\ref{shard ineq} says that $\Sh(j)$ is defined by the inequalities $\br{x,\gamma}\le0$ for all $\gamma\in\cut(\beta_t)\cap\inv(j)$.
Since $s<j$, $\inv(sj)=s\cdot(\inv(j)\setminus\set{\alpha_s})$.
Again by Proposition~\ref{shard ineq} and by the relationship between $\inv(sj)$ and $\inv(j)$ and between $\cut(\beta_t)$ and $\cut(s\beta_t)$, we see that the inequalities defining $\Sh(sj)$ are obtained from the inequalities defining $\Sh(j)$ as follows:
Throw away the inequality $\br{x,\alpha_s}\le0$ if it is present, and then replace each inequality $\br{x,\gamma}\le0$ with $\br{x,s\gamma}\le0$.
We see that $s\cdot\Sh(sj)$ is defined by a weakly smaller set of inequalities than $\Sh(j)$, but that intersecting both $s\cdot\Sh(sj)$ and $\Sh(j)$ with $\set{x\in V^*:\br{x,\alpha_s}\le0}$ yields the same set.
\end{proof}

\begin{proof}[Proof of Proposition~\ref{out c}]
Suppose $\cov(j)=\set{t}$.
By Lemma~\ref{Sigma s}, the cone $\Sh(j)$ is not all of $\beta_t^\perp$, so that it is at most half of a hyperplane.
Thus the cone $\Sh(j)$ cannot have both $\omega_c(\,\cdot\,,\beta_t)\in\Sh(j)$ and $-\omega_c(\,\cdot\,,\beta_t)\in\relint(\Sh(j))$, so it is enough to prove that $-\omega_c(\,\cdot\,,\beta_t)\in\relint(\Sh(j))$.
We do so by induction on $\ell(j)$ and the rank of $W$.
Let $s$ be initial in~$c$.

If $s\not\le j$, then $j\in W_\br{s}$ and, by induction on rank, $-\omega_{sc}(\,\cdot\,,\beta_t)\in\relint(\Sh_\br{s}(j))$.
By Lemma~\ref{OmegaRestriction} and Proposition~\ref{Sigma init para}, we see that $-\omega_c(\,\cdot\,,\beta_t)\in\relint(\Sh(j))$.

If $s<j$ and $sj=r$ for some $r\in S$, then $j=sr$, and since $j$ is join-irreducible, $r$ and $s$ do not commute.
Since $t=srs$, also $s$ and $t$ do not commute, so that $K(\alpha_s,\beta_t)\neq0$.
But $sr$ is the $c$-sorting word for $j$ and $\alpha_s,\beta_t$ is the inversion sequence of that word, so Proposition~\ref{sort omega} implies that $-\omega_c(\alpha_s,\beta_t)<0$.
Also $\cut(\beta_t)\cap\inv(j)=\set{\alpha_s}$, so $\Sh(j)=\set{x\in\beta_t^\perp:\br{x,\alpha_s}\le0}$ by Proposition~\ref{shard ineq}.
We see that $-\omega_c(\,\cdot\,,\beta_t)\in\relint(\Sh(j))$.

If $s<j$ and $\ell(sj)>1$, then by induction on $\ell(j)$, we have $-\omega_{scs}(\,\cdot\,,\beta_{sts})\in\relint(\Sh(sj))$.
Lemma~\ref{OmegaInvariance} implies that $-\omega_c(\,\cdot\,,\beta_t)\in\relint(s\cdot\Sh(sj))$.
Since $s$ is initial in $c$, by definition of $\omega_c$, we have $-\omega_c(\alpha_s,\beta_t)\le0$, with equality if and only if $t$ is contained in the parabolic subgroup generated by the elements of $S$ that commute with $s$.
However, that parabolic subgroup is reducible, with a component $W'$ generated by the singleton $\set{s}$ and one or more other components.
If $t$ is in the parabolic subgroup, then it is supported in only one of the components.
Since $j\neq s$, we see that $t$ is in a component other than $W'$.
But then $j$ is also in that component by Lemma~\ref{cover para}, contradicting the assumption that $s<j$.
We conclude that $-\omega_c(\alpha_s,\beta_t)<0$, so that $-\omega_c(\,\cdot\,,\beta_t)\in\relint\bigl((s\cdot\Sh(sj))\cap\set{x\in V^*:\br{x,\alpha_s}\le0}\bigr)$.
By Proposition~\ref{Sigma sj}, $-\omega_c(\,\cdot\,,\beta_t)\in\relint(\Sh(j))$.
\end{proof}

\begin{proof}[Proof of Proposition~\ref{out cinv}]
Again, suppose $\cov(j)=\set{t}$.
By Proposition~\ref{out c}, with $c^{-1}$ replacing $c$, we see that $-\omega_{c^{-1}}(\,\cdot\,,\beta_t)$ is in the relative interior of $\Sh(j)$ and that $\omega_{c^{-1}}(\,\cdot\,,\beta_t)$ is not in $\Sh(j)$.
But $\omega_{c^{-1}}(\,\cdot\,,\beta_t)=-\omega_c(\,\cdot\,,\beta_t)$, so $(-\Sh(j),f_j)$ is outgoing and gregarious.
\end{proof}

\section{The doubled Cambrian fan}\label{DC sec}
To complete the proof of Theorem~\ref{DCScat} and the other results on affine cluster scattering diagrams, we need the affine-type version of the doubled Cambrian fan.
This fan coincides with the $\g$-vector fan (see Theorem~\ref{Affine g}, below) and thus constitutes much of the cluster scattering fan.
Indeed, we will see below in Theorem~\ref{DFc complement} that the affine doubled Cambrian fan covers all of $V^*$ except for a certain codimension-$1$ cone, which is precisely the wall $\d_\infty$ of $\DCScat(A,c)$.

\subsection{The doubled Cambrian fan in general}\label{DFc gen}
We now describe the Cambrian fan and doubled Cambrian fan construction from \cite{camb_fan,typefree,framework,afframe}.
The ambient space for both fans is $V^*$. 

We associate a set $C_c(v)$ to each $c$-sortable element $v$, by setting $C_c(v)$ to be $\set{\alpha_1,\ldots,\alpha_n}$ (the set of simple roots) when $v$ is the identity element, and otherwise, for $s$ initial in $c$, setting  
\[C_c(v)=\begin{cases}
C_{sc}(v)\cup\set{\alpha_s}&\text{if }v\not\ge s,\text{ or}\\
sC_{scs}(sv)&\text{if }v\ge s.
\end{cases}\]
Here $C_{sc}(v)$ is the analogous set of roots, with $v$ considered as an element of $W_{\br{s}}$.

For each $c$-sortable element $v$, let $\Cone_c(v)=\set{x\in V^*:\br{x,\beta} \geq 0,\,\forall\beta \in C_c(v)}$.
This is a simplicial cone with inward-facing normal vectors $C_c(v)$.
The \newword{$c$-Cambrian fan} $\F_c$ is the set of cones $\Cone_c(v)$ where $v$ runs over all $c$-sortable elements.
The \newword{doubled $c$-Cambrian fan} $\DF_c=\F_c\cup(-\F_{c^{-1}})$ is the simplicial fan consisting of all cones of $\F_c$ and all antipodal images of cones in $\F_{c^{-1}}$.
(This is a fan by \cite[Theorem~3.24]{afframe}.)
When $W$ is of finite type, $\DF_c$ is a complete fan and coincides with $\F_c$.
When $W$ is of affine type, $\DF_c$ is not complete but, as we will see below, its support is all of $V^*$ except for a codimension-$1$ cone.
When $W$ is infinite but not affine, the support of $\DF_c$ is not dense in $V^*$.

Given a set of $c$-sortable elements of $W$, if the set has a join in the weak order, then that join is also $c$-sortable \cite[Theorem~7.1]{typefree}.
As a consequence, for any $w\in W$, there is a unique maximal $c$-sortable element that is below $w$ in the weak order.
We call this element $\pidown^c(w)$.
Thus $w$ is $c$-sortable if and only if $\pidown^c(w)=w$, and otherwise, $\pidown^c(w)<w$.
The following theorem, which is \cite[Theorem 6.3]{typefree}, shows that the geometry of the $c$-Cambrian fan is closely tied to the combinatorics of $c$-sortable elements.
It also shows that the support of $\F_c$ contains $\Tits(A)$.

\begin{theorem}\label{pidown cone}
Suppose $v$ is $c$-sortable. 
Then $\pidown^c(w)=v$ if and only if ${wD\subseteq\Cone_c(v)}$.
\end{theorem}

The following is \cite[Proposition~6.13]{typefree}.
\begin{proposition}\label{pidown para}
Let $I\subseteq S$ and let $c'$ be the restriction of $c$ to $W_I$.
Then $\pidown^{c'}(w_I)=\pidown^c(w)_I$ for any $w\in W$.
\end{proposition}

The following is an immediate consequence of \cite[Lemma~3.17]{afframe}.
\begin{proposition}\label{ab bel}
For any $s\in S$, every cone in $\DF_c$ is either above $\alpha_s^\perp$ or below~$\alpha_s^\perp$.
Specifically, if $v$ is $c$-sortable, then $\Cone_c(v)$ is above $\alpha_s^\perp$ if and only if $v\ge s$, and if $u$ is $c^{-1}$-sortable, then $-\Cone_{c^{-1}}(u)$ is above $\alpha_s^\perp$ if and only if $u\not\ge s$.
\end{proposition}

The following proposition is a restatement of \cite[Proposition~5.2]{typefree}.
(The proposition in \cite{typefree} is a statement about $C_c(v)$, but here we restate it as an assertion about $\Cone_c(v)$.)

\begin{proposition}\label{wall cov}
Let $v$ be $c$-sortable. 
Then $t\in\cov(v)$ if and only if $\Cone_c(v)$ is above $\beta_t^\perp$ and $\Cone_c(v)\cap\beta_t^\perp$ is a facet of $\Cone_c(v)$.
\end{proposition}

When $s$ is initial in $c$, the element $scs$ is another Coxeter element for $W$ and $sc$ is a Coxeter element of $W_\br{s}$.
We quote a result that relates $\DF_c$ to $\DF_{scs}$ and $\DF_{sc}$ (the latter being constructed in $V^*_\br{s}$. 
For each $s$ in $S$, let $\DF_c^{\ab(s)}$ denote the set of cones in $\DF_c$ that are above $\alpha_s^\perp$ and let $\DF_c^{\bel(s)}$ denote the set of cones in $\DF_c$ below $\alpha_s^\perp$.
As before, for any $I\subseteq S$, let $\Proj_I$ be the surjection from $V^*$ onto $(V_I)^*$ that is dual to the inclusion of $V_I$ into $V$.
The following is \cite[Proposition~3.28]{afframe}.  

\begin{proposition}\label{recursive fan}
Let $s$ be initial in $c$.
Then    
\begin{enumerate}[label=\bf\arabic*., ref=\arabic*]   
\item $\DF_c=\DF_c^{\ab(s)}\cup\DF_c^{\bel(s)}$.
\item \label{s DF}
$\DF_c^{\ab(s)}=s\bigl(\DF_{scs}^{\bel(s)}\bigr)$.
\item \label{DF below}
The map $F\mapsto \Proj_{\br{s}}^{-1}(F)\cap\set{x\in V^*: \br{x,\alpha_s}\ge 0}$ is a bijection from the set of maximal cones of $\DF_{sc}$ to the set of maximal cones of $\DF_c^{\bel(s)}$.
\item The map $F\mapsto \Proj_{\br{s}}^{-1}(F)\cap\set{x\in V^*: \br{x,\alpha_s}\le 0}$ is a bijection from the set of maximal cones of $\DF_{sc}$ to the set of maximal cones of $\DF_{scs}^{\ab(s)}$. 
\end{enumerate}
\end{proposition}

For each $I\subseteq S$, write $c'$ for the restriction of $c$ to $W_I$ and write $\F_{c'}$ for the $c'$-Cambrian fan defined in $V^*_I$.
We now quote and prove some results that relate the $c$-Cambrian fan to the $c'$-Cambrian fan.
The following is \cite[Proposition~9.4]{typefree}.
\begin{proposition}\label{Fc para}
Let $I\subseteq S$ and let $c'$ be the restriction of $c$ to $W_I$.
Then $\Proj_I^{-1}(\F_{c'})\cap\Tits(A)$ is a coarsening of $\F_c\cap\Tits(A)$.
\end{proposition}

We prove the following proposition.
\begin{proposition}\label{para refine}
Let $I\subseteq S$ and let $c'$ be the restriction of~$c$ to $W_I$.
If $F$ is a codimension-$1$ face of $\F_c$ contained in $\beta_t^\perp$ with $t\in W_I$, then ${F\cap\Tits(A)}$ is contained in $\Proj_I^{-1}(F')\cap\Tits(A)$ for some face $F'$ of $\F_{c'}$ of codimension $1$ in~$V^*_I$.
\end{proposition}

For the proof, we will need some more background.  
Recall that $D_I$ is $\Proj_I(D)$.
For each $w\in W$, we have $\Proj_I(w D)\subseteq w_I(D_I)$ and
\begin{equation}\label{Tits decomp}
\bigcup_{\substack{u\in W\\u_I=w_I}}uD=\Proj_I^{-1}(w_ID_I)\cap\Tits(A).
\end{equation}
Let $A_I$ be the Cartan matrix obtained by deleting from $A$ the rows and columns indexed by $S\setminus I$, and let $\RS_I$ be the corresponding sub root system of $\RS$.
As a consequence of \eqref{Tits decomp}, we have $\Tits(A)\subseteq \Proj_I^{-1}(\Tits(A_I))$.

\begin{proof}[Proof of Proposition~\ref{para refine}]
The face $F$ is $\Cone_c(v)\cap\beta_t^\perp$ for a $c$-sortable element $v$ and $t\in\cov(v)$.
Thus by Theorem~\ref{pidown cone}, 
\[F\cap\Tits(A)=\beta_t^\perp\cap\!\!\bigcup_{\substack{w\in W\\\pidown^c(w)=v}}\!\!wD.\]
By Proposition~\ref{pidown para}, $v_I$ is $c'$-sortable and if $\pidown^c(w)=v$ then $\pidown^{c'}(w_I)=v_I$.
Thus 
\[F\cap\Tits(A)\subseteq\beta_t^\perp\cap\!\!\!\bigcup_{\substack{w\in W\\\pidown^{c'}(w_I)=v_I}}\!\!\!wD=\beta_t^\perp\cap\!\!\bigcup_{\substack{y\in W_I\\\pidown^{c'}(y)=v_I}}\,\,\bigcup_{\substack{w\in W\\w_I=y}}wD.\]
For any $y\in W_I$, \eqref{Tits decomp} says $\bigcup_{\substack{w\in W\\w_I=y}}wD=\Proj_I^{-1}(yD_I)\cap\Tits(A)$.
Thus $F\cap\Tits(A)$ is contained in 
\[
\beta_t^\perp\cap\!\!\!\bigcup_{\substack{y\in W_I\\\pidown^{c'}(y)=v_I}}\!\!\!\!\Proj_I^{-1}(yD_I)\cap\Tits(A)
=\beta_t^\perp\cap\Proj_I^{-1}(\Cone_{c'}(v_I))\cap\Tits(A),
\]
because $\Tits(A)\subseteq \Proj_I^{-1}(\Tits(A_I))$. 
Since $t$ is a cover reflection of $v$, we have $\inv(tv)=\inv(v)\setminus\set{t}$, and since $t\in W_I$, we have $\inv((tv)_I)=\inv(v_I)\setminus\set{t}$, and thus $t\in\cov(v_I)$.
Therefore, $\Cone_{c'}(v_I)$ has a facet $F'$ of codimension $1$ in $V^*_I$ contained in $\beta_t^\perp\subseteq V_I^*$.
Now $\Proj_I^{-1}$ applied to $\beta_t^\perp\subseteq V_I^*$ is $\beta_t^\perp\subseteq V^*$.
Thus $\beta_t^\perp\cap\Proj_I^{-1}(\Cone_{c'}(v_I))\cap\Tits(A)$ is $\Proj_I^{-1}(F')\cap\Tits(A)$.
\end{proof}

\subsection{The affine doubled Cambrian fan}
When $A$ is of affine type, $\Tits(A)$ is the union of the open halfspace $\set{x\in V^*:\br{x,\delta}>0}$ with the zero vector in $V^*$.
The fact that $\Tits(A)$ is essentially a halfspace in affine type, together with the fact that $\F_c$ covers $\Tits(A)$ in general, means that the doubled Cambrian fan is a particularly useful model in affine type.
We quote some results from \cite{afframe} on $\DF_c$ in the affine case, beginning with \cite[Corollary~1.3]{afframe}:

\begin{theorem}\label{Affine g}
If $B$ is acyclic and of affine type, then $\DF_c$ coincides with the $\g$-vector fan associated to $B$.
\end{theorem}

Recall that $V_\fin$ is the subspace of $V$ spanned by $\set{\alpha_i:i\neq\aff}$. 
We identify the dual space $V^*_\fin$ with $\partial\Tits(A)=\delta^\perp$ by the inclusion that is dual to the projection map from $V$ to $V_\fin$ with kernel $\reals\delta$.  
(This projection is not the same as $\Proj_\br{s_\aff}$.  
See \cite[p. 1455]{afframe}.)
The Coxeter fan in $V^*_\fin$ is the fan defined by the hyperplanes $\beta^\perp$ for $\beta\in\RSfin$.
By the identification of $V^*_\fin$ with $\delta^\perp$, the Coxeter fan in $V^*_\fin$ defines a fan in $\delta^\perp$, which we call the \newword{Coxeter fan in $\delta^\perp$}.
The Coxeter fan in $\delta^\perp$ coincides with the fan in $\delta^\perp$ defined by the intersections with $\delta^\perp$ of the hyperplanes $\beta^\perp$ for $\beta\in\RSfin$, or equivalently the intersections with hyperplanes $\beta^\perp$ for $\beta\in\RSre$.
(These two descriptions are equivalent because every $\beta\in\RSre$ is a scalar multiple of $\phi+k\delta$ for some $\phi\in\RSfin$ and $k\in\integers$, and for $x\in\delta^\perp$, we have $\br{x,\phi+k\delta}=\br{x,\phi}$.)

We define $x_c$ to be $-\omega_c(\,\cdot\,,\delta)\in V^*$.
Since $\omega_c$ is skew-symmetric, $x_c$ is in $\delta^\perp\subset V^*$, which we identified with $V_\fin^*$.
\begin{lemma}\label{nu delta}  
The vector $\nu_c(\delta)\in V^*$ is equal to $\frac12x_c$.
\end{lemma}
\begin{proof}
Since $\delta$ is positive, $\nu_c(\delta)=-E_c(\,\cdot\,,\delta)$.
Thus for any $\gamma\in V$, we have $\br{\nu_c(\delta),\gamma}=-E_c(\gamma,\delta)$.
Since $K(\gamma,\delta)=0$, we see that $\omega_c(\gamma,\delta)$ equals
\[\omega_c(\gamma,\delta)+K(\gamma,\delta)=(E_c(\gamma,\delta)-E_c(\delta,\gamma))+(E_c(\gamma,\delta)+E_c(\delta,\gamma))=2E_c(\gamma,\delta).\qedhere\]
\end{proof}

Recall that $\RSfin^{\omega_c+}$ is the set $\set{\beta\in\RSfin:\omega_c(\beta,\delta)>0}$.
We write $|\DF_c|$ for the union of the cones of $\DF_c$.
The following is \cite[Corollary~4.9]{afframe}. 
\begin{theorem} \label{DFc complement} 
If $A$ is of affine type, then $V^*\setminus|\DF_c|$ is a nonempty open cone in $\partial\Tits(A)$, and thus is of codimension $1$ in $V^*$.
This cone is
\[ V^*\setminus|\DF_c|=\bigcap_{\beta \in \RSfin^{\omega_c+}} \{ x \in \partial \Tits(A) : \br{x,\beta} < 0 \}. \]
Also, $V^* \setminus |\DF_c|$ is the union of the relative interiors of those cones in the Coxeter fan in $\delta^\perp$ that contain $x_c$.  
\end{theorem}

Theorem~\ref{DFc complement} says that the imaginary wall $\d_\infty$ coincides with the closure of $V^*\setminus|\DF_c|$, which is the boundary of $|\DF_c|$.
We can combine results of \cite{affdenom} to describe $\d_\infty$ in another way.
According to \cite[Proposition~6.13]{affdenom}, the union of the cones $\Cone(C)$ for imaginary clusters $C$ is a cone whose extreme rays are spanned by $\SimplesT{c}$.
By Theorem~\ref{nu thm}, $\nu_c$ takes that cone to the closure of the complement of the $\g$-vector fan, which according to Theorem~\ref{Affine g} is the closure of $V^*\setminus|\DF_c|$.
We record this fact as the following theorem.

\begin{theorem}\label{d inf extreme}
The extreme rays of the cone $\d_\infty$ are spanned by the vectors $\nu_c(\SimplesT{c})$.
\end{theorem}

The following theorems are (part of) \cite[Corollary~4.18]{afframe} and (all of) \cite[Corollary~4.19]{afframe}.
\begin{theorem} \label{finite type interior}  
Suppose $A$ is of affine type and let $F$ be a face of $\DF_c$ of dimension $n-2$. 
Then the maximal cones of $\DF_c$ containing $F$ form either a finite cycle of adjacent cones or an infinite path of adjacent cones.
If the relative interior of $F$ is in the interior of $|\DF_c|$, then they form a cycle. 
Otherwise they form a path.
\end{theorem}

\begin{theorem}\label{2face in boundary}
If $A$ is of affine type, then every $(n-2)$-dimensional face of $\DF_c$ contained in $\partial\Tits(A)$ is in $\d_\infty$.
\end{theorem}

The following easy consequence of Theorem~\ref{DFc complement} will also be important.
\begin{lemma}\label{G face}
Every cone of the Coxeter fan in $\delta^\perp$ that is not contained in $\d_\infty$ is contained in some face of $\DF_c$.
\end{lemma}
\begin{proof}
Theorem~\ref{DFc complement} implies that every maximal cone $F$ in the Coxeter fan in $\delta^\perp$ is either in $|\DF_c|$ or has lower-dimensional intersection with $|\DF_c|$.
Since all cones in $\DF_c$ are defined by real roots and because the Coxeter fan is the fan defined in $\delta^\perp$ by all of the reflecting hyperplanes of $W$, if $F$ is contained in $|\DF_c|$ then it is contained in a unique maximal cone of $\DF_c$.
Thus every cone of the Coxeter fan, if it is not contained in $\d_\infty$, is contained in the intersection of one or more maximal cones of $\DF_c$.
\end{proof}

As part of \cite[Theorem~1.1]{afframe}, it is proved that the dual graph (i.e.\ adjacency graph of maximal cones) of $\DF_c$ forms what is called a reflection framework.
We will not need full details on reflection frameworks here, but in particular, the dual graph satisfies the ``reflection condition''.
A weak form of this condition can be rephrased as the following property of $\DF_c$.

\begin{proposition}\label{ref cond}
Suppose $C$ and $C'$ are adjacent maximal cones of $\DF_c$ with shared facet $F$ and let $\beta=\pm\beta_t$ be the inward-facing normal to $C$ at the facet~$F$.
Suppose $G\neq F$ is a facet of $C$ with inward-facing normal $\gamma$, let $G'\neq F$ be the facet of $C'$ with $G\cap F=G'\cap F$, and let $\gamma'$ be the inward-facing normal of $C'$ at the facet~$G'$.
Then
\[\gamma'=\left\lbrace\begin{array}{ll}
t\gamma&\mbox{if }\omega_c(\beta_t,\gamma)\ge 0,\mbox{ or}\\
\gamma&\mbox{if }\omega_c(\beta_t,\gamma)<0.\\
\end{array}\right.\]
\end{proposition}

In affine type, we can make the following statement analogous to Proposition~\ref{para refine}.
(The following proposition applies to $\DF_c$ rather than $\F_c$ and also avoids the intersections with $\Tits(A)$.)

\begin{proposition}\label{DFc para}
Suppose $W$ is of affine type.
Let $I\subsetneq S$ and let $c'$ be the restriction of $c$ to $W_I$.
\begin{enumerate}[label=\bf\arabic*., ref=\arabic*]  
\item
Every cone in $\DF_c$ is contained in a cone in $(\Proj_I)^{-1}(\F_{c'})$.
\item
If $F$ is a codimension-$1$ face of $\DF_c$ contained in $\beta_t^\perp$ with $t\in W_I$, then $F$ is contained in $(\Proj_I)^{-1}(F')$ for some face $F'$ of $\F_{c'}$ of codimension~$1$ in~$V^*_I$.
\end{enumerate}
\end{proposition}
\begin{proof}  
Since $I\subsetneq S$ and $W$ is of affine type, $W_I$ is finite, so the doubled $c'$-Cambrian fan coincides with $\F_{c'}$ and with $-\F_{(c')^{-1}}$.

Let $F$ be a cone in $\DF_c$.
Proposition~\ref{Fc para} implies that $F\cap\Tits(A)$ is contained in a cone $F'$ of $(\Proj_I)^{-1}(\F_{c'})$ and that $F\cap(-\Tits(A))$ is contained in a cone $F''$ of $(\Proj_I)^{-1}(-\F_{(c')^{-1}})=(\Proj_I)^{-1}(\F_{c'})$.
Since $I\subsetneq S$, there exists $s\in S\setminus I$, and $(\Proj_I)^{-1}(F')$ contains the line spanned by $\rho_s$ and, in particular, it crosses $\partial\Tits(A)$.
Thus we can take $(\Proj_I)^{-1}(F')=(\Proj_I)^{-1}(F'')$, so that $F'=F''$.

If $F$ is a codimension-$1$ face of $\DF_c$ contained in $\beta_t^\perp$ with $t\in W_I$, then Proposition~\ref{para refine} says that we can take $F'$ to have codimension $1$ in $V^*_I$.
\end{proof}

We now prove several propositions connecting the roots $\AP{c}$, the join-irreducible sortable elements, and the Cambrian fans/doubled Cambrian fans in the affine case.

\begin{proposition}\label{c beta adj}
Let $\RS$ be a root system of affine type and let $c$ be a Coxeter element of the corresponding Weyl group.
Given a hyperplane $H$ in $V^*$, the following are equivalent.
\begin{enumerate}[label=\rm(\roman*), ref=\roman*] 
\item \label{c para}
$H=\beta^\perp$ for a real root $\beta=c^k\gamma$ such that $k\ge0$ and $\gamma$ is contained in a proper parabolic subsystem of $\RS$.
\item \label{camb adj}
There exist adjacent cones of the Cambrian fan $\F_c$ whose shared facet is in~$H$.
\item \label{camb ji}
There exists a $c$-sortable join-irreducible element $j$ with $\cov(j)=\set{t}$ and $H=\beta_t^\perp$.  
\end{enumerate}
The element $j$ in \eqref{camb ji} is unique if it exists.
\end{proposition}

In the proof of Proposition~\ref{c beta adj}, we will use the following three propositions.
The first is an immediate consequence of \cite[Theorem~8.1]{typefree} and \cite[Proposition~8.2]{typefree}, the second is \cite[Proposition 2.30]{typefree}, and the third is an immediate consequence of \cite[Proposition~5.3]{typefree}.  

\begin{proposition}\label{exists j}
If $v$ is a $c$-sortable element and $t\in\cov(v)$, then there exists a $c$-sortable join-irreducible element $j$ with $\cov(j)=\set{t}$.
\end{proposition}

\begin{proposition}\label{sort restrict}
Suppose $I\subseteq S$ and let $c'$ be the restriction of $c$ to $W_I$. 
Then an element $v\in W_I$ is $c$-sortable if and only if it is $c'$-sortable.
\end{proposition}

\begin{proposition}\label{ji s final}
If $j$ is a $c$-sortable join-irreducible element, if $s$ is final in $c$, and if $s\le j$, then $s=j$.
\end{proposition}

\begin{proof}[Proof of Proposition~\ref{c beta adj}]
We will prove \eqref{c para}$\iff$\eqref{camb ji} and \eqref{camb adj}$\iff$\eqref{camb ji}.
The uniqueness of $j$ follows from Proposition~\ref{ji ref}.

First, suppose \eqref{c para} and let $t'$ be such that $\gamma=\beta_{t'}$.
Since $\gamma$ is in a proper subsystem of $\RS$, $t'$ is in a proper parabolic subgroup $W_I$ of $W$.
Write $c'$ for the restriction of $c$ to $W_I$.
Since $W$ is affine, $W_I$ is finite.
By the finite case of Proposition~\ref{ji ref}, there exists a $c'$-sortable join-irreducible element $j'$ with ${\cov(j')=t'}$, and Proposition~\ref{sort restrict} says that $j'$ is also a $c$-sortable element of $W$.

We will prove the following claim:  If there exists a $c$-sortable join-irreducible $j'$ with $\cov(j')=\set{t'}$ then there exists a $c$-sortable join-irreducible $j''$ with $\cov(j'')=\set{ct'c^{-1}}$.
For an easier induction, we prove a stronger claim:
If there exists a $c$-sortable join-irreducible $j'$ with $\cov(j')=\set{t'}$ and if $s$ is final in $c$, then there exists an $scs$-sortable join-irreducible $j''$ with $\cov(j'')=\set{st's}$.
(The claim about $ct'c^{-1}$ follows by $n$ applications of the stronger claim.)

If $s$ is final in $c$ and $s\le j'$, then Proposition~\ref{ji s final} says that $s=j'$, so $j'$ is also $scs$-sortable.
In this case, $t'=s$, so we can take $j''=j'=s$, and $\cov(j'')=\set{s}=\set{st's}$ as desired.
If $s\not\le j'$, then $s\le sj'$, which is $scs$-sortable and has $st's\in\cov(sj')$.
Now Proposition~\ref{exists j} says that there exists an $scs$-sortable join-irreducible element $j''$ with $\cov(j'')=\set{st's}$ as desired.

We have proved the claim.
Applying the claim $k$ times, we conclude that there exists a join-irreducible element $j$ with $\cov(j)=\set{c^kt'c^{-k}}$.
Now $\beta=c^k\gamma=\beta_{c^kt'c^{-k}}$, and we have established \eqref{camb ji}.

Conversely, suppose \eqref{camb ji}.
We will prove \eqref{c para} by showing that $\beta_t=c^k\gamma$ for some $\gamma$ contained in a proper subsystem of $\RS$.
Writing a reduced word $s_1\cdots s_n$ for $c$, recall that the $c$-sorting word for $j$ consists of $k$ copies of $s_1\cdots s_n$ with $k\ge0$ followed by the $c$-sorting word for a $c$-sortable element $j'$ contained in a standard parabolic subgroup.
If $j'$ is the identity, then $t=c^ks_nc^{-k}$ and we are done, with $\gamma=-\alpha_n$.
If $j'$ is not the identity, then $j'$ is join-irreducible.
Let $t'$ be its unique cover reflection.
Since $j'$ is in a proper parabolic subgroup, Lemma~\ref{cover para} implies that $t'$ is in the same proper parabolic subgroup, so $\beta_{t'}$ is in a proper subsystem of $\RS$, and we are done, with $\gamma=\beta_{t'}$.

Now, assume \eqref{camb adj}.
Then there is a $c$-sortable element $v$ such that $H$ defines a facet of $\Cone_c(v)$ and $\Cone_c(v)$ is above $H$.
Thus $H=\beta_t$ for some reflection $t$, and Proposition~\ref{wall cov} says that $t\in\cov(v)$.
Now Proposition~\ref{exists j} says that there is a $c$-sortable join-irreducible element $j$ with $\cov(j)=\set{t}$.
Conversely, assume \eqref{camb ji}.
Theorem~\ref{pidown cone} and the $c$-sortability of $j$ imply that $jD$ and $j_*D$ are in distinct cones of $\F_c$, and Theorem~\ref{pidown cone} implies that these cones are adjacent, and are $\Cone_c(j)$ and $\Cone_c(\pidown^c(j_*))$.
Proposition~\ref{wall cov} says that $H=\beta_t^\perp$ defines a facet $F$ of $\Cone_c(j)$.
Since $F$ contains a shared facet of $jD$ and $tjD=j_*D$ and since $j_*D$ is contained in $\Cone_c(\pidown^c(j_*))$, the cone $\Cone_c(\pidown^c(j_*))$ also has $F$ as a facet.
\end{proof}

The following proposition is part of \cite[Proposition~4.4]{afframe}.
\begin{proposition}\label{codim 1 delta}
No codimension-$1$ face of $\DF_c$ is contained in $\delta^\perp$.
\end{proposition}

As an easy consequence, we obtain the following proposition.

\begin{proposition}\label{facet F DF}
Every pair of adjacent maximal cones in $\DF_c$ is either a pair of adjacent maximal cones in $\F_c$ or a pair of adjacent maximal cones in~$-\F_{c^{-1}}$, or both.
\end{proposition}
\begin{proof}
Proposition~\ref{codim 1 delta} implies that no shared facet of maximal cones in $\DF_c$ is contained in $\delta^\perp$, and as a consequence, given a pair of adjacent maximal cones in $\DF_c$, either both cones intersect the interior of $\Tits(A)$ or both cones intersect the interior of $-\Tits(A)$.
Since the support of $\F_c$ contains $\Tits(A)$ and the support of $-\F_c$ contains $-\Tits(A)$, the proposition follows.
\end{proof}

Recall that $\APre{c}$ is the set $-\Simples\cup(\RSpos\setminus \eigenspace{c})\cup\APTre{c} = \AP{c}\setminus\set\delta$ of real roots in $\AP{c}$.
The following proposition is a rephrasing of \cite[Proposition~3.13(4)]{affdenom}.
\begin{proposition}\label{pmPhic c}
The set of positive real roots in $\AP{c}$ and their negations is the set of all roots $c^k\gamma$ such that $\gamma$ is in a proper subsystem of $\RS$ and $k\in\integers$.
\end{proposition}

\begin{proposition}\label{DC Phic}
Let $\RS$ be a root system of affine type and let $c$ be a Coxeter element of the corresponding Weyl group.
Given a hyperplane $H$ in $V^*$, the following are equivalent.
\begin{enumerate}[label=\rm(\roman*), ref=\roman*] 
\item \label{in Phic}
There exists a positive root $\beta\in\APre{c}$ such that $H=\beta^\perp$.
\item \label{adjacent}
There exist adjacent maximal cones in the doubled Cambrian fan $\DF_c$ whose shared facet is in $H$.
\item \label{either or}
There exist adjacent maximal cones in the Cambrian fan $\F_c$ whose shared facet is in $H$ or there exist adjacent maximal cones in the opposite Cambrian fan $-\F_{c^{-1}}$ whose shared facet is in $H$ (or both).
\item \label{either or j}
There exists a $c$-sortable join-irreducible element $j$ with $\cov(j)=\set{t}$ and ${H=\beta_t^\perp}$ or there exists a $c^{-1}$-sortable join-irreducible element $j'$ with \linebreak ${\cov(j')=\set{t}}$ and $H=\beta_t^\perp$ (or both).
\end{enumerate}
In \textrm{\eqref{either or j}}, $j$ is unique if it exists and $j'$ is unique if it exists.
\end{proposition}
\begin{proof}
The equivalence of \eqref{in Phic} and \eqref{either or} follows from Propositions~\ref{c beta adj} and~\ref{pmPhic c}.
Proposition~\ref{facet F DF} implies that \eqref{adjacent} and \eqref{either or} are equivalent.
The equivalence of \eqref{either or} and \eqref{either or j} follows from Proposition~\ref{c beta adj}.
The uniqueness of $j$ and $j'$ in \eqref{either or j} also follows from Proposition~\ref{c beta adj}.
\end{proof}

\begin{remark}
Some of the equivalences in Proposition~\ref{DC Phic} can also be obtained, in the skew-symmetric case, by combining the main result of \cite{Najera12} with results of \cite{afframe}.
\end{remark}

\section{Completing the cluster scattering diagram proofs}\label{scat proofs 2}
We now complete the proof of Theorem~\ref{DCScat} and the other results on affine cluster scattering fans.
To begin, there is one more wall that must be shown to be outgoing and gregarious.

\begin{proposition}\label{out inf}
The wall $(\d_\infty,f_\infty)$ is outgoing and gregarious.
\end{proposition}
\begin{proof}
By Theorem~\ref{DFc complement}, $\d_\infty$ is the union of the cones in the Coxeter fan in $\delta^\perp$ that contain $-\omega_c(\,\cdot\,,\delta)$.
Thus the wall $(\d_\infty,f_\infty)$ is gregarious.
Since $\RSfin$ has rank at least $1$, no cone of the Coxeter fan contains both $\omega_c(\,\cdot\,,\delta)$ and $-\omega_c(\,\cdot\,,\delta)$.
We conclude that $(\d_\infty,f_\infty)$ is also outgoing.
\end{proof}

The last step in the proof of Theorem~\ref{DCScat}, which will take up most of Section~\ref{scat proofs 2}, is to prove the following proposition:

\begin{proposition}\label{consist}  
$\DCScat(A,c)$ is consistent.
\end{proposition}

\subsection{Rank-2 subsystems of affine type}
At the heart of the proof of Proposition~\ref{consist} is a reduction to the finite or affine rank-$2$ case.
We need to differentiate between type $A_{2k}^{(2)}$ and all other types.
(Accordingly, the definition of $\DCScat(A,c)$ differentiates between these two possibilities.)
The following lemma tells which affine rank-$2$ case we need to consider for type $A_{2k}^{(2)}$ and for all other types.

\begin{lemma}\label{22 14}
Suppose $\beta\in\RSfin^{\omega_c+}$ and suppose $\d_\infty\cap\beta^\perp$ is a facet of $\d_\infty$.
If~$\RS$ is not of type $A_{2k}^{(2)}$, then the rank-$2$ subsystem $\RS'$ containing $\beta$ and $\delta$ is of type~$A_1^{(1)}$.
If $\RS$ is of type $A_{2k}^{(2)}$, then $\RS'$ is of type $A_2^{(2)}$.
\end{lemma}
\begin{proof}
The restriction of $A$ to the span of $\beta$ and $\delta$ is positive semidefinite but not positive definite and not zero.  
Since the restriction is of rank $2$, it is of affine type, and therefore $\RS'$ is of type $A_1^{(1)}$ or $A_2^{(2)}$.
If $\RS'$ is of type $A_2^{(2)}$, then it contains a pair of roots whose ratio of squared lengths is $2$.
But only $A_{2k}^{(2)}$ contains roots with that ratio of squared lengths, and the first assertion of the lemma follows.

Suppose $\RS$ is of type $A_{2k}^{(2)}$.
We will show that for every $\beta$ such that $\d_\infty\cap\beta^\perp$ is a facet of $\d_\infty$, the rank-$2$ subsystem $\RS'$ containing $\beta$ and $\delta$ is of type $A_2^{(2)}$.
If $k=1$, then $\RS'=\RS$, so assume $k>1$.
We first claim that it is enough to consider a single choice of $c$:
Proposition~\ref{recursive fan}.\ref{DF below} implies that, for $s$ initial in $c$, the wall $\d_\infty$ is above the hyperplane $\alpha_s^\perp$.
Thus by Proposition~\ref{recursive fan}.\ref{s DF}, the imaginary wall, as defined for $scs$, is related to the imaginary wall, as defined for $c$, by the action of $s$.
Since the action of $s$ preserves the lengths of roots, it also preserves the type of $\RS'$.
The Dynkin diagram of type $A_{2k}^{(2)}$ is a tree (in fact a path), so all Coxeter elements are related by sequences of moves of the form $c\leftrightarrow scs$ for $s$ initial or final in $c$.

Having proved the claim, we choose a Coxeter element $c=s_1\cdots s_n$ with each $s_i$ and $s_{i+1}$ not commuting, taking $\alpha_1$ to be the longest root.  
Thus, $\alpha_n$ is the shortest root and $\aff=n$.
Then $\delta=\alpha_1+2\sum_{i=2}^n\alpha_i$ and
\[B=\begin{bsmallmatrix*}[r]
0&1\\
-2&0&1\\
&\!-1&\cdot&\\
&&&\,\cdot&\\
&&&&\cdot\,&1\\
&&&&\!\!\!-1&0&1\\
&&&&&\!\!-2&0&\\
\end{bsmallmatrix*}.\]
Thus in the basis of fundamental weights, $x_c=-\omega_c(\,\cdot\,,\delta)$ is the transpose of 
\[-\begin{bsmallmatrix*}[r]
0&1\\
-2&0&1\\
&\!-1&\cdot&\\
&&&\,\cdot&\\
&&&&\cdot\,&1\\
&&&&\!\!\!-1&0&1\\
&&&&&\!\!-2&0&\\
\end{bsmallmatrix*}
\begin{bsmallmatrix*}[r]
1\\2\\2\\\cdot\\\cdot\\\cdot\\2
\end{bsmallmatrix*}
=
\begin{bsmallmatrix*}[r]
-2\\0\\\cdot\\\cdot\\\cdot\\0\\4
\end{bsmallmatrix*}
\]  
Recall that the dual space $V_\fin^*$ for the finite root system is identified with the boundary of the Tits cone by the inclusion dual to the projection map $\pi:V\to V_\fin$ with kernel $\reals\delta$.  
Thus $x_c$ is $\lambda\circ\pi$ for some unique $\lambda\in V^*_\fin$.
Since $\pi$ is the identity on $V_\fin$, we can determine $\lambda$ by how $x_c$ acts on simple roots in $\RSfin$.
We see that $\lambda=-2\rho_1$.
(Recall that $\rho_1$ is the fundamental weight corresponding to $\alpha_1\ck$.)

Theorem~\ref{DFc complement} says that $\d_\infty$ is the union of all cones in the Coxeter fan in $\delta^\perp$ that contain $x_c$.
One of these cones is spanned by (the inclusion of) $\rho_1,\ldots,\rho_{n-1}$.
The normal to the facet opposite $\rho_1$ in this cone is $\alpha_1$, which is a long root.
We see that $\alpha_1$ defines a facet of the imaginary wall $\d_\infty$.
The subsystem $\RS'$ containing $\alpha_1$ and $\delta$ has canonical roots $\alpha_1$ and $\gamma=\sum_{i=2}^n\alpha_i$.
We compute that $\gamma=s_2s_3\cdots s_{n-1}\alpha_n$.
In particular, since $\alpha_n$ is a short root, so is $\gamma$.
Now, $\RS'$ is of affine type because it contains $\delta$, and it is of type $A_2^{(2)}$ because it contains real roots of two different lengths.

We have proved that $\RS'$ is of type $A_2^{(2)}$ when $\beta=\alpha_1$.
Since $\d_\infty$ is the union of all maximal cones of the Coxeter arrangement containing $x_c=-2\rho_1$ and since all of these maximal cones are related by the action of $(W_\fin)_\br{s_1}$, in particular, all of the root subsystems $\RS'$ contemplated in the second assertion of the lemma are related by the action of $W$.
Thus for any $\beta$ as in the lemma, $\RS'$ is of type $A_2^{(2)}$.
\end{proof}

The rank-$2$ scattering diagrams of affine type are known.
The following theorem is \cite[Section~6]{Reineke} in the skew-symmetric case and \cite[Theorem~3.4]{scatcomb} in the non-skew-symmetric case.

\begin{thm}\label{rk2 aff formula}
The scattering term on the limiting wall of $\Scat^T\bigl(\begin{bsmallmatrix}\,\,\,\,\,0&\,2\\-2&\,0\end{bsmallmatrix}\bigr)$ is $\displaystyle\frac1{(1-\hy_1\hy_2)^2}$.
The term on the limiting wall of $\Scat^T\bigl(\begin{bsmallmatrix}\,\,\,\,\,0&\,1\\-4&\,0\end{bsmallmatrix}\bigr)$ is $\displaystyle\frac{1+\hy_1\hy_2^2}{(1-\hy_1\hy_2^2)^2}$.
\end{thm}

\subsection{Sortable elements and shards}
The proof of Theorem~\ref{DCScat} will also need some facts about shards associated to $c$-sortable join-irreducible elements.
We begin by quoting a fundamental fact about $c$-sortable elements.

We say that $w\in W$ is \newword{$c$-aligned} with respect to a rank-$2$ subsystem $\Phi'$ if one of the following cases holds (where $\beta$ and $\gamma$ are the canonical roots of $\Phi'$):
\begin{enumerate}[label=\rm(\roman*), ref=\roman*] 
\item $\omega_c(\beta,\gamma)=0$ and $\inv(w)\cap\Phi'\subseteq\set{\beta, \gamma}$.
\item $\omega_c(\beta,\gamma)>0$ and either $\gamma\not\in\inv(w)$ or $\inv(w)\cap\Phi'=\set{\gamma}$ or $\inv(w) \supseteq \Phi'\cap\RSpos$.
\item $\omega_c(\beta,\gamma)<0$ and either $\beta\not\in\inv(w)$ or $\inv(w)\cap\Phi'=\set{\beta}$  or $\inv(w) \supseteq \Phi'\cap\RSpos$.
\end{enumerate} 
The following is a weaker version of \cite[Theorem~4.3]{typefree}.
\begin{theorem} \label{sortable is aligned}
An element~$w$ of~$W$ is $c$-sortable if and only if $w$ is $c$-aligned with respect to every rank-$2$ subsystem of $\Phi$.
\end{theorem}

\begin{lemma}\label{a cap inv}  
Suppose $j$ is a $c$-sortable join-irreducible element with \linebreak ${\cov(j)=\set{t}}$.
For $\gamma\in\cut(\beta_t)$, the following are equivalent.
\begin{enumerate}[label=\rm(\roman*), ref=\roman*] 
\item $\gamma\in\inv(j)$.  \label{gam inv}
\item $\omega_c(\gamma,\beta_t)>0$.  \label{om gam bet}
\item $\omega_c(\gamma,\gamma')>0$, where $\gamma'$ is the other canonical root in the rank-$2$ subsystem containing $\gamma$ and $\beta_t$.  \label{om gam prime}
\end{enumerate}
\end{lemma}
\begin{proof}
Write $a_1\cdots a_k$ for the $c$-sorting word for $j$ and let $\beta_1,\ldots,\beta_k$ be the inversion sequence of $a_1\cdots a_k$.
Crucially, since $j$ is join-irreducible, $\beta_t=\beta_k$.

Since $\gamma\in\cut(\beta_t)$, we know that $\gamma$ is a canonical root in the rank-$2$ subsystem~$\RS'$ containing $\gamma$ and $\beta_t$ and that $\beta_t$ is not a canonical root in $\RS'$.
Thus, if $\gamma'$ is the other canonical root in $\RS'$, then $\beta_t$ is a positive linear combination of $\gamma$ and $\gamma'$.
Since $\omega_c$ is skew-symmetric, we see that \eqref{om gam bet} and \eqref{om gam prime} are equivalent.

If $\omega_c(\gamma,\gamma')>0$, then since $\beta_t$ is a positive combination of $\gamma$ and $\gamma'$, we see that $\omega_c(\beta_t,\gamma')>0$ and thus $\omega_c(\gamma',\beta_t)<0$.
Since $\beta_t=\beta_k$, Proposition~\ref{sort omega} implies that $\gamma'$ is not in the inversion sequence $\beta_1,\ldots,\beta_k$.
But then Lemma~\ref{biconv} implies that $\gamma\in\inv(j)$.
We see that \eqref{om gam prime} implies \eqref{gam inv}.

Suppose $\gamma\in\inv(j)$.
Then $\gamma=\beta_i$ for some $i\le k$.
Since $\gamma\in\cut(\beta_t)$, $i<k$ and $\gamma$ is a canonical root in the rank-$2$ subsystem $\RS'$ containing $\gamma$ and $\beta_t$.
Moreover, $\beta_t$ is not a canonical root in $\RS'$.
In particular, $\RS'$ has at least $3$ positive roots, and thus, if $\gamma'$ is the other canonical root in $\RS'$, we have $K(\gamma,\gamma')\neq0$.
Proposition~\ref{sort omega} says that $\omega_c(\gamma,\beta_t)\ge0$.
If $\omega_c(\gamma,\beta_t)>0$, then we have established \eqref{om gam bet}.
If $\omega_c(\gamma,\beta_t)=0$, then Proposition~\ref{sort omega} also says that $K(\gamma,\beta_t)=0$.
Since $K(\gamma,\gamma')\neq0$, there exists a root $\gamma''$ in the positive linear span of $\gamma$ and $\beta_t$ with $K(\gamma,\gamma'')\neq0$.
Lemma~\ref{biconv} implies that $\gamma''$ appears between $\gamma$ and $\beta_t$ in the inversion sequence $\beta_1,\ldots,\beta_k$, so Proposition~\ref{sort omega} implies that $\omega_c(\gamma,\gamma'')>0$.
Since $\gamma''$ is a positive combination of $\gamma$ and $\gamma'$, we conclude that $\omega_c(\gamma,\gamma')>0$, and we have established \eqref{om gam prime}.
We have shown that \eqref{gam inv} implies~\eqref{om gam bet} or \eqref{om gam prime}.
\end{proof}

As an immediate consequence of Proposition~\ref{shard ineq} and Lemma~\ref{a cap inv}, we have the following description of shards for $c$-sortable join-irreducible elements.
\begin{proposition}\label{shard ineq sort}
Suppose $c$ is a Coxeter element of $W$ and let $t$ be a reflection in~$W$.
If $j$ is a $c$-sortable join-irreducible element of $W$ and $\cov(j)=\set{t}$, then~$\Sh(j)$ is
\[\set{x\in V^*:\br{x,\beta_t}=0\text{ and }\br{x,\gamma}\le0\text{ for all }\gamma\in\cut(\beta_t)\text{ with }\omega_c(\gamma,\beta_t)>0}.\] 
\end{proposition}

As another consequence of Lemma~\ref{a cap inv}, we have the following relationship between shards for $c$-sortable and $c^{-1}$-sortable join-irreducible elements.

\begin{proposition}\label{Sigma j j'}  
Suppose $c$ is a Coxeter element of $W$ and let $t$ be a reflection in $W$.
If there exist both a $c$-sortable join-irreducible element $j$ with $\cov(j)=\set{t}$ and a $c^{-1}$-sortable join-irreducible element $j'$ with $\cov(j')=\set{t}$, then $\Sh(j')=-\Sh(j)$.
\end{proposition}
\begin{proof}
Consider the inequalities of Proposition~\ref{shard ineq sort} for both $\Sh(j)$ and $\Sh(j')$.
For any $\gamma\in\cut(\beta_t)$, take $\gamma'$ as in Lemma~\ref{a cap inv}.
Since $\omega_{c^{-1}}(\gamma,\gamma')=-\omega_c(\gamma,\gamma')$, we see that the inequalities for $\Sh(j')$ are obtained from the inequalities for $\Sh(j)$ by replacing each $\gamma$ by $\gamma'$.
Since $\gamma$ and $\gamma'$ are the canonical roots in a rank-$2$ subsystem in which $\beta_t$ is not a canonical root, for $x\in\beta_t^\perp$, the inequality $\br{x,\gamma'}\le0$ is equivalent to $\br{x,\gamma}\ge0$.
We conclude that $\Sh(j')=-\Sh(j)$.
\end{proof}

\begin{proposition}\label{j j' sort}
Suppose $W$ is of finite or affine type. 
Suppose $j$ is join-irreducible and $c$-sortable and $\Sh(j)$ intersects the interior of $-\Tits(A)$.
Then there exists a $c^{-1}$-sortable element $j'$ such that $\Sh(j)=-\Sh(j')$.
\end{proposition}
\begin{proof}
Let $t$ be such that $\cov(j)=\set{t}$.
By Proposition~\ref{Sigma j j'}, it is enough to show that there exists a $c^{-1}$-sortable join-irreducible element $j'$ with $\cov(j')=\set{t}$.
Then by Proposition~\ref{exists j}, it is enough to show that there exists a $c^{-1}$-sortable element~$v$ with $t\in\cov(v)$.
Let $s$ be initial in $c$.

If $s\not\le j$, then $j\in W_\br{s}$, and thus $\beta_t$ is in a proper subsystem of $\RS$, so Proposition~\ref{c beta adj} implies that there exists a $c^{-1}$-sortable join-irreducible element $j'$ with $\cov(j')=\set{t}$.

If $s=j$, then $t=s$, and $j'=s$ is a join-irreducible $c^{-1}$-sortable element with $\cov(j')=\set{t}$.

If $s<j$, then $sj$ is a join-irreducible $scs$-sortable element with $\cov(sj)=\set{sts}$ and $sts\neq s$.
By Proposition~\ref{Sigma sj}, we have $\Sh(sj)\supseteq(s\cdot\Sh(j))$, and since the action of $s$ preserves $-\Tits(A)$, we see that $\Sh(sj)$ intersects the interior of $-\Tits(A)$.
By induction on $\ell(j)$, there exists an $(scs)^{-1}$-sortable element $u$ with $sts\in\cov(u)$.

We will show, more strongly, that there is an $(scs)^{-1}$-sortable element $u$ with $sts\in\cov(u)$ and also $s\le u$.
Since there exists an $(scs)^{-1}$-sortable element with $sts$ as a cover reflection, Proposition~\ref{exists j} says that there exists a join-irreducible $(scs)^{-1}$-sortable element $j''$ with $\cov(j'')=\set{sts}$.
If $s\not\le j''$, then $j''\in W_\br{s}$. 

We will show that the join $s\join j''$ exists.
Since $\Sh(j)$ intersects the interior of $-\Tits(A)$, there exists an element $w\in W$ such that $w(-D)$ is separated from $-D$ by the hyperplane $\beta_{sts}^\perp$ and intersects $\Sh(j)$ in a facet of $w(-D)$.
Thus $wD$ is separated from $D$ by $\beta_{sts}^\perp$ and intersects $-\Sh(j)$ in a facet of $wD$.
In other words, $w$ is an upper element of the shard $-\Sh(j)$.
Since $\Sh(sj)\supseteq(s\cdot\Sh(j))$, we see that $sw$ is an upper element of the shard $-\Sh(sj)$.
Since $j''$ is the unique minimal upper element of $-\Sh(sj)$, we have $sw\ge j''$.
Also $sw\ge s$, because $w\not\ge s$.
So $sw$ is an upper bound for $j''$ and $s$.
Thus $s\join j''$ exists.

Since $j''$ and $s$ are both $(scs)^{-1}$-sortable, $s\join j''$ is also $(scs)^{-1}$-sortable.  
Now \mbox{\cite[Lemma~2.8]{sort_camb}} says that if $x\in W_\br{s}$, then $\cov(x\join s)=\cov(x)\cup\set{s}$.
Thus $u=s\join j''$ is the desired $(scs)^{-1}$-sortable element with $sts\in\cov(u)$ and also $s\le u$.

Now $s$ is final in $scs$ and thus initial in $(scs)^{-1}$.
Since $u$ is $(scs)^{-1}$-sortable and $s\le u$, also $su$ is $c^{-1}$-sortable.
Because $sts\in\cov(u)$, there exists $r\in S$ such that $(sts)u=ur\covered u$.
Since $sts\neq s$, we have $s\le ur$, so $s(ur)\covered su$.
That is, $t(su)=(su)r\covered su$, and thus $t\in\cov(su)$.
We have constructed a $c^{-1}$-sortable element having $t$ as a cover reflection.
\end{proof}

The following proposition relates the shards associated to $c$- and $c^{-1}$-sortable elements to the fan $\nu_c(\Fan_c^\re(\RS))$ in affine type.
Once we have proved our main results, we will be able to improve on this proposition (as Corollary~\ref{explicit shards cor}).

\begin{proposition}\label{explicit shards}
Let $\RS$ be a root system of affine type and let $c$ be a Coxeter element of the corresponding Weyl group.
\begin{enumerate}[label=\bf\arabic*., ref=\arabic*]  
\item \label{explicit j}
If $j$ is a $c$-sortable join-irreducible element with $\cov(j)=\set{t}$, then every $(n-1)$-dimensional face of $\nu_c(\Fan_c(\RS))$ contained in $\beta_t^\perp$ is contained in $\Sh(j)$ and $\Sh(j)\cap(\Tits(A)\cup(-\Tits(A)))$ is covered by $(n-1)$-dimensional faces of $\nu_c(\Fan_c(\RS))$.
\item \label{explicit j'}
If $j'$ is a $c^{-1}$-sortable join-irreducible element with $\cov(j')=\set{t'}$, then every $(n-1)$-dimensional face of $\nu_c(\Fan_c(\RS))$ contained in $\beta_t^\perp$ is contained in $-\Sh(j')$ and $-\Sh(j')\cap(\Tits(A)\cup(-\Tits(A)))$ is covered by $(n-1)$-dimensional faces of $\nu_c(\Fan_c(\RS))$.
\end{enumerate}
\end{proposition}

For the proof of Proposition~\ref{explicit shards}, we will use the following less detailed version of \cite[Proposition~5.3]{typefree}.

\begin{proposition}\label{Cc final}
Let~$s$ be final in~$c$ and let $v$ be $c$-sortable with $v \geq s$.
Then $v=s\join v_{\br{s}}$ and $C_c(v)=\set{-\alpha_s}\cup C_{cs}(v_\br{s})$.
\end{proposition}

We will also use the following proposition to show that the two assertions of Proposition~\ref{explicit shards} are equivalent.

\begin{proposition}\label{nu cinv}
$\nu_{c^{-1}}(\Fan_{c^{-1}}(\RS))=-\nu_c(\Fan_c(\RS))$.
\end{proposition}
\begin{proof}
\cite[Proposition~3.2]{affdenom} says that $\AP{c}=\AP{c^{-1}}$.
Furthermore, \cite[Proposition~4.10]{affdenom} says that $c$-compatibility and $c^{-1}$-compatibility coincide.
As an immediate consequence, $\Fan_c(\RS)=\Fan_{c^{-1}}(\RS)$.
Now \eqref{compat tau} implies that the map $\tau_c$ on~$\AP{c}$ induces a piecewise linear automorphism of $\Fan_c(\RS)$, and \cite[Lemma~9.3]{affdenom} says that~$\tau_c$ coincides with the map $\beta\mapsto\nu_{c^{-1}}^{-1}(-\nu_c(\beta))$.
\end{proof}

\begin{proof} [Proof of Proposition~\ref{explicit shards}]
By Proposition~\ref{nu cinv}, it is enough to prove Assertion~\ref{explicit j}.
Assertion~\ref{explicit j'} then follows by replacing $c$ with $c^{-1}$.

Suppose $j$ is a $c$-sortable join-irreducible element and $\cov(j)=\set{t}$.
Recall from Theorem~\ref{nu thm} that $\nu_c(\Fan_c^\re(\RS))$ equals $\DF_c$ and from Theorem~\ref{DFc complement} that the closure of the complement of $|\DF_c|$ is $\d_\infty$.
Since $\d_\infty\subseteq\delta^\perp=\partial(\Tits(A))$, we can replace $\nu_c(\Fan_c(\RS))$ by $\DF_c$ throughout Assertion~\ref{explicit j}.

Let $F$ be an $(n-1)$-dimensional face of $\DF_c$ contained in $\beta_t^\perp$.
We will prove by induction on $\ell(j)$ and on the rank of $W$ that $F\subseteq\Sh(j)$.
Let $s$ be initial in $c$.

If $s\not\le j$, then $j$ is an $sc$-sortable join-irreducible element of $W_\br{s}$.
Thus also $t\in W_\br{s}$. 
By Proposition~\ref{DFc para}, $F$ is contained in $(\Proj_\br{s})^{-1}(F')$ for some face $F'$ of $\F_{sc}$ of dimension~$n-2$ in~$V^*_\br{s}$.
By induction on rank, $F'$ is contained in $\Sh(j)^\br{s}$.
Now Proposition~\ref{Sigma init para} implies that $F$ is contained in $\Sh(j)$.

If $s=j$, then $t=s$.
Also, $\Sh(j)=\beta_t^\perp$ by Lemma~\ref{Sigma s}. 
Thus $F\subseteq\Sh(j)$.

If $s<j$, then $j\not\in W_\br{s}$, so $t\not\in W_\br{s}$.
We will show that $F$ is above $\alpha_s^\perp$.
If $F$ is a face of $\F_c$ then it is $\Cone_c(v)\cap\beta_t^\perp$ for some $c$-sortable element $v$ with $t\in\cov(v)$.
In particular, $v$ is not in $W_\br{s}$, and thus $v\ge s$.
By Proposition~\ref{ab bel}, $\Cone_c(v)$ is above $\alpha_s^\perp$, so $F$ is above $\alpha_s^\perp$.
If $F$ is not a face of $\F_c$, then it is a face of $-\F_{c^{-1}}$.
In this case, by Proposition~\ref{ab bel}, if $F$ is not above $\alpha_s^\perp$ then it is $-\Cone_{c^{-1}}(u)\cap\beta_t^\perp$ for some $c^{-1}$-sortable element $u$ with $u\ge s$.
Then Proposition~\ref{Cc final} implies that every element of $C_{c^{-1}}(u)$ is either $-\alpha_s$ or is a root in $\RS_\br{s}$.
But then $\beta_t\not\in\pm C_{c^{-1}}(u)$, contradicting the fact that $-\Cone_{c^{-1}}(u)$ has a facet defined by $\beta_t^\perp$.
We conclude that $F$ is above $\alpha_s^\perp$ in either case.

Observe that since $s<j$, the element $sj$ is join-irreducible and has $\cov(sj)=\set{sts}$.
Thus by Proposition~\ref{recursive fan}.\ref{s DF} and induction on $\ell(j)$, $sF\subseteq\Sh(sj)$.
Now, by the equality in Proposition~\ref{Sigma sj}, since $F$ is above $\alpha_s^\perp$, $F$ is contained in $\Sh(j)$.

To complete the proof of Assertion~\ref{explicit j}, we show that $\Sh(j)\cap(\Tits(A)\cup(-\Tits(A)))$ is covered by $(n-1)$-dimensional faces of $\DF_c$.
Suppose $w$ and $x$ are elements of $W$ such that $wD\cap xD$ has codimension $1$ in $V^*$.
We must show that if $wD\cap xD$ is contained in $\Sh(j)$, then there exists a codimension-$1$ cone of $\F_c$ containing $wD\cap xD$.
We must also show that if $-(wD\cap xD)$ is contained in $\Sh(j)$, then there exists a codimension-$1$ cone of $-\F_{c^{-1}}$ containing $-(wD\cap xD)$.

First, suppose $wD\cap xD$ is contained in $\Sh(j)$.
Without loss of generality, $w\covers x$ in the weak order.
By Proposition~\ref{U Sigma ji}, $w\ge j$ in the weak order.
It is also immediate that $x\not\ge j$, since $\beta_t\not\in\inv(x)$.
Since $\pidown^c(w)$ is the unique largest $c$-sortable element below $w$, and similarly for $\pidown^c(x)$, we conclude that $\pidown^c(w)\neq\pidown^c(x)$.
Thus $wD$ and $xD$ are in different cones of $\F_c$.
These cones intersect in a codimension-$1$ face of $\F_c$ containing $wD\cap xD$.

Now suppose that $-(wD\cap xD)$ is contained in $\Sh(j)$.
In particular, $\Sh(j)$ intersects the interior of $-\Tits(A)$, so Proposition~\ref{j j' sort} says that there is a $c^{-1}$-sortable join-irreducible element $j'$ with $\Sh(j)=-\Sh(j')$.
Thus $-(wD\cap xD)$ is contained in $-\Sh(j')$.
By the case proved above (replacing $c$ by $c^{-1}$), there exists a codimension-$1$ face of $-\F_{c^{-1}}$ containing $-(wD\cap xD)$ and we are done.
\end{proof}

\subsection{Consistency}
In this section, we prove Proposition~\ref{consist}.
To reduce the number of cases, we will use the following basic fact about $\ScatTB$.
(See, for example, \cite[Proposition~2.4]{scatcomb}.)

\begin{proposition}\label{scat antip}
For any exchange matrix $B$, 
\[\Scat^T(-B)=\set{(-\d,f_\d((\hy')^\beta)):\,(\d,f_\d(\hy^\beta))\in\ScatTB},\]
where as usual $\hy_i=y_ix_1^{b_{1i}}\cdots x_n^{b_{ni}}$ while $\hy'_i=y_ix_1^{-b_{1i}}\cdots x_n^{-b_{ni}}$. 
\end{proposition}

Given an $(n-2)$-dimensional intersection $F$ of walls of the scattering diagram $\DCScat(A,c)$, we use the phrase \newword{small loop about $F$} to denote a limit of closed curves about $F$, each passing exactly once through the relative interior of every wall containing $F$ and, in the limit, intersecting no other wall.
Such a limiting loop can be obtained by starting with a closed curve about $F$ close to a point in the relative interior of $F$ and performing a dilation that limits to that relative interior point.
A small loop about $F$ has a well defined path-ordered product, because for any $k\ge 1$, once the loop is dilated small enough, its path-ordered product relative to $\DCScat_k(A,c)$ does not change.
The path-ordered product also does not depend on the precise choice of a small loop (except up to reversing direction).

\begin{proof}[Proof of Proposition~\ref{consist}]  
The proof proceeds by checking consistency locally, as we now explain.
Let $F$ be an $(n-2)$-dimensional intersection of walls of $\DCScat(A,c)$.
It is enough to check that for $\gamma$ a small loop about $F$, the path-ordered product $\p_\gamma$ is trivial.
As pointed out in the proof of \cite[Proposition~4.7]{scatcomb} (the finite-type analog of this theorem), passing from the rank-$2$ case to a loop about $F'$ is a simple change of variables.
Specifically, the normal vectors to walls containing $F$ are all contained in the $2$-dimensional plane (in $V$) orthogonal to $F$.
Thus these normal vectors are roots in a rank-$2$ subsystem $\RS'$ of $\RS$.

By \cite[Proposition~2.2]{scatcomb} (which rephrases \cite[Proposition~2.5]{scatfan} in the conventions of this paper), to check consistency, it is enough to compute path-ordered products applied to monomials $x^\lambda$ for vectors $\lambda\in P$.
However, for a small loop about $F$, inspection of the wall-crossing automorphisms as described in \eqref{theta def x} and \eqref{theta def hat y} reveals that we can ignore the component of $\lambda$ that is in the linear span of $F$ (because that component is orthogonal to $\RS'$), and that we only ever consider the restriction of $\omega_c$ to the plane in $V$ orthogonal to $F$.
Thus it is enough in every case to verify that the rank-$2$ scattering diagram induced on a plane complementary to $F$ is consistent, relative to the restriction of $\omega_c$.

We will consider three cases for how $F$ is situated relative to $\pm\Tits(A)$:
Either $F\cap\Tits(A)$ is $(n-2)$-dimensional, and/or $F\cap(-\Tits(A))$ is $(n-2)$-dimensional, or $F$ is contained in the boundary of $\Tits(A)$.
In light of Proposition~\ref{scat antip}, applying the antipodal map and replacing $c$ by $c^{-1}$ turns the case where $F\cap(-\Tits(A))$ is $(n-2)$-dimensional into the case where $F\cap(\Tits(A))$ is $(n-2)$-dimensional, so we only argue one of these cases.

The case where $F\cap\Tits(A)$ is $(n-2)$-dimensional is analogous to the finite-type case treated in \cite[Proposition~4.7]{scatcomb}, and the argument given there extends without alteration to the affine case (because, although $\DF_c$ is infinite, the set of cones of $\DF_c$ containing $F$ is finite).

More specifically, if $F\cap\Tits(A)$ is $(n-2)$-dimensional, then $F$ contains an $(n-2)$-dimensional face $F'$ of $\F_c$.
Theorem~\ref{finite type interior} implies that a small loop $\gamma$ about $F'$ intersects finitely many $(n-1)$-dimensional faces of $\F_c$ (or equivalently of $\DF_c$).
By Propositions~\ref{c beta adj} and~\ref{explicit shards}, these intersections with faces of $\DF_c$ are precisely the intersections with the walls of $\DCScat(A,c)$.  
\cite[Theorem~9.8]{typefree} describes the star of $F$ in $\DF_c$ as a finite rank-$2$ Cambrian fan.
The rank-$2$ exchange matrix associated to this fan is the matrix describing the restriction of $\omega_c$ in the basis of canonical roots of $\RS'$.
(See \cite[Theorem~2.3(1)]{afframe}.)
Thus the consistency of a small loop about $F$ is equivalent to the consistency of rank-$2$ Cambrian scattering diagrams of finite type, just as in the proof of \cite[Proposition~4.7]{scatcomb}.

We have completed the argument when $F\cap\Tits(A)$ and/or $F\cap(-\Tits(A))$ is $(n-2)$-dimensional.
It remains to consider the case where $F$ is in $\partial\Tits(A)$.
Since each wall of $\DCScat(A,c)$ is defined by hyperplanes orthogonal to roots, in this case $F$ is a finite union of faces of the Coxeter fan in $\delta^\perp$ that are $(n-2)$-dimensional.
For simplicity, we will check consistency for a small loop about every $(n-2)$-dimensional face $G$ of the Coxeter fan.

Let $\RS'$ be the rank-$2$ subsystem of $\RS$ consisting of roots that are orthogonal to~$G$.
Write $\beta$ and $\gamma$ for the canonical roots of $\RS'$.
Without loss of generality, we can take $\beta\in\RSfin$.  
To see why, take any real root in $\RS'$.
This root is a positive scalar multiple of $\phi+k\delta$ for some $\phi\in\RSfin$ and $k\in\integers$.
We take $\beta$ to be whichever of $\phi$ or $-\phi$ is positive.
Since $\delta$ is orthogonal to $G$, so is $\beta$, and thus $\beta\in\RS'\cap\RSfin$.
Every root in $\RS'$ is a linear combination of $\beta$ and $\delta$, and every positive root in $\RS'\setminus\RSfin$ has a positive coefficient of $\alpha_\aff$ in its expansion in the basis of simple roots.
We conclude that $\beta$ is not a nonnegative linear combination of other positive roots in $\RS'$.
In other words, $\beta$ is a canonical root of $\RS'$.

If $\omega_c(\beta,\delta)=0$, then $\omega_c(\beta,\gamma)=0$.
In this case, Theorem~\ref{sortable is aligned} implies that there are at most two $c$-sortable join-irreducible elements whose unique cover reflection corresponds to a root in $\RS'$, and each such root is a canonical root in $\Phi'$.
Thus each associated shard is not cut along $\delta^\perp$, so if it contains $G$, it intersects $-\Tits(A)$.
Similarly, there are at most two $c^{-1}$-sortable join-irreducible elements whose unique cover reflection corresponds to a root in $\RS'$, and each is a canonical root in $\Phi'$.
Now Proposition~\ref{j j' sort} implies that any wall of $\DCScat(A,c)$ that contains $G$ and is orthogonal to a real root must contain $G$ in its relative interior.
But also, if $\d_\infty$ contains $G$, then $G$ is in the relative interior of $\d_\infty$ by Theorem~\ref{DFc complement}, because $\omega_c(\beta,\delta)=0$.
We see that a small loop about $G$ crosses $1$, $2$, or $3$ walls twice each, once in each direction.
Since the restriction  of $\omega_c$ to the span of $\RS'$ is zero, all wall crossings commute and thus the path-ordered product about the loop is trivial.

If $\omega_c(\beta,\delta)<0$, then $G$ is not in the relative interior of $\d_\infty$.
To see why, recall that $\d_\infty$ is the closure of the cone $V^*\setminus|\DF_c|$ described in Theorem~\ref{DFc complement}.
The root $-\beta$ is in $\RSfin^{\omega_c+}$ because ${\omega_c(\beta,\delta)<0}$, so $\d_\infty$ is contained in $\set{x\in\partial\Tits(A):\br{x,\beta}\ge0}$, while $G$ is contained in $\beta^\perp$.
Suppose $G$ is not contained in $\d_\infty$.
Then Lemma~\ref{G face} says that $G$ is contained in some face $G'$ of $\DF_c$, and we may as well take $G'$ minimal, so that $G$ is not in the relative boundary of $G'$.
If $G'$ is $(n-2)$-dimensional, then since $G$ is $(n-2)$-dimensional and in $\partial\Tits(A)$, also $G'$ is in $\partial\Tits(A)$.  
But now Theorem~\ref{2face in boundary} says that $G'$ is in $\d_\infty$.
This contradicts our supposition that $G$ is not in $\d_\infty$, and we conclude that $G'$ has dimension $>n-2$.
If $G'$ is $n$-dimensional, then since $G$ is not in the boundary of $G'$, there is a loop around $G$ that is contained in the interior of $G'$.
Thus Proposition~\ref{explicit shards} implies that a small loop around $G$ intersects no walls of $\DCScat(A,c)$, so that consistency about the loop is trivial.
If $G'$ is $(n-1)$-dimensional, then Propositions~\ref{DC Phic}, \ref{j j' sort}, and~\ref{explicit shards} combine to imply that a small loop  about $G$ intersects only a single wall of $\DCScat(A,c)$, so that consistency is again trivial.

Continuing in the case where $\omega_c(\beta,\delta)<0$, it remains to consider the possibility that $G$ is in $\d_\infty$ (specifically in the proper face $\d_\infty\cap\beta^\perp$ of $\d_\infty$).
Since $\beta$ and $\gamma$ are canonical in the rank-$2$ subarrangement $\RS'$, the hyperplanes $\beta^\perp$ and $\gamma^\perp$ cut all other hyperplanes containing $G$ (the hyperplanes orthogonal to roots in $\RS'\setminus\set{\beta,\gamma}$).
The rank-$2$ subarrangement $\RS'$ is infinite because it contains an infinite collection of roots that are scalings of roots $\beta+k\delta$ for integers $k$.
Thus Theorem~\ref{sortable is aligned} says that if $j$ is a $c$-sortable join-irreducible element with $\cov(j)=\set{t}$ such that $\beta_t$ is a non-canonical root in $\RS'$, then $\beta\not\in\inv(j)$.
Thus for such $j$, the shard $\Sh(j)$ is below~$\beta^\perp$.
Similarly, if $j$ is a $c^{-1}$-sortable join-irreducible element with $\cov(j)=\set{t}$ such that $\beta_t$ is a non-canonical root in $\RS'$, then $\Sh(j)$ is below $\gamma^\perp$ and thus $-\Sh(j)$ is below $\beta^\perp$. 
We have already seen that $\d_\infty$ is below $\beta^\perp$.
Thus the hyperplanes orthogonal to roots in $\RS'\setminus\set{\beta,\gamma}$ contain no walls of $\DCScat(A,c)$ on the side of $G$ opposite $\d_\infty$.
Since $G$ is in the intersection of two or more walls, and since $\DF_c$ is a fan, we conclude that both the wall contained in $\beta^\perp$ and the wall contained in $\gamma^\perp$ contain $G$.
Both of these walls extend above and below $G$ into $\Tits(A)$ and $-\Tits(A)$.

Now Theorem~\ref{finite type interior} says that the set of maximal cones of $\DF_c$ containing $G$ forms a doubly infinite sequence, with each maximal cone adjacent to the cones before and after it.
Propositions~\ref{c beta adj} and~\ref{explicit shards} imply that the intersection of each adjacent pair in the sequence is contained in a wall of $\DCScat(A,c)$, defining a doubly infinite sequence of walls containing $G$.
Since $G$ is contained in the wall in $\beta^\perp$ and in the wall in $\gamma^\perp$, there is a maximal cone of $\DF_c$ containing $G$ and having $\beta$ and $\gamma$ as inward-facing normals.
The normal vectors to all of the walls in the doubly infinite sequence are thus determined from $\beta$ and $\gamma$ by Proposition~\ref{ref cond} and the restriction of $\omega_c$ to $G^\perp$ (the span of $\RS'$).
In both directions, these normal vectors limit to $\pm\delta$, and the imaginary wall $\d_\infty$ also contains $G$.

Now we can check consistency about $G$ by considering the two-dimensional doubled Cambrian fan constructed in the span of $\RS'$, using the restriction of $\omega_c$.
That is, take $B'=\begin{bsmallmatrix}0&\omega_c(\beta\ck,\gamma)\\\omega_c(\gamma\ck,\beta)&0\end{bsmallmatrix}$, take $A'=\begin{bsmallmatrix}2&-|\omega_c(\beta\ck,\gamma)|\\-|\omega_c(\gamma\ck,\beta)|&2\end{bsmallmatrix}$, and take $c'=s_\gamma s_\beta$. 
We obtain the same sequence of normal vectors from Proposition~\ref{ref cond} as we obtained inside $\RS$.
We insert one additional limiting ray to obtain the cluster scattering diagram.
This additional ray corresponds to $\d_\infty$.
Lemma~\ref{22 14} and Theorem~\ref{rk2 aff formula} imply that, with the scattering terms induced by $\DCScat(A,c)$, this rank-$2$ scattering diagram is the cluster scattering diagram, and in particular consistent.
The consistency of the rank-$2$ cluster scattering diagram implies that the path-ordered product on a small loop about $G$ is trivial.

If $\omega_c(\beta,\delta)>0$, then replacing $c$ with $c^{-1}$ in the argument for the case $\omega_c(\beta,\delta)<0$, we again conclude that the path-ordered product for a small loop around $G$ is trivial.
\end{proof}

This completes the proof of Theorems~\ref{DCScat}, \ref{aff greg}, and~\ref{aff greg shard}.
By Proposition~\ref{DC Phic} and the fact that $\delta$ (the normal vector to $\d_\infty$) is the unique imaginary root in~$\AP{c}$, Theorem~\ref{scat Schur} follows.
We now prove our last result on affine cluster scattering diagrams.

\begin{proof}[Proof of Theorem~\ref{easy scat}]
We argue that $\set{(\d_\beta,f_\beta):\beta\in\AP{c}\cap\RSpos}$ is precisely equal to $\DCScat(A,c)$, which is $\ScatTB$ by Theorem~\ref{DCScat}.
Indeed, Proposition~\ref{shard ineq sort} says that each shard $\Sh(j)$ for $j$ a $c$-sortable join-irreducible element is $\d_{\beta_t}$, where ${\cov(j)=\set{t}}$.
Proposition~\ref{shard ineq sort} also implies that for $j$ a $c^{-1}$-sortable join-irreducible element with $\cov(j)=\set{t}$, the wall $-\Sh(j)$ is
\[\set{x\in V^*:\br{x,\beta_t}=0\text{ and }\br{x,\gamma}\ge0\text{ for all }\gamma\in\cut(\beta_t)\text{ with }\omega_{c^{-1}}(\gamma,\beta_t)>0}.\] 
But for each $\gamma\in\cut(\beta_t)$, there exists $\gamma'\in\cut(\beta_t)$ such that the signs of $\omega_{c^{-1}}(\gamma',\beta)$ and $\omega_{c^{-1}}(\gamma,\beta_t)$ are strictly opposite and, for $x\in\beta_t^\perp$, the signs of $\br{x,\gamma}$ and $\br{x,\gamma'}$ are weakly opposite.
Thus $-\Sh(j)$ equals $\d_{\beta_t}$.

We have seen that every wall of $\DCScat(A,c)$ is of the form $(\d_\beta,f_\beta)$ for $\beta\in\AP{c}\cap\RSpos$.
Theorem~\ref{scat Schur} implies that every wall $(\d_\beta,f_\beta)$ is in $\DCScat(A,c)$.
\end{proof}

\subsection{The finite-type construction}
Although the focus of this paper is affine type, we pause to mention what this paper adds to the finite-type results already established in \cite[Section~4]{scatcomb}.

When $\RS$ is of finite type, the doubled Cambrian fan is still a fan, but coincides with the Cambrian fan because $\Tits(A)$ is all of $V^*$ and $\F_c=-\F_{c^{-1}}$ in this case.
Thus the proof we give for Proposition~\ref{explicit shards} can be simplified to give a proof of the following result.
(This result is a consequence of the earlier results, but seems not to have been explicitly stated before.
It is, however, precisely equivalent to \mbox{\cite[Proposition~8.15]{shardint}}, with some translation required to see the equivalence.)

\begin{proposition}\label{explicit shards fin}   
Let $\RS$ be a root system of finite type and let $c$ be a Coxeter element of the corresponding Weyl group.
If $j$ is a $c$-sortable join-irreducible element with $\cov(j)=\set{t}$, then $\Sh(j)$ is the union of all faces of $\F_c$ contained in $\beta_t^\perp$.
\end{proposition}

Our proof of Theorem~\ref{DCScat} simplifies to give an alternate proof of the following result, which is \cite[Corollary~4.10]{scatcomb}.
\begin{theorem}\label{CScat}
If $B$ is an acyclic exchange matrix of finite type, then $\ScatTB$ is $\set{(\Sh(j),f_j):j\in\JIrr_c(W)}$.
\end{theorem}
Indeed, the proof given here is an extension of the proof given in \cite{scatcomb}.
But this paper augments the result because Proposition~\ref{shard ineq sort} provides explicit inequalities defining the walls $\Sh(j)$ of $\ScatTB$.

Finally, our proof of Proposition~\ref{out c} is valid for completely general Coxeter groups, without the assumption of affine type.
In particular, it fills in the details of the finite-type version \cite[Proposition~4.12]{scatcomb}, whose proof was only lightly sketched in~\cite{scatcomb}.

\section{Three fans}\label{fan sec}
In this section, we prove Theorems~\ref{fans thm} and~\ref{gen fans thm}, about coincidences between scattering fans, mutation fans, and generalized associahedron fans in affine type.

\subsection{More on the generalized associahedron fan}\label{gen assoc fan sec}
In preparation for the proof of Theorem~\ref{fans thm}, we will prove the following two facts about the generalized associahedron fan.

\begin{proposition}\label{cluster limit}
Let $\RS$ be a root system of affine type, let $c$ be a Coxeter element, let $C$ be an imaginary $c$-cluster, and let $C'=C\setminus\set{\delta}$.
Then there exists a bi-infinite sequence $(C_i:i\in\integers)$ of distinct real $c$-clusters ${C_i=C'\cup\set{\beta_i,\beta'_i}}$ with
\[\lim_{i\to\infty}\reals_{\ge0}\beta_i=\lim_{i\to-\infty}\reals_{\ge0}\beta_i=\lim_{i\to\infty}\reals_{\ge0}\beta'_i=\lim_{i\to-\infty}\reals_{\ge0}\beta'_i=\reals_{\ge0}\delta.\]
Furthermore, for large enough $i$, the roots $\beta_i$ and $\beta'_i$ are both on one side of $\eigenspace{c}$, while $\beta_{-i}$ and $\beta'_{-i}$ are both on the opposite side of $\eigenspace{c}$.
\end{proposition}

Recall that $\APTre{c}$ is the set of real roots in $\AP{c}$ that are in finite $\tau_c$-orbits.
Recall also the definition of $\d_\beta$ in \eqref{d beta}.

\begin{proposition}\label{walls in d infty2}
Suppose $\RS$ is of affine type and $c$ is a Coxeter element.
If $C$ is a collection of $n-2$ pairwise $c$-compatible roots in $\AP{c}$ with $\delta\in C$, then there exists a root $\beta$ in $\APTre{c}$ with $\nu_c(C)\subset\d_\beta$.
\end{proposition}

For the proofs of these propositions, we quote and prove some additional background.
The following four propositions are parts of \cite[Proposition~5.14]{affdenom}.
In this subsection, we assume $\RS$ to be of affine type.

\begin{proposition}\label{Q n}
If $C$ is a real $c$-cluster, then $C$ consists of $n$ linearly independent roots.
\end{proposition}

\begin{proposition}\label{at least 2}
If $C$ is a real $c$-cluster, then $C$ contains at least $2$ roots in the $\tau_c$-orbits of negative simple roots.
\end{proposition}

\begin{proposition}\label{Qc n-1}
If $C$ is an imaginary $c$-cluster, then $C$ consists of $n-1$ linearly independent roots.
\end{proposition}

\begin{proposition}\label{in the cone}
If $C$ is an imaginary $c$-cluster, then $C\setminus\set{\delta}$ consists of roots in~$\APTre{c}$.
\end{proposition}

The following theorem is \cite[Theorem~5.5]{affdenom}.

\begin{theorem}\label{clus def OK}
A set $C\subseteq\AP{c}$ is a real $c$-cluster if and only if it is a maximal set of pairwise $c$-compatible roots in $\APre{c}$.
\end{theorem}
Theorem~\ref{clus def OK} is not obvious, because conceivably there might exist a maximal set of pairwise $c$-compatible roots in $\APre{c}$, each of which is $c$-compatible with $\delta$.
Such a set would not be a $c$-cluster.

For $j=1,\ldots,n$, define $\psi^\proj_{c;j}=s_1\cdots s_{j-1}\alpha_j$ and $\psi^\inj_{c;j}=s_n\cdots s_{j+1}\alpha_j$. 
Further, define $\TravProj{c}=\set{\psi^\proj_{c;j}:j=1,\ldots,n}$ and $\TravInj{c}=\set{\psi^\inj_{c;j}:j=1,\ldots,n}$.
Recall that $\gamma_c$ is the unique vector in $V_\fin$ with $c\gamma_c=\delta+\gamma_c$ and that $\eigenspace{c}$ is the hyperplane $\set{v\in V:K(\gamma_c,v)=0}$ in~$V$.
The set $\TravProj{c}\cup\TravInj{c}$ has $2n$ distinct elements, the $c$-orbits of these elements are disjoint, and the union of these orbits is $\RSpos\setminus\eigenspace{c}$. 
The following propositions are \cite[Proposition~4.3]{afforb} and \cite[Proposition~3.13(2)]{affdenom}.

\begin{proposition}\label{useful}
The $c$-orbits of roots in $\TravProj{c}$ are separated from the $c$-orbits of roots in $\TravInj{c}$ by the hyperplane $\eigenspace{c}$.
Specifically, $K(\gamma_c,\beta)>0$ for $\beta\in c^{m}\TravProj{c}$ and $m\in\mathbb{Z}$, while 
$K(\gamma_c,\beta)<0$ for $\beta\in c^{m}\TravInj{c}$ and $m\in\mathbb{Z}$.
\end{proposition}

\begin{proposition}\label{apSchur2}
The set of positive roots in $\APre{c}$ is 
\[\set{c^{-k}\psi^\inj_{c;j}:k\ge0,\,1\le j\le n}\cup\set{c^k\psi^\proj_{c;j}:k\ge0,\,1\le j\le n}\cup\APTre{c}.\]
\end{proposition}

We now prove the first result of this section.

\begin{proof}[Proof of Proposition~\ref{cluster limit}]
Proposition~\ref{in the cone} says that $C'\subseteq\APTre{c}$.
Since $\APT{c}$ is finite and fixed (as a set) by the action of $c$, there is some integer $m>0$ such that $c^m$ fixes $C'$ pointwise.
Since $c$ fixes $\delta$ also, by Proposition~\ref{Qc n-1} we see that $c^m$ fixes the span of $C$, which is $\eigenspace{c}$.

Proposition~\ref{Q n} and Theorem~\ref{clus def OK} imply that there exist real roots $\beta$ and $\beta'$ such that $C'\cup\set{\beta,\beta'}$ is a $c$-cluster.
Neither $\beta$ nor $\beta'$ is in $\APTre{c}$.
(If, say, $\beta$ is in $\APTre{c}$, then by \eqref{compat delta U}, $C\cup\set{\beta}$ is a set of pairwise $c$-compatible roots, contradicting the fact that $C$ is a $c$-cluster.)

Let $\beta_i=\tau_c^{mi}(\beta)$, $\beta'_i=\tau_c^{mi}(\beta')$, and $C_i=\tau_c^{mi}(C'\cup\set{\beta,\beta'})=C'\cup\set{\beta_i,\beta'_i}$ for each $i\in\integers$.
By \eqref{compat tau}, each $C_i$ is a $c$-cluster.
By Proposition~\ref{at least 2}, there exists $k\in\integers$ and $j\in\set{1,\ldots,n}$ such that $\tau_c^k(\beta)=-\alpha_j$.
Then also $\tau_c^{k-1}(\beta)=\psi^\inj_{c,j}$ and $\tau_c^{k+1}(\beta)=\psi^\proj_{c,j}$.
Thus we can reindex the sequence so that $\beta_1\in\TravProj{c}$ and $\beta_{-1}\in\TravInj{c}$.

The subspaces $\eigenspace{c}$ and $V_\fin$ are distinct hyperplanes, because $\delta$ is in $\eigenspace{c}$ but not in~$V_\fin$.
Thus $\eigenspace{c}\cap V_\fin$ has codimension $2$.
Since $\gamma_c\not\in\eigenspace{c}$, we can write $\beta_1=a\gamma_c+b\delta+x$ with $x\in \eigenspace{c}\cap V_\fin$.
Since $\beta_1\not\in\eigenspace{c}$, also ${a\neq 0}$.
Now $c$ fixes~$\delta$ and sends $\gamma_c$ to $\gamma_c+\delta$.
Also $c^m$ fixes $x$.
Thus for $i\ge0$, $\tau_c^{mi}$ sends $\beta_i$ to $\beta_{i+1}=c^{mi}\beta_1=a\gamma_c+(b+ami)\delta+x$.
Proposition~\ref{useful} implies that $a>0$, so $\lim_{i\to\infty}\reals_{\ge0}\beta_i=\reals_{\ge0}\delta$.
Similarly, we can write $\beta_{-1}=a'\gamma_c+b'\delta+x'$ with $a'<0$ and compute $\beta_i=a'\gamma_c+(b'+a'mi)\delta+x'$ for $i\le-1$, and conclude that $\lim_{i\to-\infty}\reals_{\ge0}\beta_i=\reals_{\ge0}\delta$.
We obtain the remaining statements about limits by the symmetry between $\beta$ and $\beta'$.

Since $\beta$ and $\beta'$ are not in finite $\tau_c$-orbits, there exist simple roots $\alpha$ and $\alpha'$ and integers $k$ and $k'$ such that $\beta=\tau_c^k(-\alpha)$ and $\beta'=\tau_c^{k'}(-\alpha')$.
Thus for large enough~$i$, $\beta_i$ and $\beta'_i$ are both on the same side of $\eigenspace{c}$, by Proposition~\ref{useful}, and similarly $\beta_{-i}$ and $\beta'_{-i}$ are on the other side.
\end{proof}

Recall from Section~\ref{affdenom sec} that $\RST{c}=\RS\cap\eigenspace{c}$ and that $\SimplesT{c}$ is, in essence, the set of simple roots of $\RST{c}$.
Each root $\beta\in\RST{c}\setminus\set\delta$ has a unique expression as a linear combination of vectors in $\SimplesT{c}$ and $\SuppT(\beta)$ is the set of vectors appearing in this expression with nonzero coefficient.
Recall also that the Dynkin diagram for $\RST{c}$ consists of cycles and that $c$ acts as a rotation on each cycle, moving each root in $\SimplesT{c}$ to an adjacent root in the cycle.

Two roots $\alpha,\beta\in\APTre{c}$ are \newword{nested} if $\SuppT(\alpha)\subseteq\SuppT(\beta)$ or $\SuppT(\beta)\subseteq\SuppT(\alpha)$.
The roots are \newword{spaced} if $\SuppT(c^{-1}\alpha)\cup\SuppT(\alpha)\cup\SuppT(c\alpha)$ is disjoint from $\SuppT(\beta)$.
In particular, roots from distinct components of $\RST{c}$ are spaced.
The following is \cite[Proposition~5.12]{affdenom}.
\begin{proposition}\label{compatible in tubes}
Two distinct roots $\alpha$ and $\beta$ in $\APTre{c}$ are $c$\nobreakdash-compatible if and only if they are nested or spaced.
\end{proposition}

The following is the concatenation of \cite[Lemma~4.7]{affdenom} and \cite[Proposition~4.8]{affdenom}.
\begin{proposition}\label{compute in tubes}
Suppose $\alpha,\beta\in\APTre{c}$.
Then $E_c(\alpha\ck,\beta)$ is the number of roots in $\SuppT(\alpha)\cap\SuppT(\beta)$ minus the number of roots in $\SuppT(c\alpha)\cap\SuppT(\beta)$.  
\end{proposition}

\begin{proposition}\label{compute omega in tubes}
Suppose $\alpha,\beta\in\APTre{c}$.
Then $\omega_c(\alpha\ck,\beta)$ is the number of roots in $\SuppT(c\beta)\cap\SuppT(\alpha)$ minus the number of roots in $\SuppT(c\alpha)\cap\SuppT(\beta)$.  
\end{proposition}
\begin{proof}
We have $\omega_c(\alpha\ck,\beta)=E_c(\alpha\ck,\beta)-E_c(\beta,\alpha\ck)$.
By Proposition~\ref{compute in tubes}, if $\alpha$ and~$\beta$ are in different components of $\RST{c}$, both $E_c(\alpha\ck,\beta)$ and $E_c(\beta,\alpha\ck)$ are zero.
If $\alpha$ and~$\beta$ are in the same component of $\RST{c}$, then they have the same length, so $E_c(\beta,\alpha\ck)=E_c(\beta\ck,\alpha)$.
Thus in any case, $\omega_c(\alpha\ck,\beta)=E_c(\alpha\ck,\beta)-E_c(\beta\ck,\alpha)$.
The proposition now follows from Proposition~\ref{compute in tubes}.
\end{proof}

The following is part of \cite[Proposition~2.16]{affdenom}
\begin{proposition}\label{E delta in tubes}
If $\beta$ is a root contained in $\eigenspace{c}$, then $E_c(\beta,\delta)=0$.
\end{proposition}

In fact, we can augment Proposition~\ref{E delta in tubes} as follows.
\begin{proposition}\label{E omega delta in tubes}
For a vector $\beta\in V$, the following are equivalent.
\begin{enumerate}[label=\rm(\roman*), ref=\roman*]
\item $\beta\in\eigenspace{c}$.
\item $E_c(\beta,\delta)=0$.
\item $\omega_c(\beta,\delta)=0$.
\end{enumerate}
If $\beta\in\AP{c}$, then these conditions are equivalent to $\beta\in\APT{c}$.
\end{proposition}
\begin{proof}
Since the roots $\APT{c}$ span $\eigenspace{c}$, Proposition~\ref{E delta in tubes} implies that $E_c(\beta,\delta)=0$ for all vectors $\beta\in\eigenspace{c}$.
The matrix for $E_c$ is upper unitriangular, so it has no kernel.
Since $\eigenspace{c}$ is a hyperplane, we conclude that $E_c(\beta,\delta)=0$ if and only if $\beta\in\eigenspace{c}$.
Lemma~\ref{nu delta} says that $-\omega_c(\beta,\delta)=2\br{\nu_c(\delta),\beta}=-2E_c(\beta,\delta)$ for any $\beta$, so also $\omega_c(\beta,\delta)=0$ if and only if $\beta\in\eigenspace{c}$.
The last statement of the proposition holds by definition, and is included for emphasis.
\end{proof}

The following proposition accomplishes most of the work for the proof of Proposition~\ref{walls in d infty2}.
We will illustrate the proof in Figures~\ref{case1fig} and~\ref{case2fig}.
Each figure represents a component $\Lambda'$ of the Coxeter diagram of $\RST{c}$.
The Coxeter diagram of $\Lambda'$ consists of a cycle with no edge labels, so that $\Lambda'$ is a root system of affine type $A^{(1)}$.
The support $\SuppT$ of a root in $\APTre{c}\cap\Lambda'$ is an induced path in the cycle.
We represent roots in $\APTre{c}\cap\Lambda'$ pictorially by circling their supports.
When two supports are nested, we draw the larger support around the smaller support.
We draw the cycle so that the action of $c$ on a root is to move the support of the root one step counterclockwise.
The \newword{counterclockwise endpoint} of a root $\phi\in\APTre{c}\cap\Lambda'$ is the element of $\SuppT(\phi)$ that is taken outside of $\SuppT(\phi)$ by counterclockwise rotation.
The \newword{clockwise endpoint} of $\phi$ is defined analogously.

\begin{proposition}\label{walls in d infty1}
If $C$ is a collection of $n-2$ pairwise $c$\nobreakdash-compatible roots in $\APT{c}$ with $\delta\in C$, then there exists a root $\alpha$ in $\APTre{c}$ with the following properties.
\begin{enumerate}[label=\rm(\roman*), ref=\roman*] 
\item \label{beta property 1}
$E_c(\alpha,\beta)=0$ for all $\beta\in C$.
\item \label{beta property 2}
If $\gamma\in\cut(\alpha)\cap\APTre{c}$ and $\omega_c(\gamma,\alpha)>0$, then $E_c(\gamma,\beta)\ge0$ for all $\beta\in C$.
\end{enumerate}
\end{proposition}
\begin{proof}
The roots $\alpha$ and $\gamma$ quantified in the proposition are all in $\APTre{c}$, so Proposition~\ref{E delta in tubes} says that $E_c(\alpha,\delta)=0$ and $E_c(\gamma,\delta)=0$ in every case.
Thus we need to show that for any set $C'$ of $n-3$ pairwise $c$-compatible roots in $\APTre{c}$, there exists a root $\alpha$ in $\APTre{c}$ satisfying conditions \eqref{beta property 1} and \eqref{beta property 2} for all roots $\beta\in C'$.
In light of Proposition~\ref{compute in tubes}, condition \eqref{beta property 1} is that $|\SuppT(\alpha)\cap\SuppT(\beta)|=|\SuppT(c\alpha)\cap\SuppT(\beta)|$ for all $\beta\in C'$.
We will use Proposition~\ref{compatible in tubes} implicitly throughout the proof.

By Propositions~\ref{Qc n-1} and~\ref{in the cone}, there is a root $\phi\in\APTre{c}$ such that $C'\cup\set{\phi,\delta}$ is an imaginary $c$-cluster.
Let $\Lambda'$ be the component of $\RST{c}$ containing $\phi$.
Then the restriction of $C'\cup\set{\phi}$ to $\Lambda'$ is a maximal collection of pairwise $c$-compatible roots in $\APTre{c}\cap\Lambda'$.

First, consider the case where no root $\beta\in C'$ with $\SuppT(\beta)\supsetneq\SuppT(\phi)$ has the same counterclockwise endpoint as $\phi$.
A representative drawing of this case is shown in Figure~\ref{case1fig}, with some roots $\beta\in C'$ omitted.
\begin{figure}
\includegraphics{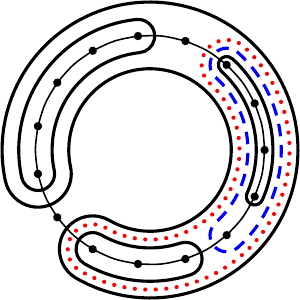}
\caption{An illustration of the proof of Proposition~\ref{walls in d infty1}.}
\label{case1fig}
\end{figure}
In the figure, $\SuppT(\phi)$ is shown with a dotted line to distinguish it from the supports of the roots in $C'$, which are shown with solid lines. 
We now explain how to construct $\alpha$.
There is a unique vertex $v$ of the cycle that is in the support of $\phi$ but is not in the support of any root $\beta\in C'$ with $\SuppT(\beta)\subseteq\SuppT(\phi)$.
Let $\alpha$ be the root whose counterclockwise endpoint is the same as the counterclockwise endpoint of $\phi$, and whose clockwise endpoint is $v$.
In the figure, $\SuppT(\alpha)$ is shown with a dashed line.
We see that $|\SuppT(\alpha)\cap\SuppT(\beta)|=|\SuppT(c\alpha)\cap\SuppT(\beta)|$ for all $\beta\in C'$.
This is condition~\eqref{beta property 1}.

Suppose $\gamma\in\APTre{c}$ cuts $\alpha$ and has $\omega_c(\gamma,\alpha)>0$.
Since $\gamma$ cuts $\alpha$, $\SuppT(\gamma)\subsetneq\SuppT(\alpha)$ and $\SuppT(\gamma)$ contains exactly one endpoint of $\alpha$.
Thus Proposition~\ref{compute omega in tubes} implies that $\SuppT(\gamma)$ contains the counterclockwise endpoint of $\alpha$.
Notice in particular that $\gamma$, $\alpha$, and $\phi$ all have the same counterclockwise endpoint.

Suppose $\beta\in C'$.
If $\beta$ is not in the component $\Lambda'$, then Proposition~\ref{compute in tubes} says that $E_c(\gamma\ck,\beta)=0$.
If $\beta\in\Lambda'$, then Proposition~\ref{compute in tubes} says that $E_c(\gamma\ck,\beta)$ is the number of roots in $\SuppT(\gamma)\cap\SuppT(\beta)$ minus the number of roots in $\SuppT(c\gamma)\cap\SuppT(\beta)$.  
There are several possible relationships between $\beta$ and $\phi$.

If $\SuppT(\beta)\supsetneq\SuppT(\phi)$, then $\beta$ does not have the same counterclockwise endpoint as $\phi$.
(Recall that we are in the case where no root $\beta\in C'$ with $\SuppT(\beta)\supsetneq\SuppT(\phi)$ has the same counterclockwise endpoint as $\phi$.)
Thus also ${\SuppT(\beta)\supsetneq\SuppT(\gamma)}$ and $\beta$ does not have the same counterclockwise endpoint as~$\gamma$.
Then $|\SuppT(\gamma)\cap\SuppT(\beta)|=|\SuppT(c\gamma)\cap\SuppT(\beta)|$, so ${E_c(\gamma\ck,\beta)=0}$.

If $\SuppT(\beta)\subseteq\SuppT(\phi)$, then since $\gamma$ and $\phi$ have the same counterclockwise endpoint, $\SuppT(\gamma)\cap\SuppT(\beta)$ is at least as large as $\SuppT(c\gamma)\cap\SuppT(\beta)$, so $E_c(\gamma\ck,\beta)\ge0$.

Otherwise, since $\phi$ is compatible with $\beta$, Proposition~\ref{compatible in tubes} says that $\beta$ and $\phi$ are spaced.
Since $\SuppT(\gamma)\subseteq\SuppT(\alpha)\subseteq\SuppT(\phi)$, also $\beta$ and $\gamma$ are spaced, so Proposition~\ref{compute in tubes} implies that $E_c(\gamma,\beta)=0$.

Next, consider the case where some root $\beta\in C'$ with $\SuppT(\beta)\supsetneq\SuppT(\phi)$ has the same counterclockwise endpoint as $\phi$.
Choose $\beta_0$ to have minimal support among such roots.
Again, there is a unique vertex $v$ of the cycle that is in the support of $\phi$ but is not in the support of any root $\beta$ with $\SuppT(\beta)\subseteq\SuppT(\phi)$.
Let $\phi'$ be the root whose counterclockwise vertex is the vertex adjacent to and clockwise of $v$, and whose clockwise vertex is the clockwise vertex of $\beta_0$.
Figure~\ref{case2fig} shows two representative cases of the construction of $\phi'$.
\begin{figure}
\includegraphics{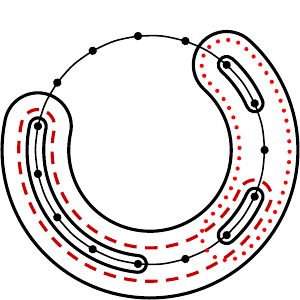}
\quad
\includegraphics{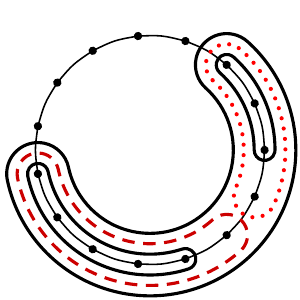}
\caption{Another illustration of the proof of Proposition~\ref{walls in d infty1}.}
\label{case2fig}
\end{figure}
Again, we have circled the support of $\phi$ with a dotted line. 
The support of $\beta_0$ is the largest circled support, and is shown with a solid line.
The support of $\phi'$ is circled with a dashed line.
Again, some of the roots $\beta\in C'$ are not shown.
We replace $\phi$ by $\phi'$, and note that $\phi'$ is $c$-compatible with every $\beta\in C$.
With this replacement, we fall into the previous case, so we construct $\alpha$ as described above.
\end{proof}

To deduce Proposition~\ref{walls in d infty2} from Proposition~\ref{walls in d infty1}, we need the following lemma.

\begin{lemma}\label{rescue}
Suppose that $\beta\in\APTre{c}$ and that $\gamma\in\APre{c}\setminus\APTre{c}$.
If $\omega_c(\gamma,\beta)>0$, then ${\d_\infty\subseteq\set{x\in V^*:\br{x,\gamma}\le0}}$.
\end{lemma}
\begin{proof}
Since $\gamma\in\APre{c}\setminus\APTre{c}$, Proposition~\ref{E omega delta in tubes} says $\omega_c(\gamma,\delta)\neq0$.
Recall that every root $\gamma\in\RSre$ is a scalar multiple of $\phi+k\delta$ for some $\phi\in\RSfin$ and $k\in\integers$.
Since $\d_\infty\subset\delta^\perp$, we have ${\d_\infty\subseteq\set{x\in V^*:\br{x,\gamma}\le0}}$ if and only if ${\d_\infty\subseteq\set{x\in V^*:\br{x,\phi}\le0}}$.
Furthermore, since $\omega_c$ is skew-symmetric, $\omega_c(\gamma,\delta)=\omega_c(\phi,\delta)$.
Thus by Theorem~\ref{DFc complement}, either $\d_\infty\subseteq\set{x\in V^*:\br{x,\gamma}\le0}$ (if $\omega_c(\gamma,\delta)>0$) or $\d_\infty\subseteq\set{x\in V^*:\br{x,\gamma}\ge0}$ (if $\omega_c(\gamma,\delta)<0$).
Therefore, we need to show that $\omega_c(\gamma,\delta)>0$.

Proposition~\ref{apSchur2} says that $\gamma$ is either of the form $c^ks_1\cdots s_{j-1}\alpha_j$ for some $k\ge0$ and $j\in\set{1,\ldots,n}$ or of the form $c^{-k}s_n\cdots s_{j+1}\alpha_j$ for some $k\ge0$ and $j\in\set{1,\ldots,n}$.
Thus we can repeatedly apply simple reflections that are initial or final in $c$ to $\gamma$ and conjugate $c$ by the same simple reflections until, writing $c'$ for the new Coxeter element, $\gamma=\alpha_s$ for $s$ either initial or final in~$c'$.
Lemma~\ref{s pos real in tubes} says that the same sequence of simple reflections applied to $\beta$ yields a root $\beta'\in\APTre{c'}$.
Lemma~\ref{OmegaInvariance} says that $\omega_{c'}(\alpha_s,\beta')=\omega_c(\gamma,\beta)>0$.
But if $s$ is final in $c'$, then $\omega_{c'}(\alpha_s,\beta')\le0$ for any positive root $\beta'$, so $s$ is initial in~$c'$.
Thus also $\omega_{c'}(\alpha_s,\delta)>0$.
By Lemma~\ref{OmegaInvariance} again, $\omega_c(\gamma,\delta)>0$.
\end{proof}

We now show how Proposition~\ref{walls in d infty2} follows from Proposition~\ref{walls in d infty1} and Lemma~\ref{rescue}.

\begin{proof}[Proof of Proposition~\ref{walls in d infty2}]
Since each $\beta'\in C$ is a positive root, each vector $\nu_c(\beta')$ is $-E_c(\,\cdot\,,\beta')$.
Thus, by the definition of $\d_\beta$ in \eqref{d beta}, we need to show that there is a positive root $\beta$ in $\APTre{c}$ satisfying condition \eqref{beta property 1} from Proposition~\ref{walls in d infty1} and a stronger property than condition \eqref{beta property 2}:  
The condition on each $E_c(\gamma,\beta')$ must hold for all $\gamma\in\cut(\beta)$ with $\omega_c(\gamma,\beta)>0$, not only for $\gamma\in\cut(\beta)\cap\APTre{c}$.
Lemma~\ref{rescue} provides the needed strengthening of the property:
Theorem~\ref{d inf extreme} and Proposition~\ref{in the cone} imply that $\nu_c(\beta')\in\d_\infty$ for all $\beta'\in\APT{c}$.
Thus, if $\gamma\in\APre{c}\setminus\APTre{c}$ and $\omega_c(\gamma,\beta)>0$, Lemma~\ref{rescue} says that $\br{\nu_c(\beta'),\gamma}\le0$, or in other words $E_c(\gamma,\beta')\ge0$.
\end{proof}

\subsection{Affine scattering fans and generalized associahedron fans}\label{2 fans sec}
We now prove the part of Theorem~\ref{fans thm} that relates $\ScatFanTB$ with $\nu_c(\Fan_c(\RS))$.  

\begin{theorem}\label{scat nu}
If $B$ is an acyclic exchange matrix of affine type, corresponding to a root system $\RS$ and a Coxeter element $c$, then $\ScatFanTB$ and $\nu_c(\Fan_c(\RS))$ coincide.
\end{theorem}

We now prepare to prove Theorem~\ref{scat nu}.

\begin{proposition}\label{faces in shards}
Every $(n-1)$-dimensional face of $\nu_c(\Fan_c(\RS))$ is contained in a wall of $\DCScat(A,c)$.
\end{proposition}
\begin{proof}
Suppose $F$ is an $(n-1)$-dimensional face of $\nu_c(\Fan_c(\RS))$.
If $F$ is contained in $\d_\infty$, then we are done, and if not, then $F$ is a face of $\DF_c$ by Theorems~\ref{nu thm}, \ref{Affine g}, and~\ref{DFc complement}, .
Writing $t$ for the reflection such that $F\subseteq\beta_t^\perp$, condition~\eqref{adjacent} of Proposition~\ref{DC Phic} holds.
Therefore by condition~\eqref{either or j} of Proposition~\ref{DC Phic}, there is a $c$-sortable join-irreducible element $j$ with ${\cov(j)=\set{t}}$ or there is a $c^{-1}$-sortable join-irreducible element $j'$ with ${\cov(j')=\set{t}}$, or both.
Now Proposition~\ref{explicit shards} says that $F$ is contained in $\Sh(j)$ or $F$ is contained in $-\Sh(j')$, or both.
\end{proof}

The following lemma is \cite[Lemma~3.3]{typefree}.

\begin{lemma}\label{Ec invariant}
If~$s$ is initial or final in $c$, then $E_c(\beta,\beta')=E_{scs}(s \beta, s \beta')$ for all $\beta$ and~$\beta'$ in $V$.
\end{lemma}

\begin{lemma}\label{nu scs}  
Suppose $\beta\in\RSpos$.
If $s$ is initial or final in $c$ and $\beta\neq\alpha_s$, then $\nu_{scs}(s\beta)=s\nu_c(\beta)$.
\end{lemma}
\begin{proof}
Since $\beta\neq\alpha_s$, the root $s\beta$ is positive, so by Lemma~\ref{Ec invariant}, for any $\gamma\in V$,
\[\br{\nu_{scs}(s\beta),\gamma}=-E_{scs}(\gamma,s\beta)=-E_c(s\gamma,\beta)=\br{\nu_c(\beta),s\gamma}=\br{s\nu_c(\beta),\gamma}.\qedhere\]
\end{proof}

\begin{proposition}\label{faces in shards imaginary}
Every $(n-2)$-dimensional face of $\nu_c(\Fan_c(\RS))$ that is contained in $\d_\infty$ is also contained in another wall of $\DCScat(A,c)$.
\end{proposition}

\begin{proof}%[Proof of Proposition~\ref{faces in shards imaginary}]
Suppose $F$ is an $(n-2)$-dimensional face of $\nu_c(\Fan_c(\RS))$ that is contained in $\d_\infty$.
If $F$ does not contain $\nu_c(\delta)$, then $F$ is a face of $\nu_c(\Fan_c^\re(\RS))$.
Theorem~\ref{clus def OK} implies that $F$ is not a maximal face of $\nu_c(\Fan_c^\re(\RS))$, so $F$ is contained in some ${(n-1)}$-dimensional face $F'$ of $\nu_c(\Fan_c^\re(\RS))$.
Proposition~\ref{faces in shards} says that $F'$ is contained in some wall of $\DCScat(A,c)$.
But Proposition~\ref{codim 1 delta} says that $F'$ is not contained in $\delta^\perp$, so that wall is not $\d_\infty$.
If $F$ contains $\nu_c(\delta)$, then Proposition~\ref{walls in d infty2} says that $F$ is contained in $\d_\beta$ for some $\beta\in\APTre{c}$.
\end{proof}

\begin{proof}[Proof of Theorem~\ref{scat nu}]
Throughout the proof, we use Theorem~\ref{DCScat}, which implies that $\ScatFanTB$ is the set of  $\DCScat(A,c)$-cones and their faces.
By construction and by Proposition~\ref{Sigma j j'}, the ramparts of $\DCScat(A,c)$ are precisely its walls.
Thus points $p$ and $q$ are in the same $\DCScat(A,c)$-class if and only if there is a path $\gamma$ from~$p$ to $q$ such that the function taking a point to the set of all walls containing it is constant on $\gamma$.
Recall that the $\DCScat(A,c)$-cones are the closures of $\DCScat(A,c)$-classes.

We first show that every cone of $\nu_c(\Fan_c(\RS))$ is contained in a $\DCScat(A,c)$-cone.
We will make several successive reductions/rephrasings:
It is enough to show that every maximal cone $F$ of $\nu_c(\Fan_c(\RS))$ is contained in a $\DCScat(A,c)$-cone.
Since $F$ is the closure of $\relint(F)$ and since the closure of a $\DCScat(A,c)$-class is a $\DCScat(A,c)$-cone, it is enough to show that $\relint(F)$ is in a $\DCScat(A,c)$-class.
Finally, it is enough to show that any wall that intersects $\relint(F)$ actually contains $\relint(F)$, because if so, any two points $p,q\in\relint(F)$ are connected by a path inside $\relint(F)$ that does not enter or leave any wall.

A maximal cone $F$ of $\nu_c(\Fan_c(\RS))$ is $\nu_c(\Cone(C))$ for some $c$-cluster~$C$.  
We break into two cases, based on whether $C$ is a real or imaginary $c$-cluster.

First suppose $C$ is real.
Theorems~\ref{nu thm} and~\ref{Affine g} combine to say that $F$ is a maximal cone of $\DF_c$.
In this case, we need to prove that no wall of $\DCScat(A,c)$ intersects the interior of $F$.
Suppose to the contrary that some wall of $\DCScat(A,c)$ intersects the interior of~$F$.
Since $\d_\infty$ is in the boundary of $|\DF_c|$, it does not intersect the interior of~$F$, so without loss of generality (up to replacing $c$ by $c^{-1}$ and applying the antipodal map), there is a $c$-sortable join-irreducible element $j$ such that $\Sh(j)$ intersects the interior of $F$.
However, $\Sh(j)$ is normal to some real root and thus not contained in $\delta^\perp$, so the intersection of $\Sh(j)$ with the interior of $F$ contains a point $p$ in $\Tits(A)\cup(-\Tits(A))$.
However, Proposition~\ref{explicit shards}.\ref{explicit j} then says that $p$ is contained in an $(n-1)$-dimensional face of $\nu_c(\Fan_c(\RS))$.
This contradiction shows that no wall of $\DCScat(A,c)$ intersects the interior of $F$.

Now suppose $C$ is imaginary.
Then $F$ is contained in $\d_\infty$, which is the closure of $V^*\setminus|\DF_c|$.
We will show that the relative interior of $F$ is disjoint from all walls of $\DCScat(A,c)$ other than $\d_\infty$.

Suppose for the sake of contradiction that some point $p$ in the relative interior of $F$ is contained in a wall $\d$ of $\DCScat(A,c)$ with $\d\neq\d_\infty$.
Since $\d$ is a codimension-$1$ cone not contained in the hyperplane containing $\d_\infty$, there exist a point $q\not\in\d_\infty$ such that the line segment $\seg{pq}$ is contained in $\d$.

Consider the bi-infinite sequence $\ldots,C_{-1},C_0,C_1,\ldots$ of distinct real $c$-clusters constructed in Proposition~\ref{cluster limit}.
The sequences 
\[\nu_c(\Cone(C_1)),\nu_c(\Cone(C_2)),\ldots\text{ and }\nu_c(\Cone(C_{-1})),\nu_c(\Cone(C_{-2})),\ldots\]
converge to $F$ from opposite sides of the hyperplane containing $F$.
Thus there exists some $i\in\integers$ such that $\seg{pq}$ intersects the interior of $\nu_c(\Cone(C_i))$.
Therefore $\d$ intersects the interior of $\nu_c(\Cone(C_i))$.
But $\nu_c(\Cone(C_i))$ is a full-dimensional cone of $\nu_c(\Fan_c(\RS))$, so we have already shown that no wall of $\DCScat(A,c)$ intersects the interior of $\nu_c(\Cone(C_i))$.
This contradiction proves that the relative interior of $F$ is disjoint from all walls of $\DCScat(A,c)$ other than $\d_\infty$.

We complete the proof by showing that every $\DCScat(A,c)$-cone is contained in some cone of $\nu_c(\Fan_c(\RS))$.
Since each $\DCScat(A,c)$-cone is the closure of a $\DCScat(A,c)$-class, it is enough to show that each $\DCScat(A,c)$-class is contained in some cone of $\nu_c(\Fan_c(\RS))$.

Let $\C$ be a $\DCScat(A,c)$-class.
By definition, any two points of $\C$ are connected by a path which never enters or leaves a wall of $\DCScat(A,c)$.
Since $\d_\infty$ (the closure of $V^*\setminus|\DF_c|$) is a wall of $\DCScat(A,c)$, the $\DCScat(A,c)$-class $\C$ is either contained in $|\DF_c|$ or contained in $\d_\infty$.

First, suppose $\C$ is contained in $|\DF_c|$.  
Proposition~\ref{faces in shards} says that every $(n-1)$-dimensional face of $\DF_c$ is contained in some wall of $\DCScat(A,c)$.
A point $p\in\C$ is in the relative interior of some face $F$ of $\DF_c$.
If some $q\in\C$ is not also in $\relint(F)$, then every path from $p$ to $q$ enters or leaves a wall of $\DCScat(A,c)$:
Since $\DF_c$ is a simplicial fan, leaving $\relint(F)$ to a larger face leaves some $(n-1)$-dimensional face, and thus leaves the wall containing that face.
Leaving $\relint(F)$ to a smaller face enters some $(n-1)$-dimensional face, and thus enters the wall containing that face.
We conclude that all of $\C$ is in the relative interior of $F$, which is, by Theorems~\ref{nu thm} and~\ref{Affine g}, a face of $\nu_c(\Fan_c(\RS))$.

Now suppose $\C$ is contained in $\d_\infty$.
The faces of $\nu_c(\Fan_c(\RS))$ that are contained in $\d_\infty$ constitute a fan whose maximal faces have dimension $n-1$.
This fan is simplicial as a fan in $\delta^\perp$.
Proposition~\ref{faces in shards imaginary} says that every $(n-2)$-dimensional face of the fan in $\d_\infty$ is contained in some wall of $\DCScat(A,c)$ in addition to $\d_\infty$.
Arguing exactly as in the previous case, we see that all of $\C$ is contained in the relative interior of a face of $\nu_c(\Fan_c(\RS))$ contained in $\d_\infty$.
\end{proof}

As a consequence of Theorem~\ref{scat nu}, we have a strengthening of Proposition~\ref{explicit shards}.
(Part of the following, specifically Corollary~\ref{explicit shards cor}.\ref{d inf union}, was already known by Theorems~\ref{nu thm} and~\ref{DFc complement}.)

\begin{corollary}\label{explicit shards cor} Let $\RS$ be a root system of affine type and let $c$ be a Coxeter element of the corresponding Weyl group.
\begin{enumerate}[label=\bf\arabic*., ref=\arabic*]  
\item \label{explicit j cor}
If $j$ is a $c$-sortable join-irreducible element with $\cov(j)=\set{t}$, then $\Sh(j)$ is the union of all faces of $\nu_c(\Fan_c(\RS))$ contained in $\beta_t^\perp$.
\item \label{explicit j' cor}
If $j$ is a $c^{-1}$-sortable join-irreducible element with $\cov(j)=\set{t}$, then $-\Sh(j)$ is the union of all faces of $\nu_c(\Fan_c(\RS))$ contained in $\beta_{t}^\perp$.
\item \label{d inf union}
$\d_\infty$ is the union of all faces of $\nu_c(\Fan_c(\RS))$ contained in $\delta^\perp$.
\end{enumerate}
\end{corollary}
\begin{proof}
By definition, each rampart $R$ of any scattering diagram $\D$ is a union of $\D$-classes, and if $R$ is closed, it is immediate that $R$ is the union of all $\D$-cones it contains.  
By the definition of a rampart, $R$ is also the union of all $\D$-cones contained in the hyperplane containing $R$.

Since the ramparts of $\DCScat(A,c)$ are precisely its walls, in particular, each rampart is closed.
All three assertions now follow.
\end{proof}

\subsection{Affine mutation fans}\label{aff mut fan sec}
It now remains only to show that $\ScatFan^T(B)$ and $\nu_c(\Fan_c(\RS))$ coincide with the mutation fan $\F_{B^T}$.
We begin by reviewing the definition of the mutation fan.

Let $B$ be an exchange matrix and let $\tB$ be an \newword{extended exchange matrix}:  a matrix with $n$ columns, whose first $n$ rows are $B$, with an additional $\ell$ rows (with $\ell\ge0$) consisting of real entries.
For each $k=1,\ldots,n$, we define the \newword{mutation} $\mu_k(\tB)$ of $\tB$ at $k$ to be the matrix $\tB'=[b'_{ij}]$ given by
\begin{equation}\label{b mut}
b_{ij}'=\left\lbrace\!\!\begin{array}{ll}
-b_{ij}&\mbox{if }i=k\mbox{ or }j=k;\\
b_{ij}+\sgn(b_{kj})\max(b_{ik}b_{kj},0)&\mbox{otherwise.}
\end{array}\right.
\end{equation}
Here, $\sgn(x)$ is the sign ($-1$, $0$, or $1)$ of $x$.

If $\kk=k_q,\ldots,k_1$ is a sequence of indices in $\set{1,\ldots,n}$, then $\mu_\kk$ means the composition $\mu_{k_q}\circ\mu_{k_{q-1}}\circ\cdots\circ\mu_{k_1}$.
Given an exchange matrix $B$ and a sequence $\kk$, the \newword{mutation map} is a map $\eta_\kk^B:V^*\to V^*$ defined as follows:
Starting with $x=\sum_{i=1}^na_i\rho_i\in V^*$, let $\tB$ be the extended exchange matrix obtained by adjoining a single row $(a_i:i=1,\ldots,n)$ below $B$.
Writing $(a'_i:i=1,\ldots,n)$ for the last row of $\mu_\kk(\tB)$, we define $\eta_\kk^B(x)$ to be $\sum_{i=1}^na'_i\rho_i\in V^*$.
Each mutation map is a piecewise linear homeomorphism from $V^*$ to itself.

Given $x=\sum_{i=1}^na_i\rho_i\in V^*$, write $\vsgn(x)$ for $(\sgn(a_1),\ldots,\sgn(a_n))\in\set{-1,0,1}^n$.
We define a notion of $B$-equivalence for vectors in $x_1,x_2\in V^*$.
Say $x_1\equiv^B x_2$ if and only if $\vsgn(\eta^B_\kk(x_1))=\vsgn(\eta^B_\kk(x_2))$ for every sequence $\kk$ of indices.
We call the $\equiv^B$-equivalence classes \newword{$B$-classes} and we call the closures of $B$-classes \newword{$B$-cones}.
The \newword{mutation fan}~$\F_B$ is the set of $B$-cones and their faces.
This is a complete fan by \cite[Theorem~5.13]{universal}.

In light of Theorems~\ref{scat ref mut} and~\ref{scat nu}, we can complete the proof of Theorem~\ref{fans thm} by proving the following proposition.

\begin{proposition}\label{mut ref nu}
If $B$ is an acyclic exchange matrix of affine type, corresponding to a root system $\RS$ and a Coxeter element $c$, then $\F_{B^T}$ refines $\nu_c(\Fan_c(\RS))$.
\end{proposition}

To prove Proposition~\ref{mut ref nu}, we quote and prove some additional background results, mostly about the mutation fan.
We begin with the following less-detailed version of \cite[Theorem~8.7]{universal}.

\begin{theorem}\label{mut g}
For any exchange matrix $B$, the $\g$-vector fan associated to~$B$ is a subfan of the mutation fan $\F_{B^T}$.
\end{theorem}

The following theorem is \cite[Theorem~3.8]{dominance}, which is obtained by combining results of other papers as explained in~\cite{dominance}.

\begin{theorem}\label{camb FB finite}
If $B$ is an acyclic exchange matrix of finite type and $c$ is the associated Coxeter element, then the mutation fan $\F_{B^T}$ coincides with the Cambrian fan~$\F_c$.
\end{theorem}

Define $B_\br{\ell}$ to be the matrix obtained from $B$ by deleting row $\ell$ and column~$\ell$.
We retain the original indexing of $B_\br{\ell}$, so its rows and columns are indexed by $\set{1,\ldots,n}\setminus\set{\ell}$.
We also realize the mutation maps for $B_\br{\ell}$, the $B_\br{\ell}$-classes and $B_\br{\ell}$-cones, and the mutation fan $\F_{B_\br{\ell}}$ in the subspace $V_\br{\ell}^*$ (the dual space to $V_\br{\ell}=\Span\set{\alpha_i:i\neq\ell}$).
As before, $\Proj_\br{\ell}:V^*\to V^*_\br{\ell}$ is the projection dual to the inclusion of $V_\br{\ell}$ into $V$.
Define $\Proj_\br{\ell}^{-1}(\F_{B_\br{\ell}})$ to be the fan with cones $\Proj_\br{\ell}^{-1}(C)=\set{x\in V^*:\Proj_{\br{\ell}}(x)\in C}$, where $C$ ranges over cones of $\F_{B_\br{\ell}}$.

\begin{proposition}\label{FB Proj}
For $\ell\in\set{1,\ldots,n}$, the fan $\F_B$ refines $\Proj_\br{\ell}^{-1}(\F_{B_\br{\ell}})$.
\end{proposition}
\begin{proof}
The proposition is equivalent to the following assertion:
For $\ell\in\set{1,\ldots,n}$, if $x_1\equiv^Bx_2$ then $\Proj_\br{\ell}(x_1)\equiv^{B_\br{\ell}}\Proj_\br{\ell}(x_2)$.

Suppose $x_1\equiv^Bx_2$ and let $\kk$ be a sequence of indices in $\set{1,\ldots,n}\setminus\set{\ell}$.
Because~$\ell$ doesn't occur in $\kk$, the matrix mutations used to compute $\eta^B_\kk$ commute with deleting row $\ell$ and column $\ell$.
Thus $\vsgn\circ\eta^{B_\br{\ell}}_\kk\circ\Proj_\br{\ell}=\vsgn\circ\Proj_\br{\ell}\circ\eta^B_\kk$.
We conclude that $x_1\equiv^Bx_2$ implies $\Proj_\br{\ell}(x_1)\equiv^{B_\br{\ell}}\Proj_\br{\ell}(x_2)$.
\end{proof}

The following is \cite[Proposition~5.3]{universal}.
\begin{proposition}\label{linear}
For any exchange matrix $B$, and every sequence $\kk$ of indices, the mutation map $\eta_\kk^B$ is linear on every $B$-cone.
\end{proposition}

The following lemma is easy.
\begin{lemma}\label{mu c}
If $B$ is an acyclic exchange matrix and $c=s_1\cdots s_n$ is the corresponding Coxeter element, then $\mu_{12\cdots n}(B)=B$.
\end{lemma}

The following lemma is immediate by comparing the definition of the mutation map with the action of simple reflections on fundamental weights \eqref{dual on rho}.

\begin{lemma}\label{helpful}
Suppose $B$ is acyclic and the Coxeter element associated to~$B$ is $c=s_1\cdots s_n$.
Let $x=\sum_{i=1}^na_i\rho_i\in V^*$.
If $a_1\ge0$, then $\eta_1^{B^T}(x)=-a_1\rho_1+\sum_{i=2}^na_i\rho_i$.
If $a_1\le0$, then $\eta_1^{B^T}(x)=s_1x$.
\end{lemma}

The map $\sigma_s:\AP{c}\to\AP{scs}$ was defined in Section~\ref{affdenom sec} for $s$ initial or final in $c$.

\begin{proposition}\label{nu sigma}
If $c=s_1\cdots s_n$, then the maps $\nu_{s_1cs_1}\circ\sigma_{s_1}$ and $\eta_1^{B^T}\circ\nu_c$ from $\AP{c}$ to $V^*$ coincide.
\end{proposition}
\begin{proof}
We use Lemma~\ref{helpful} throughout.
Suppose $\beta\in\AP{c}$.

If $\beta=-\alpha_\ell$ for some $\ell\neq 1$, then $\nu_{s_1cs_1}(\sigma_{s_1}(\beta))=\nu_{s_1cs_1}(-\alpha_\ell)=\rho_\ell$ and $\eta_1^{B^T}(\nu_c(\beta))=\eta_1^{B^T}(\rho_\ell)=\rho_\ell$.

If $\beta=-\alpha_1$, then $\nu_{s_1cs_1}(\sigma_{s_1}(\beta))=\nu_{s_1cs_1}(\alpha_1)=-\sum_{i=1}^nE_{s_1cs_1}(\alpha_i\ck,\alpha_1)\rho_i$.
Since $s_1$ is final in $s_1cs_1$, this is $-\rho_1$.
Also, $\eta_1^{B^T}(\nu_c(\beta))=\eta_1^{B^T}(\rho_1)=-\rho_1$.

If $\beta=\alpha_1$, then $\nu_{s_1cs_1}(\sigma_{s_1}(\beta))=\nu_{s_1cs_1}(-\alpha_1)=\rho_1$.
Also, \[\nu_c(\beta)=-\sum_{i=1}^nE_c(\alpha_i\ck,\alpha_1)=-\rho_1-\sum_{i\neq 1}a_{i1}\rho_i=s_1\rho_1,\]
so $\eta_1^{B^T}(\nu_c(\beta))=\rho_1$.

On positive roots $\beta\neq\alpha_1$, the map $\sigma_{s_1}$ agrees with $s_1$, so $\nu_{s_1cs_1}(\sigma_{s_1}(\beta))=s_1\nu_c(\beta)$ by Lemma~\ref{nu scs}.
Since $\nu_c(\beta)$ has a nonpositive $\rho_1$-coefficient in the basis of fundamental weights, $\eta_1^{B^T}(\nu_c(\beta))=s_1\nu_c(\beta)$.
\end{proof}

The next result is \cite[Proposition~7.3]{universal}.
\begin{proposition}\label{eta FB}
For any exchange matrix~$B$ and any sequence $\kk$ of indices, the mutation map $\eta^{B^T}_\kk$ is an isomorphism from $\F_{B^T}$ to $\F_{\mu_\kk(B^T)}$.
\end{proposition}

The following result is \cite[Corollary~4.4]{scatfan}.

\begin{theorem}\label{mut scat fan}
For any exchange matrix~$B$ and any sequence $\kk$ of indices, the mutation map $\eta^{B^T}_\kk$ is an isomorphism from $\ScatFan^T(B)$ to $\ScatFan^T(\mu_\kk(B))$.
\end{theorem}

\begin{proposition}\label{eta nice} 
Suppose $B$ is an acyclic exchange matrix and the Coxeter element associated to $B$ is $c=s_1\cdots s_n$.
Then the piecewise linear map $\eta^{B^T}_{12\cdots n}$
\begin{enumerate}[label=\bf\arabic*., ref=\arabic*]  
\item \label{nu tau}   
has $\eta^{B^T}_{12\cdots n}\circ\nu_c=\nu_c\circ\tau_c$ as maps on $\AP{c}$, 
\item\label{eta aut mut}
is an automorphism of $\F_{B^T}$,
\item\label{eta c fan}
is an automorphism of $\ScatFanTB$, and
\item\label{eta is c}
fixes $\d_\infty$ as a set, agrees with $c$ on $\d_\infty$, and has finite order on $\d_\infty$. 
\end{enumerate}
\end{proposition}
\begin{proof}
We first prove Assertion~\ref{nu tau} by checking that $\eta^{B^T}_{12\cdots n}=\nu_c\circ\tau_c\circ\nu_c^{-1}$ as maps on the set $\nu_c(\AP{c})$.
Proposition~\ref{nu sigma} says that $\eta^{B^T}_1=\nu_{s_1cs_1}\circ\sigma_{s_1}\circ\nu_c^{-1}$.
Since $s_2$ is initial in $s_1cs_1$, which is the Coxeter element associated to $\mu_1(B)$, and since transpose commutes with mutation, we can write a similar expression for $\eta^{\mu_1(B^T)}_2$.  
Composing, we see that $\eta^{B^T}_{21}=\nu_{s_2s_1cs_1s_2}\circ\sigma_{s_2}\sigma_{s_1}\circ\nu_c^{-1}$.
Continuing in this manner, we see that $\eta^{B^T}_{n(n-1)\cdots1}=\nu_c\tau_c^{-1}\nu_c^{-1}$.
Inverting and appealing to Lemma~\ref{mu c}, we see that $\eta^{B^T}_{12\cdots n}=\nu_c\circ\tau_c\circ\nu_c^{-1}$ as desired.
Taking $\kk=12\cdots n$ in Proposition~\ref{eta FB} and applying Lemma~\ref{mu c}, we obtain Assertion~\ref{eta aut mut}.
Assertion \ref{eta c fan} follows from Theorem~\ref{mut scat fan} in the same way.

Now, $\br{x_c,\alpha_1\ck}=-\omega_c(\alpha_1\ck,\delta)<0$ by the definition of $\omega_c$, because $\delta$ is a positive combination of all of the simple roots, and because there exists some $j>1$ with $a_{1j}\neq0$.
Thus every cone of the Coxeter fan in $\delta^\perp$ that contains $x_c$ is in the halfspace $\set{x\in V^*:\br{x,\alpha_1\ck}\le0}$.
Thus by Theorem~\ref{DFc complement}, every point $\sum_{i=1}^na_i\rho_i\in\d_\infty$ satisfies $a_1\le0$, so Lemma~\ref{helpful} says that $\eta_1^{B^T}$ acts as $s_1$ on~$\d_\infty$.
By Theorem~\ref{mut scat fan}, $\eta_1^{B^T}$ takes $\d_\infty$ to the analogous imaginary wall relative to $\mu_1(B^T)$.
Thus we can apply Lemma~\ref{helpful} again to see that $\eta_2^{\mu_1(B^T)}$ acts as $s_2$ on that imaginary wall, and continue until we show that $\eta_{n(n-1)\cdots1}^B$ acts as $c^{-1}$ on $\d_\infty$ and fixes $\d_\infty$ as a set. 
By taking inverses and appealing to Lemma~\ref{mu c}, we see that $\eta^{B^T}_{12\cdots n}$ fixes $\d_\infty$ as a set and agrees with $c$ on $\d_\infty$.
It has finite order on $\d_\infty$ because it acts as $\nu_c\circ\tau_c\circ\nu_c^{-1}$ on the rays of $\nu_c(\Fan_c(\RS))$ and because $\d_\infty$ is the image of the span of $\APT{c}$ under~$\nu_c$.
This is Assertion~\ref{eta is c}.
\end{proof}

We now prove Proposition~\ref{mut ref nu}.

\begin{proof}[Proof of Proposition~\ref{mut ref nu}]
We need to show that if two points are in the same cone of $\F_{B^T}$ then they are in the same cone of $\nu_c(\Fan_c(\RS))$.
Since $\F_{B^T}$ and $\nu_c(\Fan_c(\RS))$ both contain the $\g$-vector fan as a subfan (Theorems~\ref{nu thm} and~\ref{mut g}), and since the $\g$-vector fan covers the entire ambient space except for the relative interior of~$\d_\infty$ (Theorem~\ref{DFc complement}), it remains only to prove the case where the two points are in~$\d_\infty$.
We will prove the contrapositive:
Suppose two points $p,q\in\d_\infty$ are not in the same cone of $\nu_c(\Fan_c(\RS))$.
We will show that they are not in the same cone of $\F_{B^T}$.

All maximal cones of $\nu_c(\Fan_c(\RS))$ contained in $\d_\infty$ are simplicial and of dimension ${n-1}$.
Thus since $p$ and $q$ are not in the same cone of $\nu_c(\Fan_c(\RS))$, there exists a face $F$ of $\nu_c(\Fan_c(\RS))$ of dimension $n-2$, contained in $\d_\infty$, such that $p$ and $q$ are not in the linear span of $F$ but the line segment $\seg{pq}$ intersects $F$.
The face $F$ is the nonnegative span of $\nu_c(C)$ for a set of $n-2$ pairwise $c$-compatible roots in $\AP{c}$.
If $\delta\not\in C$, then $F$ is on the relative boundary of $\d_\infty$, so that the intersection of $\d_\infty$ with the linear span of $F$ is a face of $\d_\infty$.
But this is impossible since $p$ and $q$ are on opposite sides of the linear span of $F$.
Thus $\delta\in C$, so Proposition~\ref{walls in d infty2} says that there exists $\beta\in\APTre{c}$ such that 
$F\subseteq\d_\beta$.
In particular, the line segment $\seg{pq}$ intersects~$\d_\beta$, and $p$ and~$q$ are on opposite sides of $\beta^\perp$.

To complete the proof, we will make use of the symmetry provided by the action of $c$.
Everything constructed so far in this proof is invariant under this symmetry.
Specifically, we check the following four statements of invariance.

First, by Proposition~\ref{eta nice}.\ref{eta is c}, the mutation map $\eta_{12\cdots n}^B$ is linear on $\d_\infty$ and fixes $\d_\infty$, and $\eta_{12\cdots n}^BF=cF$.
On the other hand, $F$ is the nonnegative span of $\nu_c(C)$, so Proposition~\ref{eta nice}.\ref{nu tau} says that $\eta_{12\cdots n}^BF$ is the nonnegative span of $\nu_c(\tau_c(C))$.
Now, \eqref{compat tau} implies that $\tau_c$ induces an automorphism of $\Fan_c(\RS)$, and we see that $\eta_{12\cdots n}^BF$ is an $(n-2)$-dimensional face of $\nu_c(\Fan_c(\RS))$.
The argument reverses, so $F$ is an $(n-2)$-dimensional face of $\nu_c(\Fan_c(\RS))$ if and only if $cF$ is an $(n-2)$-dimensional face of $\nu_c(\Fan_c(\RS))$.

Second, the map $c$ takes $\beta^\perp$ to $(c\beta)^\perp$ (by definition, because the action of $c$ on~$V^*$ is dual to the action of $c$ on $V$).
Thus $\seg{pq}$ intersects $F$ and $F\subseteq\beta^\perp$ and $p$ and $q$ are on opposite sides of $\beta^\perp$ if and only if $\seg{(cp)(cq)}$ intersects $cF$ and $cF\subseteq(c\beta)^\perp$ and $cp$ and $cq$ are on opposite sides of $(c\beta)^\perp$.

Third, we claim further that $F\subseteq\d_\beta$ if and only if $cF\subseteq\d_{c\beta}$.
Lemma~\ref{rescue} says that we can ignore roots in $\APre{c}\setminus\APTre{c}$ to verify the claim.
Thus the claim follows by Lemmas~\ref{OmegaInvariance} and~\ref{Ec invariant} and by the fact (apparent, for example, from the proof of Proposition~\ref{walls in d infty1}) that the cutting relation on $\APTre{c}$ is invariant under that action of~$c$.

Fourth and finally, by Proposition~\ref{eta nice}.\ref{eta aut mut} and Proposition~\ref{eta nice}.\ref{eta is c}, $p$ and $q$ are in the same cone of $\F_{B^T}$ if and only if $c(p)$ and $c(q)$ are in the same cone of $\F_{B^T}$.

By the definition of $\APTre{c}$ (see Section~\ref{affdenom sec}), there exists $k\in\integers$ such that $c^k\beta\in\RSfin$.
By the statements above about invariance under the action of $c$, we may as well assume that $k=0$, so that $\beta\in\RSfin$.

Theorem~\ref{easy scat} says that there is a $c$-sortable join-irreducible element $j$ such that $\Sh(j)=\d_\beta$ or a $c^{-1}$-sortable join-irreducible element $j'$ such that ${-\Sh(j')=\d_\beta}$.
Indeed, since $\beta\in\RSfin$, Propositions~\ref{ji ref} and~\ref{Sigma j j'} imply that both $j$ and $j'$ exist, and by Lemma~\ref{cover para}, we see that $j\in W_\fin$.
Thus $\Sh(j)=\Proj_{S_\fin}^{-1}\Sh_{S_\fin}(j)$ by Lemma~\ref{Sigma para}.
Therefore the segment connecting $\Proj_{S_\fin}(p)$ to $\Proj_{S_\fin}(q)$ intersects $\Sh_{S_\fin}(j)$, but neither $\Proj_{S_\fin}(p)$ nor $\Proj_{S_\fin}(q)$ is in the linear span of $\Sh_{S_\fin}(j)$.

Let $B_\fin$ be the exchange matrix obtained by deleting from $B$ the row and column indexed by the affine index $\aff$.
This is of finite type, so combining Theorem~\ref{camb FB finite} with Proposition~\ref{explicit shards fin}, $\Sh_{S_\fin}(j)$ is the union of all faces of the mutation fan $\F_{(B_\fin)^T}$ contained in the hyperplane in $V_\fin^*$ orthogonal to $\beta$.
We see that $\Proj_{S_\fin}(p)$ and $\Proj_{S_\fin}(q)$ are not in the same cone of $\F_{(B_\fin)^T}$.
Proposition~\ref{FB Proj} now implies that $p$ and $q$ are not in the same cone of the mutation fan $\F_{B^T}$.
\end{proof}

This completes the proof of Theorem~\ref{fans thm}.
We conclude with the following proof.

\begin{proof}[Proof of Theorem~\ref{gen fans thm}]
An exchange matrix is of affine type if and only if it is mutation-equivalent to an acyclic exchange matrix of affine type.
Theorem~\ref{fans thm}  is the special case of Theorem~\ref{gen fans thm} where $B$ is acyclic.
Thus Proposition~\ref{eta FB} and Theorem~\ref{mut scat fan} combine to prove Theorem~\ref{gen fans thm} for all exchange matrices of affine type.
\end{proof}

\subsection*{Acknowledgments}
We thank an anonymous referee for many helpful comments.

% bibliography
\bibliographystyle{plain}
\bibliography{bibliography}

\begin{thebibliography}{10}

\bibitem{BjBr}
Anders Bj\"{o}rner and Francesco Brenti.
\newblock {\em Combinatorics of {C}oxeter groups}, volume 231 of {\em Graduate
  Texts in Mathematics}.
\newblock Springer, New York, 2005.

\bibitem{Bourbaki}
N.~Bourbaki.
\newblock {\em \'{E}l\'ements de math\'ematique. {F}asc. {XXXIV}. {G}roupes et
  alg\`ebres de {L}ie. {C}hapitre {IV}: {G}roupes de {C}oxeter et syst\`emes de
  {T}its. {C}hapitre {V}: {G}roupes engendr\'es par des r\'eflexions.
  {C}hapitre {VI}: syst\`emes de racines}.
\newblock Actualit\'es Scientifiques et Industrielles, No. 1337. Hermann,
  Paris, 1968.

\bibitem{Bridgeland}
Tom Bridgeland.
\newblock Scattering diagrams, {H}all algebras and stability conditions.
\newblock {\em Algebr. Geom.}, 4(5):523--561, 2017.

\bibitem{associahedra}
Fr{\'e}d{\'e}ric Chapoton, Sergey Fomin, and Andrei Zelevinsky.
\newblock Polytopal realizations of generalized associahedra.
\newblock {\em Canad. Math. Bull.}, 45(4):537--566, 2002.
\newblock Dedicated to Robert V. Moody.

\bibitem{greedytheta}
Man~Wai Cheung, Mark Gross, Greg Muller, Gregg Musiker, Dylan Rupel, Salvatore
  Stella, and Harold Williams.
\newblock The greedy basis equals the theta basis: a rank two haiku.
\newblock {\em J. Combin. Theory Ser. A}, 145:150--171, 2017.

\bibitem{Deodhar}
Vinay~V. Deodhar.
\newblock A note on subgroups generated by reflections in {C}oxeter groups.
\newblock {\em Arch. Math. (Basel)}, 53(6):543--546, 1989.

\bibitem{Dyer}
Matthew Dyer.
\newblock Reflection subgroups of {C}oxeter systems.
\newblock {\em J. Algebra}, 135(1):57--73, 1990.

\bibitem{FeLamp}
Anna Felikson and Philipp Lampe.
\newblock Exchange graphs for mutation-finite non-integer quivers.
\newblock {\em J. Geom. Phys.}, 188:Paper No. 104811, 29, 2023.

\bibitem{FeShThTu12}
Anna Felikson, Michael Shapiro, Hugh Thomas, and Pavel Tumarkin.
\newblock Growth rate of cluster algebras.
\newblock {\em Proc. Lond. Math. Soc. (3)}, 109(3):653--675, 2014.

\bibitem{FeShTu12a}
Anna Felikson, Michael Shapiro, and Pavel Tumarkin.
\newblock Cluster algebras of finite mutation type via unfoldings.
\newblock {\em Int. Math. Res. Not. IMRN}, (8):1768--1804, 2012.

\bibitem{FoZe03a}
Sergey Fomin and Andrei Zelevinsky.
\newblock Cluster algebras. {II}. {F}inite type classification.
\newblock {\em Invent. Math.}, 154(1):63--121, 2003.

\bibitem{FoZe03}
Sergey Fomin and Andrei Zelevinsky.
\newblock {$Y$}-systems and generalized associahedra.
\newblock {\em Ann. of Math. (2)}, 158(3):977--1018, 2003.

\bibitem{ca4}
Sergey Fomin and Andrei Zelevinsky.
\newblock Cluster algebras. {IV}. {C}oefficients.
\newblock {\em Compos. Math.}, 143(1):112--164, 2007.

\bibitem{GHKK}
Mark Gross, Paul Hacking, Sean Keel, and Maxim Kontsevich.
\newblock Canonical bases for cluster algebras.
\newblock {\em J. Amer. Math. Soc.}, 31(2):497--608, 2018.

\bibitem{HIKT}
Eric~J. Hanson, Kiyoshi Igusa, Moses Kim, and Gordana Todorov.
\newblock Infinitesimal semi-invariant pictures and co-amalgamation.
\newblock {\em J. Lond. Math. Soc. (2)}, 109(1):Paper No. e12786, 57, 2024.

\bibitem{Humphreys}
James~E. Humphreys.
\newblock {\em Reflection groups and {C}oxeter groups}, volume~29 of {\em
  Cambridge Studies in Advanced Mathematics}.
\newblock Cambridge University Press, Cambridge, 1990.

\bibitem{Kac90}
Victor~G. Kac.
\newblock {\em Infinite-dimensional {L}ie algebras}.
\newblock Cambridge University Press, Cambridge, third edition, 1990.

\bibitem{MRZ}
Robert Marsh, Markus Reineke, and Andrei Zelevinsky.
\newblock Generalized associahedra via quiver representations.
\newblock {\em Trans. Amer. Math. Soc.}, 355(10):4171--4186, 2003.

\bibitem{Najera12}
Alfredo N\'{a}jera~Ch\'{a}vez.
\newblock On the c-vectors of an acyclic cluster algebra.
\newblock {\em Int. Math. Res. Not. IMRN}, (6):1590--1600, 2015.

\bibitem{hyperplane}
Nathan Reading.
\newblock Lattice and order properties of the poset of regions in a hyperplane
  arrangement.
\newblock {\em Algebra Universalis}, 50(2):179--205, 2003.

\bibitem{hplane_dim}
Nathan Reading.
\newblock The order dimension of the poset of regions in a hyperplane
  arrangement.
\newblock {\em J. Combin. Theory Ser. A}, 104(2):265--285, 2003.

\bibitem{congruence}
Nathan Reading.
\newblock Lattice congruences of the weak order.
\newblock {\em Order}, 21(4):315--344 (2005), 2004.

\bibitem{cambrian}
Nathan Reading.
\newblock Cambrian lattices.
\newblock {\em Adv. Math.}, 205(2):313--353, 2006.

\bibitem{sortable}
Nathan Reading.
\newblock Clusters, {C}oxeter-sortable elements and noncrossing partitions.
\newblock {\em Trans. Amer. Math. Soc.}, 359(12):5931--5958, 2007.

\bibitem{sort_camb}
Nathan Reading.
\newblock Sortable elements and {C}ambrian lattices.
\newblock {\em Algebra Universalis}, 56(3-4):411--437, 2007.

\bibitem{shardint}
Nathan Reading.
\newblock Noncrossing partitions and the shard intersection order.
\newblock {\em J. Algebraic Combin.}, 33(4):483--530, 2011.

\bibitem{universal}
Nathan Reading.
\newblock Universal geometric cluster algebras.
\newblock {\em Math. Z.}, 277(1-2):499--547, 2014.

\bibitem{scatcomb}
Nathan Reading.
\newblock A combinatorial approach to scattering diagrams.
\newblock {\em Algebr. Comb.}, 3(3):603--636, 2020.

\bibitem{scatfan}
Nathan Reading.
\newblock Scattering fans.
\newblock {\em Int. Math. Res. Not. IMRN}, (23):9640--9673, 2020.

\bibitem{dominance}
Nathan Reading.
\newblock Dominance phenomena: mutation, scattering and cluster algebras.
\newblock {\em Trans. Amer. Math. Soc.}, 376(2):773--835, 2023.

\bibitem{camb_fan}
Nathan Reading and David~E. Speyer.
\newblock Cambrian fans.
\newblock {\em J. Eur. Math. Soc. (JEMS)}, 11(2):407--447, 2009.

\bibitem{typefree}
Nathan Reading and David~E. Speyer.
\newblock Sortable elements in infinite {C}oxeter groups.
\newblock {\em Trans. Amer. Math. Soc.}, 363(2):699--761, 2011.

\bibitem{framework}
Nathan Reading and David~E. Speyer.
\newblock Combinatorial frameworks for cluster algebras.
\newblock {\em Int. Math. Res. Not. IMRN}, (1):109--173, 2016.

\bibitem{afframe}
Nathan Reading and David~E. Speyer.
\newblock Cambrian frameworks for cluster algebras of affine type.
\newblock {\em Trans. Amer. Math. Soc.}, 370(2):1429--1468, 2018.

\bibitem{afftheta}
Nathan Reading and Salvatore Stella.
\newblock Theta functions in acyclic affine type.
\newblock Preprint, 2026 (arXiv:2603.23429).

\bibitem{afforb}
Nathan Reading and Salvatore Stella.
\newblock The action of a {C}oxeter element on an affine root system.
\newblock {\em Proc. Amer. Math. Soc.}, 148(7):2783--2798, 2020.

\bibitem{affdenom}
Nathan Reading and Salvatore Stella.
\newblock An affine almost positive roots model.
\newblock {\em J. Comb. Algebra}, 4(1):1--59, 2020.

\bibitem{Reineke}
Markus Reineke.
\newblock Poisson automorphisms and quiver moduli.
\newblock {\em J. Inst. Math. Jussieu}, 9(3):653--667, 2010.

\bibitem{Seven}
Ahmet~I. Seven.
\newblock Cluster algebras and semipositive symmetrizable matrices.
\newblock {\em Trans. Amer. Math. Soc.}, 363(5):2733--2762, 2011.

\end{thebibliography}
\vspace{-0.175 em}

\end{document}